\documentclass[a4paper]{amsart}
\pdfoutput=1 
\usepackage[utf8]{inputenc}
\usepackage[T1]{fontenc}
\usepackage{lmodern}
\usepackage{amssymb}
\usepackage{mathtools}
\usepackage[lite]{amsrefs}
\renewcommand{\PrintDOI}[1]{\href{http://dx.doi.org/\detokenize{#1}}{doi: \detokenize{#1}}}

\BibSpec{book}{%
	+{}  {\PrintPrimary}                {transition}
	+{,} { \textit}                     {title}
	+{.} { }                            {part}
	+{:} { \textit}                     {subtitle}
	+{,} { \PrintEdition}               {edition}
	+{}  { \PrintEditorsB}              {editor}
	+{,} { \PrintTranslatorsC}          {translator}
	+{,} { \PrintContributions}         {contribution}
	+{,} { }                            {series}
	+{,} { \voltext}                    {volume}
	+{,} { }                            {publisher}
	+{,} { }                            {organization}
	+{,} { }                            {address}
	+{,} { \PrintDateB}                 {date}
	+{,} { }                            {status}
	+{}  { \parenthesize}               {language}
	+{}  { \PrintTranslation}           {translation}
	+{;} { \PrintReprint}               {reprint}
	+{.} { }                            {note}
	+{.} {}                             {transition}
	+{} { \PrintDOI}                   {doi}
	+{} { available at \url}            {eprint}
	+{}  {\SentenceSpace \PrintReviews} {review}
}

\usepackage{enumitem} 
\setlist[enumerate,1]{label=\textup{(\arabic*)}}
\usepackage{tikz-cd}

\usepackage{microtype}
\usepackage[pdftitle={Type semigroups for twisted groupoids and a dichotomy for groupoid C*-algebras},
pdfauthor={Bartosz Kwa\'sniewski, Ralf Meyer, Akshara Prasad},
pdfsubject={Mathematics}]{hyperref}

\numberwithin{equation}{section}
\theoremstyle{plain}
\newtheorem{theorem}[equation]{Theorem}
\newtheorem{thmx}{Theorem}

\newtheorem{corx}[thmx]{Corollary}
\newtheorem{lemma}[equation]{Lemma}
\newtheorem{proposition}[equation]{Proposition}
\newtheorem{corollary}[equation]{Corollary}
\theoremstyle{definition}
\newtheorem{definition}[equation]{Definition}
\theoremstyle{remark}
\newtheorem{remark}[equation]{Remark}
\newtheorem{example}[equation]{Example}

\DeclareMathOperator{\Bis}{Bis}
\DeclareMathOperator{\supp}{supp}
\DeclareMathOperator{\Cu}{Cu}

\newcommand{\OO}{\mathcal O}
\newcommand{\LL}{\mathcal L}
\newcommand{\RR}{\mathcal{R}}
\newcommand{\NN}{\mathcal{N}}

\newcommand{\ZZ}{\mathcal{Z}}
\newcommand{\K}{\mathcal K}
\newcommand{\cl}[1]{\overline{#1}}

\newcommand{\Free}{\mathbb F}
\newcommand{\F}{F}
\newcommand{\E}{\mathbb E}
\newcommand{\EL}{\mathbb{EL}} 

\newcommand*{\alb}{\hspace{0pt}} 
\newcommand*{\nb}{\nobreakdash}
\newcommand*{\Star}{\(^*\)\nobreakdash-}

\newcommand*{\C}{\mathbb C}
\newcommand*{\Z}{\mathbb Z}
\newcommand*{\Q}{\mathbb Q}
\newcommand*{\R}{\mathbb R}
\newcommand*{\N}{\mathbb N}
\newcommand*{\Sphere}{\mathbb S}
\newcommand*{\T}{\mathbb T}
\newcommand*{\Gr}{\mathcal G}
\newcommand{\GrH}{\mathcal H}
\newcommand*{\s}{s}
\newcommand*{\rg}{r}
\newcommand*{\Locmult}{\mathcal{M}_\mathrm{loc}}
\newcommand*{\hull}{\mathrm{I}} 
\newcommand*{\Null}{\mathcal N}
\newcommand*{\Bound}{\mathbb B}
\newcommand*{\Comp}{\mathbb K}
\newcommand*{\Fin}{\mathbb F}
\newcommand*{\red}{\mathrm r}
\newcommand*{\ess}{\mathrm{ess}}
\newcommand*{\diff}{\mathrm d}
\newcommand*{\Sect}{\mathfrak S}

\newcommand*{\Cst}{\mathrm C^*}
\newcommand*{\Mult}{\mathcal M}
\newcommand*{\Cont}{\mathrm C}
\newcommand*{\Contc}{\Cont_\mathrm c} 
\newcommand*{\Borel}{\mathfrak{B}}
\newcommand*{\Meager}{\mathfrak{M}}
\newcommand*{\B}{\mathcal B}
\newcommand*{\defeq}{\mathrel{\vcentcolon=}}
\newcommand*{\congto}{\xrightarrow\sim}

\DeclarePairedDelimiter{\abs}{\lvert}{\rvert}
\DeclarePairedDelimiter{\norm}{\lVert}{\rVert}
\DeclarePairedDelimiterX{\setgiven}[2]{\{}{\}}{#1\,{:}\,\mathopen{}#2}

\newcommand*{\into}{\rightarrowtail}
\newcommand*{\onto}{\twoheadrightarrow}

\begin{document}

\title[Type semigroups and twisted groupoid $\Cst$-algebras]{Type
  semigroups for twisted groupoids\\
  and a dichotomy for groupoid $\Cst$-algebras}

\author{Bartosz Kosma Kwa\'sniewski}
\email{bartoszk@math.uwb.edu.pl}
\address{Faculty of Mathematics\\
  University  of Bia\l ystok\\
  ul.\@ K.~Cio\l kowskiego 1M\\
  15-245 Bia\l ystok\\
  Poland}

\author{Ralf Meyer}
\email{rmeyer2@uni-goettingen.de}
\address{Mathematisches Institut\\
  Georg-August-Universit\"at G\"ottingen\\
  Bunsenstra\ss e 3--5\\
  37073 G\"ottingen\\
  Germany}

\author{Akshara Prasad}
\email{a.prasad@uni-goettingen.de}
\address{Mathematisches Institut\\
  Georg-August-Universit\"at G\"ottingen\\
  Bunsenstra\ss e 3--5\\
  37073 G\"ottingen\\
  Germany}

\thanks{The research of Bartosz Kwa\'sniewski was supported by the
  National Science Centre, Poland, through the WEAVE-UNISONO grant
  no.~2023/05/Y/ST1/00046. We thank Pere Ara
  for bringing~\cite{Ara-Moreno-Pardo:Nonstable_K_for_Graph_Algs} to our attention and a number of valuable comments which improved Corollary~\ref{cor:type_semigroups_for_Exel_Pardo} and Proposition~\ref{prop:reducing_to_graphs}.
  We are very grateful to the anonymous referees for their insightful,
  expert comments, which have improved and refined the presentation in many places.
  Among other things, they resulted in Remarks \ref{rem:Tarski_proof}, \ref{rem:Lsc_Cuntz_semigroup} and the short proof of Lemma~\ref{lem:regular_measure_extension}.
  We thank Diego Mart\'inez for his comments on
  Theorem
  \ref{thm:stably_finite_twisted}.\ref{item:stably_finite_Cartan8} and
  the notion of a precompact bisection in the non-Hausdorff case.
  We also thank Enrique Pardo, Aaron Kettner, Eusebio Gardella and the editors for their remarks and suggestions.
}

\begin{abstract}
  We develop a theory of type semigroups for arbitrary twisted, not
  necessarily Hausdorff \'etale groupoids.  The type semigroup is a
  dynamical version of the Cuntz semigroup.  We relate it to traces,
  ideals, pure infiniteness, and stable finiteness of the reduced
  and essential \(\Cst\)\nb-algebras.  If the reduced
  \(\Cst\)\nb-algebra of a twisted groupoid is simple and the type
  semigroup satisfies a weak version of
  almost unperforation, then the \(\Cst\)\nb-algebra is either
  stably finite or purely infinite.  We apply our theory to Cartan
  inclusions.  We calculate the type semigroup for the possibly
  non-Hausdorff groupoids associated to self-similar group actions
  on graphs and deduce a dichotomy for the resulting Exel--Pardo
  algebras.
\end{abstract}

\maketitle

\section{Introduction}
\label{sec:introduction}

Over the course of a hundred years, the Banach--Tarski
paradox~\cite{Banach1924} has had a huge impact on various fields of
mathematics and inspired countless researchers.  It
led von Neumann to introduce amenable groups, which divided the
world of groups into those that do not allow paradoxical actions,
and those which do. The
border between the two cases is whether a finitely additive measure
exists.  Similarly, von Neumann factors either admit a semifinite
trace (are of Type I or~II) or they are purely infinite (Type~III),
and then every projection~\(p\) is equivalent to its doubling~\(p\oplus p\).
In fact, examples of Type~III factors were constructed by von
Neumann using paradoxical actions of nonamenable groups.  An
analogous dichotomy occurs in the classification of
\(\Cst\)\nb-algebras, namely, a classifiable simple
\(\Cst\)\nb-algebra is either stably finite or purely infinite.  The
classification programme has been a powerful driving force for
operator algebraists in recent decades, and it is surprising how the
spirit of Tarski's Theorem remains present in important parts of this
theory.

Tarski~\cite{Tarski1938} discovered that preordered abelian monoids
are a good framework to study paradoxicality.
Tarski considered monoids with the algebraic preorder,
and he proved a fundamental monoid version of the Hahn--Banach
Theorem: for any element \(x\neq 0\) in a monoid \((S,+)\) that is
not paradoxical, there is a state \(\nu\colon S\to [0,\infty]\) with
\(\nu(x)=1\).  Here \(x\in S\setminus\{0\}\) is \emph{paradoxical}
if \((n+1)x\le n x\) for some \(n\in \N\).
This is applied in~\cite{Tarski1938} to an action of a group~\(G\)
on a set~\(X\) by expanding the action to allow the addition of
equidecomposability types and form a semigroup \(S(G,X)\) -- the so
called \emph{type semigroup}
(see~\cite{Wagon-Tomkowicz:Banach-Tarski_Paradox}).
Since in this construction decompositions into arbitrary
subsets are allowed, \(S(G,X)\) has some useful properties.
For instance,  an element \(x\neq 0\) is
paradoxical in \(S(G,X)\) if and only if it is \emph{properly
  infinite}, that is, \(2x=x\) (see
\cite{Wagon-Tomkowicz:Banach-Tarski_Paradox}*{Corollary~10.21}).
We say that a monoid with this property
\emph{has plain paradoxes}.
This condition has been considered before as property (QQ)
in~\cite{Ortega-Perera-Rordam:CFP_refinement} and is strictly weaker than
\emph{almost unperforation}, which is used in a number of sources (see
\cites{Elliott-Robert-Santiago:The_Cone,
  Gardella-Perera:Cuntz_semigroups, Ortega-Perera-Rordam:CFP_Cuntz,
  Rordam:stable_and_real_rank}).
The referee also pointed out some work by Cotlar and
Aumann on extension of functionals~\cite{Aumann:Erweiterungen_additiv}
that deserves more attention in connection with Tarski's Theorem.

More recent literature studies type semigroups in a
\(\Cst\)\nb-algebraic context, where only decompositions into some open
subsets should be allowed to get elements in the Cuntz semigroups of
groupoid \(\Cst\)\nb-algebras.  This may destroy the property of
having plain paradoxes.  Nevertheless, these type semigroups have
been useful.  For instance, if~\(X\) is a Cantor set with an action
of a group~\(G\) and~\(\OO\) is the family of all compact open subsets
in~\(X\), then a type semigroup was used by
Kerr--Nowak~\cite{Kerr-Nowak:Residually_finite} and by
Sierakowski--R\o rdam~\cite{Rordam-Sierakowski:Purely_infinite} to
study the properties of the dynamical system and the resulting
crossed product \({\Cont_0(X)\rtimes G}\).  In particular, Tarski's
Theorem was applied to characterise when \(\Cont_0(X)\rtimes G\) is
stably finite~\cite{Kerr-Nowak:Residually_finite} and purely
infinite~\cite{Rordam-Sierakowski:Purely_infinite}.  Independently,
B\"onicke--Li~\cite{Boenicke-Li:Ideal} and
Rainone--Sims~\cite{Rainone-Sims:Dichotomy} generalised these
results and the type semigroup to the setting of an ample, Hausdorff
groupoid.  The theory of type semigroups was further developed and applied, for
instance, in \cites{Ara_Bonicke_Bosa_Li:type_semigroup,
  Ara-Exel:Dynamical_systems,
  Pask-Sierakowski-Sims:Quasitraces_finite_infinite,
  Ma:Purely_infinite_groupoids}.   It is closely related to the purely
algebraic theory of type monoids associated to Boolean inverse semigroups
(see \cites{Wehrung:Monoids_Boolean, Ara-Bosa-Pardo-Sims:Separated_graphs}).
In this ample case, the type
semigroup is closely related to \(K\)-theory: it maps in a canonical way to the
Murray-von Neumann monoid of equivalence classes of projections in the
groupoid \(\Cst\)\nb-algebra.  In fact, a method of inducing maps on
rings from states on their \(K_0\)\nb-group using an analogue of
Tarski's Theorem has a long tradition
(see~\cite{Goodearl-Handelman:Rank_K0}).

For a general \(\Cst\)\nb-algebra~\(A\), which may have very few
projections, a good replacement for the Murray--von Neumann semigroup
is the Cuntz semigroup \(W(A)\) constructed from positive elements.
Cuntz~\cite{Cuntz:Dimension_functions} described quasitraces on~\(A\)
through states on~\(W(A)\), in order to ``make available for
\(\Cst\)\nb-algebras the results of Goodearl and
Handelman~\cite{Goodearl-Handelman:Rank_K0}''.  The Cuntz semigroup
turned out to be crucial in Elliott's programme of classifying
nuclear separable simple \(\Cst\)\nb-algebras (alias \emph{Elliott
  \(\Cst\)\nb-algebras}).  Toms~\cite{Toms:On_class} showed that the
classification only works under an additional regularity assumption.
According to the Toms--Winter conjecture, this should be one of the
following conditions: finite nuclear dimension, \(\ZZ\)\nb-stability
or strict comparison.  Nowadays, it is known that finite nuclear
dimension and \(\ZZ\)\nb-stability are equivalent for
infinite-dimensional, separable, simple, unital, nuclear C*-algebras.
In~\cite{Rordam:stable_and_real_rank} R\o{}rdam proved that every
\(\ZZ\)\nb-stable \(\Cst\)\nb-algebra~\(A\) has almost
unperforated~\(W(A)\), and that almost unperforation implies strict
comparison for exact simple unital~\(A\).  The latter
implication used a characterisation of stable domination in~\(W(A)\) in
terms of states on~\(W(A)\), which was made explicit in
\cite{Ortega-Perera-Rordam:CFP_Cuntz}*{Proposition~2.1} but it
appeared in proofs of \cite{Rordam:On_simple_II}*{Proposition 3.1},
\cite{Rordam:stable_and_real_rank}*{Proposition 3.2}.  It is
attributed to Goodearl--Handelman~\cite{Goodearl-Handelman:Rank_K0} in
these sources.  However, we believe it is much closer to Tarski's
Theorem.  Almost unperforation is closely related to the dichotomy
between pure infiniteness and stable finiteness for simple
\(\Cst\)\nb-algebras.  R\o{}rdam's examples in
\cite{Rordam:finite_and_infinite}*{Theorem~6.10} show that this
dichotomy may fail in general.

The theory of Cuntz semigroups suggests a definition of type
semigroups for actions on arbitrary spaces, not necessarily totally
disconnected.  Ma introduced in~\cite{Ma:Purely_infinite_groupoids} a
relevant semigroup generated by all open subsets in the unit space
of a Hausdorff locally compact \'etale groupoid, and before that he considered this semigroup
coming from discrete group actions \cite{Ma:type_semigroups_comparison}.  In the present
article, we improve upon his work and push this idea much further.
We consider an arbitrary \'etale groupoid~\(\Gr\) with a locally
compact Hausdorff unit space~\(X\).  So we allow~\(\Gr\) to be
non-Hausdorff.  In our construction, our definition of the type
semigroup \(S_\B(\Gr)\) depends on a set of bisections~\(\B\)
that implement the ``decompositions,'' which we call an
\emph{inverse semigroup basis} for~\(\Gr\) in
Definition~\ref{def:isg_basis}.  The flexibility of this choice is
very useful.  In the ample case, a natural choice for~\(\B\) is the set
of all compact open bisections; then~\(S_\B(\Gr)\) is the ordered
quotient of the type semigroup as defined
in~\cites{Boenicke-Li:Ideal, Rainone-Sims:Dichotomy}.  In general,
to relate \(S_\B(\Gr)\) with functions on the groupoid, we assume
that~\(\B\) consists of \(\sigma\)\nb-compact subsets.  In the twisted
case, we also need that the twist becomes trivial on the
subsets in~\(\B\).  These assumptions allow to build a canonical
homomorphism \(S_\B(\Gr)\to W(\Cst(\Gr,\LL))\).  Since the Cuntz
semigroup is natural \cites{Antoine-Perera-Thiel:Tensor_products_Cu, Combes-Zettl:Order_traces}, we may replace \(\Cst(\Gr,\LL)\) by any of its
quotients, such as the reduced \(\Cst\)\nb-algebra
\(\Cst_\red(\Gr,\LL)\) or the essential \(\Cst\)\nb-algebra
\(\Cst_\ess(\Gr,\LL)\).  Yet another natural assumption is
that~\(\B\) consists of bisections that are \emph{precompact} in the
sense that their ranges and sources are precompact subsets of~\(X\),
as some structural results
hold only in this case.  A crucial ingredient in the
construction and analysis of \(S_\B(\Gr)\) is the \emph{way-below
  relation}~\(\ll\), also called \emph{compact containment}.  In
particular, if the groupoid is not ample, we need to consider
\emph{regular ideals} and \emph{regular states} on \(S_\B(\Gr)\),
where regularity is defined using the auxiliary relation~\(\ll\).
For instance, we prove the following dynamical analogue of the
correspondence between quasi-traces and functionals on the Cuntz
semigroup (see \cite{Elliott-Robert-Santiago:The_Cone}*{Theorem 4.4},
\cite{Gardella-Perera:Cuntz_semigroups}*{Theorem 6.9}),
which may also be seen as a significant improvement of the
correspondence between regular states and regular groupoid dimension
functions in~\cite{Ma:Purely_infinite_groupoids}*{Theorem 4.11} (see
Theorem~\ref{thm:Riesz_for_monoids}):
\begin{thmx}
  \label{thmx:inducing_traces}
  There is a bijection between regular states on~\(S_\B(\Gr)\) and
  lower semicontinuous traces on~\(\Cst_\red(\Gr,\LL)\) that are
  induced from \(\Cont_0(X)\) through the canonical generalised
  conditional expectation.
\end{thmx}
A proof of Theorem \ref{thmx:inducing_traces} in full generality needs
a measure extension result  for  deficient topological measures from \cite{Butler:Deficient_topological_measures},
which we learned from a referee.
We need to use the reduced \(\Cst\)\nb-algebra in
Theorem~\ref{thmx:inducing_traces} because it is not clear how to
induce traces on the essential algebra~\(\Cst_\ess(\Gr,\LL)\).  Pure
infiniteness and simplicity criteria, however, work for
\(\Cst_\ess(\Gr,\LL)\) rather than for \(\Cst_\red(\Gr,\LL)\) (see
\cites{Kwasniewski-Meyer:Pure_infiniteness,
  Kwasniewski-Meyer:Essential}).
Hence we get best results when \(\Cst_\red(\Gr,\LL)=
\Cst_\ess(\Gr,\LL)\).
This is automatic when~\(\Gr\) is Hausdorff, but also holds in a
number of non-Hausdorff cases.
There has been some recent progress about when this happens after we
completed this work (see, for instance, \cite{Hume:Zero_singular_ideal}).

The following sample of results follows from Theorems
\ref{the:purely_infinite_semigroup}
and~\ref{thm:stably_finite_twisted}, and Corollaries
\ref{cor:pure_infinite_implies_paradoxical}
and~\ref{cor:dichotomy_for_topologically_free_groupoids}.
All the more technical notation will be explained below.

\begin{thmx}
  \label{thmx:sample_results}
  Let \((\Gr,\LL)\) be a twisted \'etale groupoid and let~\(\B\) be
  an inverse semigroup basis for~\(\Gr\) that consists of precompact
  \(\sigma\)\nb-compact bisections that trivialise the
  twist~\(\LL\).
  \begin{enumerate}
  \item\label{enu:sample_results1}%
    If~\(\Cst_\red(\Gr,\LL)\) is simple, it is stably finite if and
    only if~\(S_\B(\Gr)\) has a nontrivial state.

  \item\label{enu:sample_results2}%
    Suppose that~\(\Gr\) is residually topologically free and
    \((\Gr,\LL)\) is essentially exact.  There is a natural
    bijection between regular ideals in the type
    semigroup~\(S_\B(\Gr)\) and ideals in the essential
    \(\Cst\)\nb-algebra \(\Cst_\ess(\Gr,\LL)\).

  \item\label{enu:sample_results3}%
    Suppose, in addition to the assumption
    of~\ref{enu:sample_results2}, that there are only finitely many
    regular ideals in \(S_\B(\Gr)\) or that these ideals can be
    separated by the compact open subsets in \(\B \cap 2^X\).  Then
    if~\(S_\B(\Gr)\) is purely infinite, the
    \(\Cst\)\nb-algebra~\(\Cst_\ess(\Gr,\LL)\) is purely
    infinite.

  \item\label{enu:sample_results4}%
    If~\(S_\B(\Gr)\) is almost unperforated
    and~\(\Cst_\red(\Gr,\LL)\) is purely infinite,
    then~\(S_\B(\Gr)\) is purely infinite.
  \item\label{enu:sample_results5}%
    Assume that~\(\Gr\) is topologically free and that~\(S_\B(\Gr)\)
    has plain paradoxes.  If~\(\Cst_\red(\Gr,\LL)\) is
    simple, then it is either purely infinite or stably finite.
  \end{enumerate}
\end{thmx}

The dichotomy result in~\ref{enu:sample_results5} generalises results
of Rainone--Sims \cite{Rainone-Sims:Dichotomy}*{Theorem 7.4} and
B\"o{}nicke--Li \cite{Boenicke-Li:Ideal}*{Theorem~D}.
The latter require the groupoid to be Hausdorff and ample, and do not
consider twists.
A similar statement for locally compact Hausdorff \'etale groupoids,
without a twist, appears in
\cite{Ma:Purely_infinite_groupoids}*{Theorem 6.11} but it assumes
dynamical comparision, and is not using type semigroups.
The main ingredients that go into proving~\ref{enu:sample_results5} are
\ref{enu:sample_results1}, \ref{enu:sample_results3}
and~\ref{enu:sample_results4}.
The type semigroup characterisation of stable finiteness
in~\ref{enu:sample_results1} generalises of
\cite{Boenicke-Li:Ideal}*{Theorem~E} and of
\cite{Rainone-Sims:Dichotomy}*{Theorem 6.5}, both of which require the
groupoid to be ample, Hausdorff and have compact unit space.
A relevant result here is also
\cite{Kerr-Nowak:Residually_finite}*{Theorem 5.2}, which is stated for
minimal actions of free groups.
Pure inifiniteness criteria in \ref{enu:sample_results3} are a far
reaching generalisation of \cite{Boenicke-Li:Ideal}*{Corollary~C}, and
together with \ref{enu:sample_results4}, also of
\cite{Rainone-Sims:Dichotomy}*{Theorem 7.3}.
The statement \ref{enu:sample_results4} is a partial converse to
\ref{enu:sample_results3}.
Finally, \ref{enu:sample_results2} obtains a bijection between ideals
in the type semigroup and in \(\Cst_\ess(\Gr,\LL)\), which may be
viewed as a far-reaching generalisation of \cite{Boenicke-Li:Ideal}*{Theorem 3.10}.

To the best of our knowledge, type semigroups were not applied to
twisted dynamics before.  Allowing twists is important, as it allows
to apply our theory to Renault's Cartan \(\Cst\)\nb-inclusions
\(A\subseteq B\).  The class of \(\Cst\)\nb-algebras admitting
Cartan subalgebras is vast, and it contains all classifiable
\(\Cst\)\nb-algebras (see~\cite{Li:Classifiable_Cartan}).  For any
Cartan inclusion \(A\subseteq B\), we construct an ordered monoid
\(W(A,B)\) that can be used to study properties of~\(B\).  In
particular, applying~\ref{enu:sample_results5} in
Theorem~\ref{thmx:sample_results} we get the following version of
the dichotomy known to hold for simple \(\ZZ\)\nb-stable
\(\Cst\)\nb-algebras (see
Corollary~\ref{cor:Cartan_type_semigroup}):

\begin{corx}
  Let \(A\subseteq B\) be a Cartan inclusion such that the
  associated semigroup \(W(A,B)\) has plain paradoxes.
  Assume that~\(B\) is simple or, equivalently, that \(W(A,B)\) is
  simple.  Then~\(B\) is either properly infinite or stably finite.
\end{corx}

If a \(\Cst\)\nb-algebra with finite ideal structure is purely
infinite, then its ideals are separated by projections, which
is also called the ideal property.
In the presence of the ideal property, pure infiniteness is equivalent
to strong pure infiniteness, which is equivalent to
\(\mathcal{O}_\infty\)\nb-\alb{}stability under some extra conditions
(see~\cite{Kirchberg-Rordam:Infinite_absorbing}).
In general, it is not known whether strong pure infiniteness and pure
infiniteness are equivalent.
Our techniques to prove that \(\Cst\)\nb-algebras with infinitely many
ideals are purely infinite only work if we assume some projections to
exist (see Theorem~\ref{thmx:sample_results}.\ref{enu:sample_results3}
or Theorem~\ref{the:purely_infinite_semigroup}), and they imply the
ideal property at the same time.
For instance, we propose the following dichotomy for Kumjian's
\(\Cst\)\nb-diagonals as an analogue of R\o{}rdam's dichotomy for
separable nuclear
\(\Cst\)\nb-algebras~\cite{Rordam:stable_and_real_rank} (see Corollary
\ref{cor:diagonal_type_semigroup}):

\begin{corx}
  Let~\(B\) be a nuclear \(\Cst\)\nb-algebra with a
  \(\Cst\)\nb-diagonal \(A\subseteq B\) of real rank zero, such that
  \(W(A,B)\) is almost unperforated.  Then~\(B\) is either purely
  infinite or has a nontrivial lower semicontinuous trace.
\end{corx}

\(\Cst\)\nb-algebras with groupoid models which are not necessarily Hausdorff
have been receiving increased attention in recent years.
An important class of examples of \(\Cst\)\nb-algebras are those
coming from self-similar groups by
Nekrashevych~\cite{Nekrashevych:Cstar_selfsimilar}, which were
generalised by Exel--Pardo~\cite{Exel-Pardo:Self-similar} to
self-similar group actions.  Their motivation was to cover a class
of examples by Katsura~\cite{Katsura:Actions_Kirchberg}, which gave
models for all Kirchberg algebras.  These \(\Cst\)\nb-algebras are
all groupoid \(\Cst\)\nb-algebras, but the underlying groupoids may
fail to be Hausdorff even in very classical examples such as the
Grigorchuk group.  For a self-similar action of a group~\(\Gamma\)
on a row-finite graph~\(E\) with no sources, we calculate the
relevant type semigroup \(W(\Gamma,E)\) and we show that it is isomorphic
to the type semigroup \(W(E_\Gamma)\) of the ``quotient graph'' \(E_\Gamma\) constructed by Larki~\cite{Larki:dichotomy_for_self-similar_graphs}.
Using this and a result from \cite{Ara-Moreno-Pardo:Nonstable_K_for_Graph_Algs} we conclude
that the
monoid \(W(\Gamma,E)\) is always unperforated (and so  has plain paradoxes). As a consequence, we get the following dichotomy for
Exel--Pardo algebras
(see Theorem~\ref{thm:Dichotomy_Exel_Pardo}):

\begin{corx}
  Let \((\Gamma,E)\) be a self-similar action of a discrete
  group~\(\Gamma\) on a row-finite graph~\(E\) with no sources.  If
  the Exel--Pardo algebra~\(\OO_{(\Gamma,E)}\) is simple, then it is
  either purely infinite or stably finite.
\end{corx}

This generalises a theorem of
Larki~\cite{Larki:dichotomy_for_self-similar_graphs}, who assumes
that~\(\Gamma\) is amenable and that the action is ``pseudo-free'';
the latter implies, in particular, that the underlying groupoid is
Hausdorff.  The use of type semigroups clarifies a key step in the
proof in~\cite{Larki:dichotomy_for_self-similar_graphs}.

\smallskip

The paper is organised as follows.  In
Section~\ref{sec:preordered_abelian}, we discuss a generalisation of Tarski's theorem due to R\o{}rdam and study relevant facts in the generality of
preordered monoids.  We also briefly recall generalities about the
Cuntz semigroup.  In Section~\ref{sec:twisted_groupoids} we gather
and prove some basic facts concerning \(\Cst\)\nb-algebras
associated to twisted \'etale (not necessarily Hausdorff) groupoids.
Section~\ref{sec:type_semigroup} introduces the type semigroup.  We
first treat the case of a topological space, and then take the
quotient by an equivalence relation of dynamical nature that depends
on the choice of a family~\(\B\) of bisections of the groupoid.  We
also define the type semigroup in an equivalent way using a
stabilised groupoid, discuss its behavior under Morita equivalence,
and relate it to the Cuntz semigroups of twisted groupoid
\(\Cst\)\nb-algebras.  Section~\ref{sec:ideals_states} is devoted to
regular states and ideals in the type semigroup and their
relationship with the groupoid itself and groupoid
\(\Cst\)\nb-algebras.  In Section~\ref{sec:pure_infiniteness} we
prove our main results on pure infiniteness and stable finiteness.
We also define a type semigroup for a Cartan inclusion.  Finally, in
Section~\ref{sec:self-similar} we calculate the type semigroups for
self-similar actions and apply our results to Exel--Pardo algebras.

\section{Preordered abelian monoids}
\label{sec:preordered_abelian}

Any abelian monoid carries an intrinsic algebraic preorder, and many
other sources work in this setting (see, for instance,
\cites{Ara_Bonicke_Bosa_Li:type_semigroup,Boenicke-Li:Ideal,
  Ara-Exel:Dynamical_systems, Rainone-Sims:Dichotomy,
  Pask-Sierakowski-Sims:Quasitraces_finite_infinite}).  Our work,
however, needs more general partial orders.  To clarify the
relationships and differences, as well as to get a clear picture, we
allow general preordered monoids where possible.

We denote by \(\N=\{0,1,\dotsc\}\) the abelian monoid of natural
numbers starting from zero.  A \emph{preordered
  \textup{(}abelian\textup{)} monoid} is an abelian semigroup~\(S\)
with a neutral element~\(0\) and a preorder relation~\(\le\)
on~\(S\) such that \(0\le x\) for every \(x\in S\), and \(x\le y\)
implies \(x+z\le y+z\) for all \(x,y,z\in S\).  When~\(\le\) is a
partial order, we call~\(S\) an \emph{ordered
  \textup{(}abelian\textup{)} monoid}.  A preordered monoid~\(S\) is
\emph{conical} if \(x\le 0\) implies \(x=0\).  Every abelian
monoid~\(S\) is
a preordered monoid when equipped with the \emph{algebraic order}:
\(x\le y\) if \(x+z=y\) for some \(z\in S\).  Every preordered
monoid~\(S\) induces an ordered monoid
\(\widetilde{S}\defeq S/{\approx}\), where \(x\approx y\) if
\(x\le y\) and \(y\le x\).  Then \([x]+[y]=[x+y]\), and
\([x]\le [y]\) if \(x\le y\).  The map
\(S\ni x\mapsto [x]\in \widetilde{S}\) sends nonzero elements to
nonzero ones if and only if~\(S\) is conical.

An \emph{ideal} in a preordered monoid~\(S\) is a submonoid~\(I\)
such that \(x\le y\in I\) implies \(x\in I\) (this is called an
\emph{order ideal} in~\cite{Ara_Bonicke_Bosa_Li:type_semigroup}).
The ideal generated by \(y\in S\) is
\[
  \langle y \rangle\defeq \setgiven{x\in S}{x \le n\cdot y \text{
      for some } n\in \N}.
\]
We call \(y\in S\setminus\{0\}\) an \emph{order unit} if
\(\langle y \rangle = S\) or, equivalently, if for every \(x\in S\)
there is \(n\in \N\) with \(x\le n y\).  The monoid~\(S\) is
\emph{simple} if it has no nontrivial ideals or, equivalently, if every
nonzero element is an order unit.  Any simple preordered monoid is
conical or satisfies \(\tilde{S} = \{0\}\).

For an ideal~\(I\), the quotient preordered monoid
is defined as \(S/I\defeq S/{\sim}\) where we declare \(x\sim y\) if
and only if \(x+a=y+b\) for some \(a,b\in I\).  Then
\([x]+[y]=[x+y]\), and \([x]\le [y]\) if \(x\le y+a\) for some
\(a\in I\).  Note that \(S/I\) is always conical.
Any quotient of an algebraically preordered monoid by an ideal is
algebraically ordered.

\begin{definition}
  \label{def:paradoxical}
  An element~\(x\) in a preordered monoid~\(S\) is \emph{infinite}
  if it is nonzero and \(x+y\le x\) for some
  \(y\in S\setminus\{0\}\).  Otherwise, \(x\) is \emph{finite}.  We
  call \(x\in S\setminus\{0\}\) \emph{properly infinite} if
  \(2 x\le x\).  We call~\(S\) \emph{purely infinite} if every
  \(x\in S\setminus\{0\}\) is properly infinite and \(S\neq \{0\}\).
\end{definition}

\begin{remark}
  \label{rem:about_infnite_elements}
  If~\(S\) is conical, then an element~\(x\) is (properly) infinite
  in~\(S\) if and only if~\([x]\) is (properly) infinite in the
  ordered monoid~\(\widetilde{S}\).  In an ordered monoid, the
  inequalities \(x+y\le x\) and \(2 x\le x\) are equalities.
\end{remark}

The next lemma is analogous to the \(\Cst\)\nb-algebraic result
\cite{Kirchberg-Rordam:Non-simple_pi}*{Proposition~3.14}.

\begin{lemma}
  \label{lem:infiniteness_ideal}
  Let~\(S\) be a preordered abelian monoid and \(y\in S\).  Then
  \[
    I(y)\defeq\setgiven{z\in S}{y+z \le y}
  \]
  is an ideal in~\(S\)
  contained in~\(\langle y \rangle\).  If \(y\neq 0\) or if~\(S\) is
  conical, then
  \begin{enumerate}
  \item\label{enu:infiniteness_ideal1}%
    \(y\) is infinite if and only if \(I(y)\neq 0\);
  \item\label{enu:infiniteness_ideal2}%
    \(y\) is properly infinite if and only if
    \(I(y)= \langle y \rangle \neq 0\);
  \item\label{enu:infiniteness_ideal3}%
    the image of~\(y\) in \(S/I(y)\) is finite.
  \end{enumerate}
\end{lemma}

\begin{proof}
  If \(x\le z \in I(y)\), then \(y+x\le y+z\le y\) and so
  \(x\in I(y)\).  If \(x,z\in I(y)\), then \(x+z\in I(y)\) because
  \(x+y+z\le y+z\le y\), and \(x\in \langle y\rangle\) because
  \(x\le x+y\le y\).  Hence \(I(y)\) is an ideal contained
  in~\(\langle y\rangle\).  \ref{enu:infiniteness_ideal1} is obvious
  (if~\(S\) is conical, then \(I(y)\neq 0\) implies \(y\neq 0\)).
  \ref{enu:infiniteness_ideal2} holds because a nonzero
  element~\(y\) is properly infinite if and only if \(y\in I(y)\),
  if and only if \(I(y)= \langle y \rangle\).  To
  see~\ref{enu:infiniteness_ideal3} assume that the image of~\(y\)
  in \(S/I(y)\) is infinite.  Then there are
  \(x\in S\setminus I(y)\) and \(z,z'\in I(y)\) such that
  \(x+y +z\le y+z'\).  But then \(x+y\le x +y+z\le y+z' \le y\).  So
  \(x\in I(y)\), a contradiction.
\end{proof}

\begin{lemma}
  \label{lem:residually_infinite}
  For any \(y\in S\setminus\{0\}\) in a preordered monoid~\(S\) the
  following are equivalent:
  \begin{enumerate}
  \item\label{enu:residually_infinite1}%
    \(y\) is properly infinite;
  \item\label{enu:residually_infinite2}%
    for every ideal~\(I\) in~\(S\) with \(y\notin I\), the image
    of~\(y\) in~\(S/I\) is infinite;
  \item\label{enu:residually_infinite3}%
    \(\langle y \rangle=\setgiven{x\in S}{x \le y}\).
  \end{enumerate}
  A  preordered monoid~\(S\) is simple and purely infinite if
  and only if \(\widetilde{S}\subseteq \{0,\infty\}\).
\end{lemma}

\begin{proof}
  If~\(y\) is properly infinite, then its image in~\(S/I\) is
  clearly properly infinite for every ideal~\(I\) in~\(S\) with
  \(y\notin I\).  Hence \ref{enu:residually_infinite1}
  implies~\ref{enu:residually_infinite2}.  Conversely, if~\(y\) is
  not properly infinite, then~\(I(y)\) is a proper subideal
  of~\(\langle y\rangle\) by
  Lemma~\ref{lem:infiniteness_ideal}.\ref{enu:infiniteness_ideal2}.
  In particular, \(y\not\in I(y)\).  The image of~\(y\) is finite
  in~\(S/I(y)\) by
  Lemma~\ref{lem:infiniteness_ideal}.\ref{enu:infiniteness_ideal3}.
  Thus~\ref{enu:residually_infinite2}
  implies~\ref{enu:residually_infinite1}.  Since
  \(2y\in \langle y \rangle\), \ref{enu:residually_infinite3}
  implies \ref{enu:residually_infinite1}.  Conversely, if~\(y\) is
  properly infinite, then \(x+y\le y\) for all
  \(x\in \langle y \rangle\) by
  Lemma~\ref{lem:infiniteness_ideal}.\ref{enu:infiniteness_ideal2}.
  Since \(x\le x+y\), we conclude that \(x\le y\).
  Hence~\ref{enu:residually_infinite1}
  implies~\ref{enu:residually_infinite3}.

  If~\(S\) is simple and purely infinite,
  then~\ref{enu:residually_infinite3}
  implies \(x\le y\) and \(y\le x\) for all
  \(x,y\in S\setminus\{0\}\).  Thus all nonzero elements are
  equivalent, that is, \(\tilde{S} \subseteq \{0,\infty\}\).
  Conversely, in the latter case~\(\tilde{S}\) is purely infinite
  and then so is~\(S\).
\end{proof}

\begin{remark}
  \label{rem:purely_infinite_simple_ordered}
  By the above lemma, an ordered monoid~\(S\) is simple purely
  infinite if and only if \(S\cong \{0,\infty\}\), where~\(\infty\)
  is an idempotent.  Preordered monoids that are purely infinite and
  simple may have a much more complex structure.  For instance,
  the Murray--von Neumann semigroup of any Kirchberg
  \(\Cst\)\nb-algebra is a conical, algebraically ordered refinement
  monoid that is purely infinite and simple.
\end{remark}

Recall that  a preordered monoid~\(S\) is \emph{unperforated} if for all
\(x,y\in S\) we have \(x\le y\) whenever \(nx\le n y\) for some \(n\ge 1\).
We will use weaker versions of this condition in our results.

\begin{definition}
  \label{def:stably_dominated}
  If \((n+1)x\le n y\) for some \(n\ge 1\), we call~\(x\)
  \emph{stably dominated by}~\(y\) and write \(x<_\s y\) (see
  \cite{Ortega-Perera-Rordam:CFP_Cuntz}*{Definition~2.2}).  An
  element \(x\in S\setminus\{0\}\) is \emph{paradoxical} if
  \(x<_\s x\).  The preordered monoid~\(S\) is \emph{almost
    unperforated} if \(x<_\s y\) implies \(x\le y\) for all
  \(x,y\in S\).
\end{definition}

\begin{remark}
  \label{rem:paradoxical_elements}
  Every properly infinite element is paradoxical.  If
  \((n+1)x\le n x\), then \((n+k)x\le n x\) for every \(k\ge 1\).
  In particular, \(x <_\s x\) implies \(2x <_\s x\).  Hence if~\(S\)
  is almost unperforated, then being paradoxical is the same as
  being properly infinite.  If \(x<_\s y\) then~\(x\) is in the
  ideal \(\langle y \rangle\) generated by~\(y\).  Hence, in view of
  Lemma~\ref{lem:residually_infinite}, if~\(S\) is purely
  infinite it is almost unperforated.
\end{remark}

\begin{lemma}\label{lem:paradoxical_implies_properly_infinite}
  An element \(x\in S\) in a conical preordered abelian monoid is
  paradoxical if and only if \(n x\) is properly infinite for some
  \(n\ge 1\).
\end{lemma}

\begin{proof}
  If~\(nx\) is properly infinite, then \(x\neq 0\) and
  \(2nx\le n x\).  This implies \((n+1)x\le n x\), so~\(x\) is
  paradoxical.  If~\(x\) is paradoxical, then \(nx +x\le n x\) for
  some \(n\ge 1\).  Since \(x\neq 0\) and~\(S\) is conical, this
  implies \(nx\neq 0\).  As in Remark~\ref{rem:paradoxical_elements}
  we get \(2n x \le nx\).  Hence~\(nx\) is properly infinite.
\end{proof}

\begin{definition}
  \label{def:state}
  A \emph{state} on an ordered abelian monoid~\(S\) is an additive
  and order-preserving map \(\nu\colon S\to [0,\infty]\) with
  \(\nu(0)=0\) or, equivalently, \(\nu \not\equiv \infty\).  It is
  \emph{faithful} if \(\nu(S\setminus\{0\})\subseteq (0,\infty]\)
  and it is \emph{finite} if \(\nu(S)\subseteq [0,\infty)\).  It is
  \emph{trivial} if it takes only the values \(0\) and~\(\infty\),
  and \emph{nontrivial} otherwise.
\end{definition}

\begin{remark}
  \label{rem:trvial_states_and_ideals}
  Let~\(S\) be a preordered monoid.  There is a bijection between
  trivial states on~\(S\) and ideals in~\(S\), mapping a state~\(\nu\)
  to the ideal \(\setgiven{t\in S}{\nu(t)=0}\).  Since \([0,\infty]\)
  is partially ordered, there is a bijection between states~\(\nu\) on
  a preordered monoid~\(S\) and states~\(\widetilde{\nu}\) on the
  associated ordered monoid~\(\widetilde{S}\), where
  \(\nu(x)= \widetilde{\nu}([x])\) for \(x\in S\).  This implies a natural
  bijection between ideals in \(S\) and~\(\widetilde{S}\).
\end{remark}

\begin{remark}
  \label{rem:faithful_implies_conical}
  If~\(S\) admits a faithful state, then~\(S\) is conical.  If~\(S\)
  is conical, then the correspondence in
  Remark~\ref{rem:trvial_states_and_ideals} restricts to a bijection
  between faithful states on~\(S\) and faithful states on \(\widetilde{S}\).
\end{remark}

\begin{remark}
  \label{rem:faithful_states_algebraic_order}
  Let~\(S\) be a monoid with the algebraic preorder that admits a
  faithful finite state~\(\nu\).  Then \(S=\widetilde{S}\) is an ordered
  monoid.  Indeed, if \(x,y\in S\) satisfy \([x]=[y]\), then
  \(x+x'=y\) and \(x=y+y'\) for some \(x',y'\in S\).  Therefore,
  \(\nu(x+y) +\nu(x')+\nu(y')=\nu(x+y)\), which implies \(x'=y'=0\)
  because~\(\nu\) is faithful and finite.  Hence \(x=y\).
\end{remark}

\begin{lemma}
  \label{lem:simplicity_implies_faithfulness_of_states}
  Every nontrivial state on a simple monoid is faithful and finite.
\end{lemma}

\begin{proof}
  If \(S\) is simple then for any \(x,y\in S\setminus\{0\}\) there are
  \(m,n\in \N \setminus\{0\}\) with \(x\leq n y\) and \(y \leq m x\).
  Hence if \(\nu\) is a state with \(\nu(x)\in (0,\infty)\), then also
  \(\nu(y)\in (0,\infty)\).
\end{proof}

Tarski's Theorem says, roughly speaking, that all elements that are
not paradoxical may be seen by nontrivial states.
Its original version for monoids with algebraic preorder was proven
in~\cite{Tarski1938} (see also
\cite{Wagon-Tomkowicz:Banach-Tarski_Paradox}*{Theorem~9.1}).
Tarski’s results were further developed in the 1950s in work of Cotlar
and Aumann~\cite{Aumann:Erweiterungen_additiv} on extensions of
additive monotone functionals on ordered semigroups; this forms part
of the historical background of the extension-type arguments used
here.
For ordered monoids, Tarski's Theorem follows from
Ortega--Perera--R\o{}rdam's characterisation of stable domination in
terms of states (see
\cite{Ortega-Perera-Rordam:CFP_Cuntz}*{Proposition~2.1}), whose proof
refers to Goodearl--Handelman's extension result
\cite{Goodearl-Handelman:Rank_K0}*{Lemma~4.1}, where only ordered
groups are considered.
However, the authors of~\cite{Ortega-Perera-Rordam:CFP_Cuntz} should
in fact refer to R\o{}rdam's
\cite{Rordam:On_simple_II}*{Proposition~3.1}, which uses
\cite{Goodearl-Handelman:Rank_K0}*{Lemma~4.1} in a non-trivial way.
Moreover, the proof of
\cite{Ortega-Perera-Rordam:CFP_Cuntz}*{Proposition~2.1} repeats the
arguments in the proofs of R\o{}rdam's results
\cite{Rordam:On_simple_II}*{Proposition~3.2} and
\cite{Rordam:stable_and_real_rank}*{Proposition~3.2}, which are valid
in the setting of partially ordered semigroups.
Therefore, both the characterisation of stable domination and Tarski's
Theorem hold in the full generality of partially ordered
semigroups\footnote{A full self-contained proof based on an adaptation
  of the proof of Tarski's Theorem can be found in Appendix A of the
  initial version of this article, see arXiv:2502.17190 v2}:

\begin{theorem}[Stable domination characterisation]
  \label{thm:Handelman's_tarski}
  Let~\(S\) be a preordered abelian monoid and let \(x,y\in S\).  Then
  \(x<_\s y\) if and only if \(x\in \langle y \rangle\) and
  \(\nu(x)<\nu(y)\) for all states~\(\nu\) on~\(S\) with \(\nu(y)=1\).
\end{theorem}

\begin{remark}
  \label{rem:dynamical_comparison_for_semigroups}
  Theorem~\ref{thm:Handelman's_tarski} holds by the proof of
  \cite{Rordam:stable_and_real_rank}*{Proposition~3.2}, which is
  stated in a weaker form.  Namely, it says that almost unperforation
  is equivalent to the \emph{strict comparison} property: \(x\le y\)
  whenever~\(x\) is in the ideal generated by~\(y\) and
  \(\nu(x)<\nu(y)\) for all states~\(\nu\) on~\(S\) with \(\nu(y)=1\).
\end{remark}

\begin{corollary}[Tarski's Theorem]
  \label{cor:original_Tarski}
  In any preordered abelian monoid~\(S\), an element
  \(y\in S\setminus \{0\}\) is not paradoxical if and only if there is
  a state \(\nu\colon S\to [0,\infty]\) with \(\nu(y)=1\).
\end{corollary}

\begin{proof}
  Apply Theorem~\ref{thm:Handelman's_tarski} to \(x=y\).
\end{proof}

\begin{remark}
  \label{rem:Tarski_proof}
  Tarski's Theorem is also a consequence of \([0,\infty]\) being
  injective in the category of positively ordered monoids, see
  \cite{Blackadar-Rordam:Extending_states}*{Corolllary 2.7} or
  \cite{Wehrung:Injective}*{Example 3.10}.  Indeed, it is
  straightforward that \(y\in S\setminus \{0\}\) is not paradoxical if
  and only if there is a well-defined state \(\phi\) on the submonoid
  \(S_0\defeq\{0,y,2y,\dotsc\}\subseteq S\) with \(\phi(y)=1\). If
  this holds, then~\(\phi\) extends to~\(S\) because \([0,\infty]\) is
  injective.
\end{remark}

We give a name to a consequence of almost unperforation that suffices
for many of our results. This is called property (QQ)
in~\cite{Ortega-Perera-Rordam:CFP_Cuntz}.

\begin{definition}
  A preordered abelian monoid~\(S\) \emph{has plain paradoxes} if
  \((n+1)x\le n x\) implies \(2x\le x\) for all \(x\in S\),
  \(n\ge 1\), that is, if every paradoxical element in~\(S\) is
  properly infinite.
\end{definition}

\begin{remark}
  \label{rem:assumptions_needed_for_dichotomy}
  Let~\(S\) be a preordered abelian monoid. It has plain paradoxes in
  the following cases:
  \begin{itemize}
  \item if~\(S\) is almost unperforated by
    Remark~\ref{rem:paradoxical_elements};
  \item if~\(S\) is an algebraically ordered monoid with the
    refinement property and the \emph{strong Corona factorisation property}
    by \cite{Ortega-Perera-Rordam:CFP_refinement}*{Theorem~5.14};
  \item if~\(S\) is a simple refinement monoid with the \emph{Corona
      factorisation property};
  \item if~\(S\) is simple and every paradoxical element is infinite
    (see Lemma~\ref{lem:residually_infinite}).
  \end{itemize}
\end{remark}

\begin{remark}
  \label{rem:almost_unperforation_for_ordered_quotient}
  A preordered monoid~\(S\) is (almost) unperforated if and only if its
  ordered quotient~\(\widetilde{S}\) is.  The same holds with
  paradoxes being plain.  In particular, if \(x\in S\setminus\{0\}\)
  and \([x]=0\) in~\(\widetilde{S}\), then \(x\le 0\) is properly
  infinite in~\(S\).  For conical~\(S\), see
  Remark~\ref{rem:about_infnite_elements}.
\end{remark}

\begin{corollary}
  \label{cor:dichotomy_for_monoids}
  Let~\(S\) be a nonzero preordered abelian monoid.  If~\(S\) has
  plain paradoxes, then either~\(S\) admits a nontrivial
  state or~\(S\) is purely infinite.  Conversely, if~\(S\) is simple
  and either~\(S\) admits a nontrivial \textup{(}necessarily
  faithful finite\textup{)} state or~\(S\) is purely infinite,
  then~\(S\) has plain paradoxes.
\end{corollary}

\begin{proof}
  By Tarski's Theorem, if~\(S\) has no nontrivial states, then all
  nonzero elements in~\(S\) are paradoxical.  Then~\(S\) is purely
  infinite if and only if~\(S\) has plain paradoxes.
  If~\(S\) is simple and admits a nontrivial state, this state has
  to be faithful and finite, and then~\(S\) does not have
  paradoxical elements by the easy direction in Tarski's Theorem.
  In particular, \(S\) has plain paradoxes.
\end{proof}

\begin{example}
  The smallest simple monoid that fails to have plain paradoxes
  is \(S\defeq\{0,1,\infty\}\), where~\(\infty\) is an absorber,
  \(1+1=\infty\), and~\(S\) is equipped with the natural linear
  order, which is also the algebraic order.  The element~\(1\) is
  paradoxical, but not properly infinite.
\end{example}

\begin{definition}[\cite{Gierz_Hofmann_Keimel_Lawson_Mislove_Scott:Continuous_lattices}]
  \label{def:way_below}
  Let \((S,{\le})\) be a poset and \(x,y\in S\).  We write \(x\ll y\)
  and say that~\(x\) is \emph{way below}~\(y\) if, for any directed
  subset \(D\subseteq S\) for which \(\sup D\) exists and
  \(\sup D \ge y\), there is \(d\in D\) with \(x \le d\).  We
  call~\(S\) \emph{continuous} if for every \(x\in S\) the set
  \(\setgiven{y\in S}{y\ll x }\) is directed and \(x\) is its supremum.  (Continuity is called the ``axiom
  of approximation'' in
  \cite{Gierz_Hofmann_Keimel_Lawson_Mislove_Scott:Continuous_lattices}*{Definition~I-1.6}.)
\end{definition}

\begin{remark}[see
  \cite{Gierz_Hofmann_Keimel_Lawson_Mislove_Scott:Continuous_lattices}*{Proposition~I-1.2}]
  \label{rem:easy_ll_properties}
  If \(x\ll y\), then \(x\le y\).  If \(w\le x\ll y\le z\), then
  \(w \ll z\).  In particular, the relation~\(\ll\) is transitive.
  If \(x,y \ll z\) and \(x\vee y\) exists, then \(x\vee y \ll z\).
  Therefore, if \(x\vee y\) exists for all \(x,y\in S\), then the
  set of \(x\in S\) with \(x\ll y\) is directed.
\end{remark}

\begin{definition}
  \label{def:closed_lower_semicontinuous}
  Let \((S,{\le})\) be an ordered monoid.  We call a state
  \(\nu\colon S\to [0,\infty]\) \emph{lower semicontinuous} if
  \(\nu(x)=\sup {}\setgiven{\nu(y)}{y\in S \text{ and }y\ll x }\)
  for every \(x\in S\).  An ideal~\(I\) in~\(S\) is \emph{closed}
  if~\(f\in I\) whenever~\(\{k\in I\colon k\ll f\}\subseteq I\).
\end{definition}

\begin{remark}
  The above definition works best, and is usually formulated
  for~\(S\) continuous and \emph{directed complete}, which means
  that every increasing net in~\(S\) has a supremum in~\(S\).  Then
  a state \(\nu\colon S\to [0,\infty]\) is lower semicontinuous if
  it respects suprema of increasing nets, and an ideal~\(I\) is
  closed if it is closed under suprema of increasing nets.
  States that respect suprema of increasing sequences are also called \emph{functionals} in the literature.
\end{remark}

Let~\(B\) be a \(\Cst\)\nb-algebra.  We recall the definitions of
the (uncompleted) ordered \emph{Cuntz semigroup} \(W(B)\) introduced
in~\cite{Cuntz:Dimension_functions} and its completed version
\(\Cu(B)\) studied in~\cite{Coward-Elliott-Ivanescu:Cuntz_invariant}
(see \cite{Gardella-Perera:Cuntz_semigroups} for a modern account).
For positive elements \(a, b\in B^+\), we write \(a \precsim b\) and
call~\(a\) \emph{Cuntz subequivalent} to~\(b\) if, for every
\(\varepsilon>0\), there is \(x \in B\) with
\(\norm{a-x^* b x} <\varepsilon\).  By
\cite{Kirchberg-Rordam:Infinite_absorbing}*{Lemma 2.3.(iv)},
\begin{equation}
  \label{eq:algebraic_Cuntz_comparison}
  a \precsim b \quad \Longleftrightarrow\quad
  \forall_{\varepsilon >0}\ \exists_{z\in B}\quad
  (a-\varepsilon)_+=z^*z\ \text{and}\ zz^*\in b B b.
\end{equation}
Here \((a-\varepsilon)_+\in B\) is the positive part of
\(a-\varepsilon\cdot 1\in \Mult(B)\).

We call \(a,b\in B^+\) \emph{Cuntz equivalent} and write
\(a\approx b\) if \(a \precsim b\) and \(b\precsim a\).  Let
\(\Cu(B)\defeq (B\otimes \Comp)^+/{\approx}\) be the set of Cuntz
equivalence classes of positive elements in \(B\otimes \Comp\).  It is
an abelian ordered monoid where for
\(a,b \in (B\otimes \Comp)^+\) we write \([a]\le [b]\) if and only
if \(a\precsim b\), and \([a]+[b] \defeq [a'+b']\) where \(a'\)
and~\(b'\) are orthogonal and Cuntz equivalent to \(a\) and~\(b\),
respectively.  The semigroup~\(W(B)\) is then the subsemigroup of
\(\Cu(B)\) of Cuntz classes with a representative in
\(M_{\infty}(B)^+\defeq\bigcup_{n=1} M_n(B)^+\).  Namely,
\(W(B)=M_{\infty}(B)^+/{\approx}\), where for \(a\in M_n(B)^+\) and
\(b\in M_m(B)^+\), we write \(a\precsim b\) if \(x_k b x_k^*\to a\)
for some sequence \((x_k)\in M_{m,n}(B)\), and
\(a\approx b\) if \(a\precsim b\) and \(b\precsim a\).  Then
\(W(B)\) is equipped with the addition induced by the direct sum
\((a\oplus b) = \text{diag}(a, b)\in M_{n+m}(B)\), and with the
order induced by \(a\precsim b\).

The Cuntz semigroup accommodates a lot of information about the
\(\Cst\)\nb-algebra.  For instance, the \(\Cst\)\nb-algebra~\(B\) is
\emph{purely infinite} if and only if the Cuntz semigroup \(\Cu(B)\)
is purely infinite (see~\cite{Kirchberg-Rordam:Non-simple_pi}).  In
general, an element \(b\in B^+\setminus\{0\}\) is \emph{properly
  infinite} or \emph{infinite} in \(B\), respectively, if and only
if~\([b]\) is properly infinite or infinite in \(W(B)\).
The ordered
monoid \(\Cu(B)\) together with \(K_1(B)\), or an appropriately modified version of \(\Cu(B)\), is a complete invariant for finite classifiable
\(\Cst\)\nb-algebras
(see~\cites{Antoine-Dadarlat-Perera-Santiago:Recovering_Elliott, Gardella-Perera:Cuntz_semigroups, Cantier:unitary_Cuntz_semigroup}).

The Cuntz semigroup \(\Cu(B)\) is both sequentially directed complete
and sequentially continuous (one can drop ``sequentially'' when~\(B\)
is separable), and \([a]\ll[b]\) if and only if
\(a\precsim (b-\varepsilon)_+\) for some \(\varepsilon >0\) (see
\cites{Gardella-Perera:Cuntz_semigroups,
  Antoine-Perera-Thiel:Tensor_products_Cu}).  There is a bijection
between ideals in~\(B\) and closed ideals in \(\Cu(B)\), and
quasi-traces on~\(B\) and functionals on \(\Cu(B)\) (see
\cites{Ciuperca-Robert-Santiago:Cuntz_semigroup_ideals,
  Elliott-Robert-Santiago:The_Cone,
  Gardella-Perera:Cuntz_semigroups}).  In particular, \(B\) is simple
if and only if \(\Cu(B)\) is \emph{topologically simple}, that is, it
contains no nontrivial closed ideals.  By the main result
of~\cite{Rordam:stable_and_real_rank}, if~\(B\) is
\(\ZZ\)\nb-absorbing, then \(\Cu(B)=W(B\otimes \Comp)\) is almost
unperforated.  One may formulate a dichotomy for simple
\(\Cst\)\nb-algebras as follows:

\begin{proposition}
  Let~\(B\) be a \(\Cst\)\nb-algebra and let \(S\subseteq \Cu(B)\)
  be a subsemigroup which is dense in the sense that
  \(\sup {}\setgiven{y\in S}{y\ll x }=x \) for every \(x\in \Cu(B)\).
  If~\(S\) is simple and has plain paradoxes \textup{(}which
  is automatic when~\(B\) is \(\ZZ\)\nb-absorbing\textup{)},
  then~\(B\) is simple and either purely infinite or stably finite.
\end{proposition}

\begin{proof}
  Since~\(S\) is simple and dense in~\(\Cu(B)\), the latter has no
  nontrivial closed ideals.  Then~\(B\) is simple.  If~\(S\) is purely
  infinite and simple, then \(S=\{0,\infty\}\) by
  Remark~\ref{rem:purely_infinite_simple_ordered}.  Since it is
  dense in \(\Cu(B)\), also \(\Cu(B)=\{0,\infty\}\).  Then~\(B\) is
  simple and purely infinite.  Otherwise,
  Corollary~\ref{cor:dichotomy_for_monoids}, provides a faithful
  finite state~\(\nu\) on~\(S\).  Then the formula
  \(\overline{\nu}(x)\defeq\sup {}\setgiven{\nu(y)}{y\ll x,\ y\in S}\)
  defines a faithful functional on \(\Cu(B)\).  This,
  in turn, induces a faithful semi-finite lower semicontinuous
  2\nb-quasi-trace on~\(B\), and so~\(B\) is stably finite (see, for
  instance,
  \cite{Blanchard-Kirchberg:Non-simple_infinite}*{Remark~2.27(viii)}).
\end{proof}

\begin{remark}
  Allowing subsemigroups of the Cuntz semigroup is important.  For
  instance, the ordered monoid \(W(\C)=\N\) is simple, but
  \(\Cu(\C)=\N\cup \{\infty\}\) is only topologically simple.
\end{remark}

\section{Twisted groupoid \texorpdfstring{\(\Cst\)}{C*}-algebras}
\label{sec:twisted_groupoids}

Throughout this paper, \(\Gr\) stands for an \'etale groupoid with a
locally compact Hausdorff unit space~\(X\).  Hence the range and
source maps \(\rg,\s\colon\Gr\to X\subseteq \Gr\) are open and
locally injective, and~\(\Gr\) is locally compact and locally
Hausdorff.

Recall that an (open) \emph{bisection}, or a \emph{slice},
of~\(\Gr\), is an open subset \(U\subseteq \Gr\) such that
\(\rg|_U\) and~\(\s|_U\) are injective.  Then the map
\(\theta_U\defeq \rg|_U\circ \s|_U^{-1}\colon \s(U)\to \rg(U)\) is a
partial homeomorphism of~\(X\).  The family \(\Bis(\Gr)\) of all
bisections becomes an inverse semigroup with the operations
\[
  U\cdot V\defeq \setgiven{\gamma\cdot\eta}{\gamma\in U,\ \eta\in
    V},
  \qquad
  U^{-1}\defeq \setgiven{\gamma^{-1}}{\gamma\in U}
\]
for \(U,V\in \Bis(\Gr)\), and \(\emptyset\in\Bis(\Gr)\) is a zero
and \(X\in \Bis(\Gr)\) is a unit element.  The partial
homeomorphisms \((\theta_U)_{U\in \Bis(\Gr)}\) constitute an inverse
semigroup action of \(\Bis(\Gr)\) on~\(X\).  The corresponding
transformation groupoid \(X\rtimes \Bis(\Gr)\) is canonically
isomorphic to~\(\Gr\).  More generally, for any inverse subsemigroup
\(S\subseteq \Bis(\Gr)\) there is a canonical homomorphism
\(X\rtimes S\to \Gr\).  It is an isomorphism if and only if~\(S\) is
\emph{wide}, that is, \(\bigcup S = \Gr\) and \(U\cap V\) is a union
of bisections in~\(S\) for all \(U,V\in S\) (see, for instance,
\cite{Kwasniewski-Meyer:Essential}*{Proposition~2.2}).

A \emph{twist} over~\(\Gr\) is a Fell line bundle~$\LL$ over~$\Gr$
in the sense of~\cite{Kumjian:Fell_bundles}.  Thus
\(\LL=(L_\gamma)_{\gamma\in \Gr}\) is a locally trivial bundle of
one-dimensional complex Banach spaces, together with multiplication
maps
\(L_{\gamma}\times L_{\eta}\ni (z_{\gamma},z_{\eta})\mapsto
z_{\gamma}\cdot z_{\eta}\in L_{\gamma\eta}\) for
\((\gamma,\eta)\in \Gr^2\), and involutions
\(L_{\gamma}\ni z\mapsto \overline{z}\in L_{\gamma^{-1}}\) for
\(\gamma\in\Gr\), which are continuous and consistent with each
other in a certain way.  Any such bundle is locally trivial.  In
fact, the family
\begin{equation}
  \label{eq:SL_definition}
  S(\LL)\defeq \setgiven{U\in \Bis(\Gr)}{\text{ the restricted bundle }\LL|_U\text{ can be trivialised}}
\end{equation}
is a wide unital inverse subsemigroup of \(\Bis(\Gr)\) (see, for
instance,
\cite{Bardadyn-Kwasniewski-McKee:Banach_algebras}*{Lemma~4.2}).  In
particular, we always assume that \emph{\(\LL|_X = X \times \C\) is
  trivial}.  When \(\Sigma\defeq\setgiven{z\in \LL}{\abs{z}=1}\) is
given the topology and multiplication from~\(\LL\), it becomes a
topological groupoid.  There is a natural exact sequence
\(X\times \T\into \Sigma \onto \Gr\), turning~\(\Sigma\) into a
central \(\T\)\nb-extension of~\(\Gr\).  Every central
\(\T\)\nb-extension of~\(\Gr\) arises this way (see
\cite{Kumjian:Fell_bundles}*{Example~2.5.iv}).  This gives an
equivalence between Fell line bundles and twists as defined by
Kumjian and Renault in \cites{Kumjian:Diagonals,
  Renault:Cartan.Subalgebras}.  Fell line bundles which are trivial
as bundles are equivalent to (normalised) continuous
\(2\)\nb-cocycles \(\sigma\colon G^2 \to \T\) (see, for instance,
\cite{Bardadyn-Kwasniewski-McKee:Banach_algebras}*{Example~4.3}).

Fix a twisted groupoid \((\Gr, \LL)\).  For \(U \subseteq \Gr\) let
$\Contc(U , \LL)$ be the space of continuous compactly supported
sections of~\(\LL|_U\).  In general, \(\Gr\) need not be Hausdorff.
In this generality, the space of \emph{quasi-continuous compactly
  supported functions} is defined as
\[
  \Sect(\Gr , \LL)\defeq \operatorname{span}
  \setgiven{ f \in \Contc(U , \LL) }{ U \in \Bis(\Gr)},
\]
where a section in \(\Contc(U , \LL)\) is extended to a section
from~\(\Gr\) to~\(\LL\) that vanishes off~\(U\).  When~\(\Gr\) is
Hausdorff, then \(\Sect(\Gr , \LL)=\Contc(\Gr , \LL)\).  In general,
\(\Sect(\Gr , \LL)\) is a \Star{}algebra with operations:
\[
  (f * g)(\gamma) \defeq \sum_{\rg(\eta)
    = \rg(\gamma)} f(\eta) \cdot g(\eta^{-1} \gamma),\qquad
  f^*(\gamma)\defeq f(\gamma^{-1})^* ,\quad
  f,g \in \Sect(\Gr , \LL) ,\ \gamma \in \Gr ,
\]
where $\cdot$ and the final $*$ indicate the product and involution
from~$\LL$.  The \emph{universal \(\Cst\)\nb-algebra}
\(\Cst(\Gr,\LL)\) of \((\Gr , \LL)\) is defined as the maximal
\(\Cst\)\nb-completion of \(\Sect(\Gr , \LL)\).  It contains
\(\Cont_0(X)\) as a \(\Cst\)\nb-subalgebra.

Let \(\Borel(X)\) denote the \(\Cst\)\nb-algebra of bounded
Borel functions on~\(X\).
There is a unique completely contractive positive map \(E\colon
\Cst(\Gr,\LL)\to \Borel(X)\) with \(E(f)=f|_X\) for all \(f\in \Sect(\Gr , \LL)\).
Let \(\Null_E\defeq \setgiven{a\in \Cst(\Gr,\LL)}{E(a^*a)=0}\).
This is an ideal, and the quotient
\(\Cst_\red(\Gr,\LL)\defeq \Cst(\Gr,\LL)/\Null_E\) is the
\emph{reduced \(\Cst\)\nb-algebra} of \((\Gr , \LL)\).
The convolution formula  \(\Lambda(f) \xi \defeq f * \xi\) for \(f \in
\Sect(\Gr , \LL)\), \(\xi \in \ell^{2}(\Gr,\LL)\), defines an
injective \Star{}homomorphism \(\Lambda \colon \Sect(\Gr , \LL) \to
\Bound(\ell^{2}(\Gr,\LL))\), called the \emph{regular representation},
that extends to a faithful representation of \(\Cst_\red(\Gr,\LL)\).

Let
\(\Meager(X) \defeq \setgiven{f\in \Borel(X)}{ f \text{ vanishes on a
    comeagre subset}}\).  This is an ideal in~\(\Borel(X)\), and the
quotient \(\Cst\)\nb-algebra \(\Borel(X)/ \Meager(X)\) coincides both
with the injective hull and with the local multiplier algebra of
\(\Cont_0(X)\), that is,
\[
  \hull(\Cont_0(X))\cong \Locmult(\Cont_0(X))
  \cong \Borel(X)/\Meager(X)
\]
(see \cite{Kwasniewski-Meyer:Essential}*{Subsection~4.4} and the
references therein).  It naturally contains \(\Cont_0(X)\) as a
subalgebra.  Composing \(E\colon \Cst(\Gr,\LL)\to \Borel(X)\) with the
quotient map \(\Borel(X)\onto \Borel(X)/\Meager(X)\) gives a
pseudo-expectation
\(\EL\colon \Cst(\Gr,\LL)\to \Borel(X)/\Meager(X) \cong \hull(\Cont_0(X))\) for the
\(\Cst\)\nb-inclusion \(\Cont_0(X)\subseteq \Cst(\Gr,\LL)\).  Let
\[
  \Null_{\EL}\defeq \setgiven{a\in \Cst(\Gr,\LL)}{\EL(a^*a)=0}.
\]
The quotient \(\Cst_\ess(\Gr,\LL)\defeq \Cst(\Gr,\LL)/\Null_{\EL}\) is
called the \emph{essential \(\Cst\)\nb-algebra} of \((\Gr , \LL)\).
Then \(\Cst_\ess(\Gr,\LL)\) can be viewed as a quotient of
\(\Cst_\red(\Gr,\LL)\).
Let \(\Gr_{\text{H}}\) be the set of elements $\gamma \in \Gr$ that are Hausdorff in the sense that
for any \(\eta \in \Gr \setminus \{ \gamma \}\) the elements \(\gamma\) and \(\eta\) can be separated by disjoint open sets in \(\Gr\). By
\cite{Bardadyn-Kwasniewski-McKee:Banach_algebras_simple_purely_infinite}*{Lemma~4.3
  and Proposition~4.15}, \(\ell^2(\Gr_{\text{H}},\LL)\) is an invariant subspace for the regular representation \(\Lambda\) and the corresponding
subrepresentation \(\Lambda_{\ess} \colon \Sect(\Gr , \LL) \to \Bound(\ell^{2}(\Gr_{\text{H}},\LL))\)
extends to a representation of \(\Cst_\ess(\Gr,\LL)\) which is faithful when ~\(\Gr\) has
a countable cover by open bisections. \label{page:essential_rep} Thus it is natural to call \(\Lambda_{\ess}\) the \emph{essential representation}.

While \(\Sect(\Gr , \LL)\) embeds into \(\Cst(\Gr,\LL)\) and
\(\Cst_\red(\Gr,\LL)\) as a dense \Star{}subalgebra, the canonical
map from \(\Sect(\Gr , \LL)\) to \(\Cst_\ess(\Gr,\LL)\) need not
be injective.  Of course, its range is still dense.  In addition,
the spaces \(\Cont_0(U,\LL)\), \(U\in\Bis(\Gr)\), embed into
\(\Cst_\ess(\Gr,\LL)\) and form an inverse semigroup grading of
\(\Cst_\ess(\Gr,\LL)\).  In particular, \(\Cont_0(X)\) is embedded
as a \(\Cst\)\nb-subalgebra of \(\Cst_\ess(\Gr,\LL)\).

A \(\Cst\)\nb-algebra~\(D\) is called an \emph{exotic
  \(\Cst\)\nb-algebra} of \((\Gr , \LL)\) if there are surjective
\Star{}homomorphisms \(\Cst(\Gr,\LL)\onto D \onto \Cst_\ess(\Gr,\LL)\)
whose composite is the quotient map
\(\Cst(\Gr,\LL) \onto \Cst_\ess(\Gr,\LL)\) (see
\cite{Kwasniewski-Meyer:Essential}*{Subsection~4.2} or
\cite{Bardadyn-Kwasniewski-McKee:Banach_algebras_simple_purely_infinite}*{Subsection~4.1}).

\begin{proposition}
  \label{prop:essential_reduced_coincide}
  Let \(\E\colon \Cst_\red(\Gr,\LL)\to \Borel(X)\) be the canonical
  generalised expectation given by restriction of sections to~\(X\).
  The canonical quotient map is an isomorphism
  \(\Cst_\red(\Gr,\LL) \cong \Cst_\ess(\Gr,\LL)\) if and only if
  \(\setgiven*{x\in X}{\E(f)(x)\neq 0}\) is not meagre for every
  \(f\in \Cst_\red(\Gr,\Sigma)^+\setminus\{0\}\).  This always holds
  when~\(\Gr\) is Hausdorff or when~\(\Gr\) is ample and every
  compact open subset in~\(\Gr\) is equal to the interior of its
  closure \textup{(}such open subsets are also called
  regular\textup{)}.
\end{proposition}

\begin{proof}
  This is a part of
  \cite{Kwasniewski-Meyer:Essential}*{Proposition~7.4.7}, except for
  the statement about ample~\(\Gr\).  In the untwisted case, the
  latter is proven in
  \cite{Clark-Exel-Pardo-Sims-Starling:Simplicity_non-Hausdorff}*{Lemma~4.11}.
  This is generalised to the twisted case in
  \cite{Bardadyn-Kwasniewski-McKee:Banach_algebras_simple_purely_infinite}*{Lemma~4.7,
    see also Remark~4.13}.
\end{proof}

\begin{lemma}
  \label{lem:hereditary_subalgebras_essential}
  For any \(f_U\in \Cont_0(X)^+\) with open support \(U\subseteq X\)
  the hereditary \(\Cst\)\nb-subalgebra \(f_U\Cst_\ess(\Gr,\LL)f_U\)
  of \(\Cst_\ess(\Gr,\LL)\) generated by~\(f_U\) is canonically
  isomorphic to the essential algebra for the restricted pair
  \((\Gr_U,\LL_U)\) where
  \(\Gr_U\defeq \rg^{-1}(U) \cap \s^{-1}(U) \subseteq \Gr\) and
  \(\LL_U\defeq \LL|_{\Gr_U}\).

  Moreover, \(f_U\Cst_\ess(\Gr,\LL)f_U\) is an ideal if and only
  if~\(U\) is \(\Gr\)\nb-invariant.  In particular, if
  \(\Cst_\ess(\Gr,\LL)\) is simple, then~\(\Gr\) is minimal, that
  is, there are no nontrivial open \(\Gr\)\nb-invariant subsets
  in~\(X\).
\end{lemma}

\begin{proof}
  We may treat \(\Sect(\Gr_U,\LL_U)\) as a \Star{}subalgebra of
  \(\Sect(\Gr , \LL)\) in an obvious way.  This inclusion gives a
  \Star{}homomorphism \(\Sect(\Gr_U,\LL_U)\to \Cst(\Gr,\LL)\).  Its
  range is dense in the \(\Cst\)\nb-subalgebra
  \(f_U\Cst(\Gr,\LL)f_U\).  By the universal property of full groupoid \(\Cst\)-algebras, this
  homomorphism extends to a surjective \Star{}homomorphism
  \(\Psi\colon \Cst(\Gr_U,\LL_U)\onto f_U\Cst(\Gr,\LL)f_U\).  Let
  \(\EL\) and~\(\EL_U\) denote the canonical pseudo-expectations on
  \(\Cst(\Gr,\LL)\) and \(\Cst(\Gr_U,\LL_U)\), respectively.
  Identifying \(\Borel(U)/\Meager(U)\) with
  \(f_U \bigl(\Borel(X)/\Meager(X)\bigr) f_U\), we get
  \(\EL\circ \Psi =\EL_U\).  This implies that~\(\Psi\) is faithful,
  as \(\EL\) and~\(\EL_U\) are.  The second part of the assertion is
  now straightforward.
\end{proof}

\begin{remark}
  \label{rem:hereditary_subalgebras_reduced}
  Lemma~\ref{lem:hereditary_subalgebras_essential} holds with~\(\Cst_\ess(\Gr,\LL)\)
  replaced by~\(\Cst_\red(\Gr,\LL)\), by
  virtually the same proof.
\end{remark}

By Lemma~\ref{lem:hereditary_subalgebras_essential}, for every open
\(\Gr\)\nb-invariant subset~\(U\), extending sections from~\(U\)
to~\(\Gr\) by zero outside~\(U\) gives an embedding
\(\Cst_\ess(\Gr_U,\LL_U) \into \Cst_\ess(\Gr,\LL)\) whose range is
an ideal in \(\Cst_\ess(\Gr,\LL)\).

\begin{definition}
  The twisted groupoid \((\Gr,\LL)\) is \emph{essentially exact} if
  for any open invariant subset \(U\subseteq X\), the restriction of
  sections gives a well-defined, surjective \Star{}homomorphism
  \(\Cst_\ess(\Gr,\LL) \onto \Cst_\ess(\Gr_{X\setminus U},
  \LL_{X\setminus U})\) such that the sequence of essential twisted
  groupoid \(\Cst\)\nb-algebras
  \[
    \Cst_\ess(\Gr_U,\LL_U)
    \into \Cst_\ess(\Gr,\LL)
    \onto \Cst_\ess(\Gr_{X\setminus U},\LL_{X\setminus U})
  \]
  is exact
  (\cite{Kwasniewski-Meyer:Pure_infiniteness}*{Definition~4.23,
    Example~4.25}).  The sequence as above always exists for reduced
  \(\Cst\)\nb-algebras, and if it is exact, then we call
  \((\Gr,\LL)\) \emph{exact} (see
  \cite{Kwasniewski-Meyer:Pure_infiniteness}*{Definition~4.17,
    Example~4.18}).  If the twist is trivial, the name \emph{inner
    exact} for such groupoids is often used; it was introduced
  in~\cite{AnantharamanDelaroch:Weak_containment}.
\end{definition}

When~\(\Gr\) is Hausdorff, then exactness and essential exactness
coincide.

\begin{definition}
  The groupoid~\(\Gr\) is called \emph{effective} if the interior of
  the isotropy bundle
  \(\operatorname{Iso}(\Gr)\defeq\setgiven{\gamma\in\Gr}{\rg(\gamma)=\s(\gamma)}\)
  is~\(X\).  It is \emph{topologically free} if the interior of
  \(\operatorname{Iso}(\Gr)\setminus X\) is empty.  It is
  \emph{residually topologically free} if each restriction~\(\Gr_Y\)
  to a closed \(\Gr\)\nb-invariant subset \(Y\subseteq X\) is
  topologically free.
\end{definition}

\begin{proposition}
  \label{prop:ideals_twisted_groupoids}
  Let \((\Gr,\LL)\) be a twisted groupoid where~\(\Gr\) is \'etale,
  residually topologically free and essentially exact, with a
  locally compact Hausdorff unit space \(X\defeq \Gr^0\).  Then all
  ideals in \(\Cst_\ess(\Gr,\LL)\) are of the form
  \(\Cst_\ess(\Gr_U,\LL_U)\) for an open \(\Gr\)\nb-invariant subset
  \(U\subseteq X\).  If, in addition, \(X\) is metrisable, then the
  primitive ideal space of \(\Cst_\ess(\Gr,\LL)\) is homeomorphic to
  the quasi-orbit space \(X/{\sim}\), which is defined as the orbit
  space of the equivalence relation with \(x \sim y\) if and only if
  \(\cl{\Gr x} = \cl{\Gr y}\).
\end{proposition}

\begin{proof}
  This is a special case of
  \cite{Kwasniewski-Meyer:Pure_infiniteness}*{Corollary~5.10}.
\end{proof}

As the construction of the Cuntz semigroup involves stabilisation, we
discuss the stabilisation of twisted groupoid algebras.  We may tensor
\((\Gr,\LL)\) with the full equivalence relation \(\RR= \N\times \N\)
with the trivial twist to get a natural twisted groupoid
\((\Gr\times \RR, \LL\otimes \C)\).  Namely, \(\Gr\times \RR\) is the
standard product of groupoids, with the unit space
\(X\times \setgiven{(n,n)}{n\in \N}\cong X\times \N\).  We equip it
with the Fell line bundle \(\LL\otimes \C\), where
\((\LL\otimes \C)_{(\gamma,\nu)}\defeq \LL_{\gamma}\), for
\((\gamma,\nu)\in\Gr\times \RR\), and the multiplication, involution
and modulus is inherited from~\(\LL\).  The topology on
\(\LL\otimes \C\) is determined by the requirement that for any
\(a\in \Contc(U,\LL)\), \(U\in \Bis(\Gr)\), and \(b\in \Contc(V)\),
\(V\in \Bis(\RR)\), the section of \(\LL\otimes \C\) given by
\(a \odot b (\gamma,\nu)\defeq b(\nu) a(\gamma)\) is continuous, and
so \(a \odot b \in \Contc(U\times V ,\LL\otimes \C)\).

\begin{lemma}
  \label{lem:stabilisation_groupoids}
  If \((\Gr,\LL)\) is any twisted \'etale groupoid, then there are
  natural \Star{}\alb{}isomorphisms
  \(\Cst(\Gr\times \RR, \LL\otimes \C)\cong \Cst(\Gr,\LL)\otimes
  \Comp\) and
  \(\Cst_\red(\Gr\times \RR, \LL\otimes \C)\cong
  \Cst_\red(\Gr,\LL)\otimes \Comp\).
  If~\(\Gr\) has a countable cover by bisections, then
  \(\Cst_\ess(\Gr\times \RR, \LL\otimes \C)\cong
  \Cst_\ess(\Gr,\LL)\otimes \Comp\).
\end{lemma}

\begin{proof}
  Obviously, \(\Cst( \RR)=\Cst_\red( \RR)=\Cst_\ess(\RR)\cong \Comp\).
  Let \(\Lambda^{\Gr\times\RR}\) be the regular representation of
  \(\Cst_\red(\Gr\times \RR, \LL\otimes \C)\) on
  \(\ell^2(\Gr\times \RR, \LL\otimes \C)\cong \ell^2(\Gr,\LL)\otimes
  \ell^{2}(\RR)\).  Let \(\Lambda^{\Gr}\) and~\(\Lambda^{\RR}\) be the
  regular representations of \(\Cst_\red(\Gr, \LL)\) and
  \(\Cst_\red(\RR)\), respectively.  Then
  \(\Lambda^{\Gr}\otimes\Lambda^{\RR}\) is a faithful representation
  of \(\Cst_\red(\Gr,\LL)\otimes \Comp\) on
  \(\ell^2(\Gr,\LL)\otimes \ell^{2}(\RR)\) (see, for instance,
  \cite{Blackadar:Operator_algebras}*{II.9.1.3}).  Clearly,
  \(\Lambda^{\Gr}(a)\otimes\Lambda^{\RR}(b)=\Lambda^{\Gr\times\RR}(a
  \odot b)\) for any \(a\in \Contc(U,\LL)\), \(U\in \Bis(\Gr)\), and
  \(b\in \Contc(V)\), \(V\in \Bis(\RR)\).  Thus the representations
  \(\Lambda^{\Gr}\otimes\Lambda^{\RR}\) and \(\Lambda^{\Gr\times\RR}\)
  have the same ranges, and since they are both faithful,
  \(\Cst_\red(\Gr\times \RR, \LL\otimes \C)\cong
  \Cst_\red(\Gr,\LL)\otimes \Comp\).

  The above implies also that we have an injective \Star{}homomorphism
  from the algebraic tensor product
  \(\Sect(\Gr , \LL) \otimes \Sect(\RR)\) into
  \(\Sect(\Gr\times \RR, \LL\otimes\C)\) that sends a simple tensor
  \(a \otimes b\) to the section \(a \odot b\) of \(\LL\otimes\C\)
  given by \(a \odot b (\gamma,\nu)\defeq b(\nu)a(\gamma)\) for
  \((\gamma,\nu)\in \Gr\times \RR\).  In fact, since~\(\RR\) is
  discrete, this \Star{}homomorphism is surjective.  Indeed, every
  bisection \(U\subseteq \Gr\times\RR\) is a disjoint union of
  bisections \(U_{\nu}\times \{\nu\}\), \(\nu\in \RR\), where
  \(U_\nu\defeq \setgiven{\gamma }{(\gamma,\nu)\in U}\) is a bisection
  of~\(\Gr\).
  Hence the support of any section \(f\in \Contc( U, \LL\otimes\C)\)
  can be covered by some \(U_{\nu_i}\times \{\nu_i\}\) for some
  \(\nu_1,\dots \nu_{n}\in \RR\).
  Using a partition of unity subordinate to this cover, we may
  write~\(f\) as \(\sum_{i=1}^n f_i \odot 1_{\nu}\), where \(f_i\in
  \Contc(U_{\nu_i},\LL)\) for \(i=1,\dotsc, n\).
  Thus~\(f\) is in the range of the homomorphism and surjectivity
  follows.
  Therefore, we have a \Star{}isomorphism
  \[
    \Sect(\Gr\times  \RR , \LL\otimes\C)
    \cong \Sect(\Gr , \LL) \otimes \Sect(\RR).
  \]
  It extends to a \Star{}homomorphism
  \(\Psi\colon\Cst(\Gr\times \RR,\LL\otimes\C)\to \Cst(\Gr,\LL)\otimes
  \Comp\) because \(\Cst(\Gr\times \RR,\LL\otimes\C)\) is the
  completion of~$\Sect(\Gr\times \RR , \LL\otimes\C)$ in the maximal
  \(\Cst\)\nb-norm.  It also implies that
  \(\Cst(\Gr\times \RR,\LL\otimes\C)\) contains a completion of
  \(\Sect(\Gr , \LL)\) as a \(\Cst\)\nb-subalgebra, and in the
  commutant of this subalgebra there is a copy of
  \(\Comp\cong \overline{\Sect(\RR)}\) (see, for instance,
  \cite{Blackadar:Operator_algebras}*{II.9.2.1}).
  Thus there is a \Star{}homomorphism
  \[
    \Cst(\Gr,\LL)\otimes \Comp\to \Cst(\Gr\times \RR,\LL\otimes\C)
  \]
  which is inverse to~\(\Psi\).
  Accordingly, \(\Cst(\Gr\times \RR, \LL\otimes \C)\cong
  \Cst(\Gr,\LL)\otimes \Comp\).

  If~\(\Gr\) has a countable cover by bisections, then the above
  argument for reduced algebras works for essential algebras by
  replacing regular representations with their subrepresentations that
  we called essential on page~\pageref{page:essential_rep}.
\end{proof}

\begin{remark} Statements similar to those in Lemma \ref{lem:stabilisation_groupoids}
  hold for tensor products with matrix algebras \(M_n\) upon replacing \(\RR\) by a finite pair groupoid.
\end{remark}

Recall that a \emph{trace} on a \(\Cst\)\nb-algebra~\(D\) is an
additive function \(\tau\colon A^+\to[0,\infty]\) such that
\(\tau(\lambda a) = \lambda \tau(a)\) for all \(\lambda\in[0,\infty)\)
and satisfies \(\tau(a^*a)=\tau(aa^*)\)
for all \(a\in A\) (see, for instance,
\cite{Blackadar:Operator_algebras}*{II.6.7--8}).  We call~\(\tau\)
\emph{faithful} if \(\tau(a)=0\) implies \(a=0\) for all \(a\in A^+\).
If~\(\tau\) is finite, it is automatically continuous and extends to a
bounded linear functional on~\(A\) with the trace property
\(\tau(ab)=\tau(ba)\) for all \(a,b\in A\), and when normalised it is
a \emph{tracial state}.  A lower semicontinuous trace~\(\tau\) is
\emph{semifinite} if
\(\tau(a)=\sup {}\setgiven{\tau(b)}{\tau(b)<\infty \text{ and }b\le
  a}\) for all \(a\in A^+\).  There is a natural bijection between
lower semicontinuous traces~\(\tau\) on \(\Cont_0(X)\) and regular
Borel measures~\(\mu\) on~\(X\), defined by
\(\tau(f)=\int_X f\,\diff\mu\) for \(f\in \Cont_0(X)^+\).
A Borel
measure \(\mu\colon \B(X)\to [0,\infty]\) is said to be
\emph{regular} if it is both inner and
outer regular. Here, \emph{inner regular} means that
\(\mu(V)=\sup {}\setgiven{\mu(K)}{K \subseteq V,\ K \text{ compact}}\) for every open~\(V\)
and \emph{outer regular} means that \(\mu(E)=\inf {}\setgiven{\mu(V)}{E\subseteq V \text{ is open}}\) for
every \(E \in \B(X)\).
We will now describe how the bijection occurs. If~\(\tau\) is lower semicontinuous, it is determined by its
restriction to \(\Contc(X_{\tau})^+\subseteq \Contc(X)^+\) for the open
subset
\[
  X_{\tau}\defeq \setgiven{x\in X }{\tau(f)<\infty \text{ for some }f\in
    \Contc(X)^+\text{ with } f(x)>0}
\] in~\(X\).  The restriction
\(\tau|_{\Contc(X_{\tau})^+}\) is finite, and the classical Riesz's Theorem
(\cite{Rudin}*{Theorem~2.14}) associates it with a unique Radon
measure on~\(X_{\tau}\) such that \(\tau(f)=\int_{X_{\tau}} f\,\diff\mu\) for \(f\in \Cont_c(X_{\tau})^+\).
By defining \(\mu(E)=\infty\) for every ~\(E\in \B(X)\) that is not
contained in~\(X_{\tau}\) one gets the desired measure \(\mu\) on \(\B(X)\) associated
to~\(\tau\).

Any lower semicontinuous trace~\(\tau\) on \(\Cont_0(X)\) has a unique
extension to a \emph{normal} (preserving suprema of directed bounded
nets) trace~\(\overline{\tau}\) on the algebra of bounded Borel
functions~\(\Borel(X)\) by \(\overline{\tau}(f)=\int_X f\,\diff\mu\)
for \(f\in \Borel(X)^+\) and the corresponding measure \(\mu\).
Unfortunately, there is no canonical way of extending states from
\(\Cont_0(X)\) to \(\Borel(X)/\Meager(X)\).  Since
\(\Borel(X)/\Meager(X)\) is a monotone completion of \(\Cont_0(X)\),
one could expect that normal extensions exist.  However, except in
trivial cases, \(\Borel(X)/\Meager(X)\) is \emph{wild}, which means
that it does not admit any nonzero normal states (see
\cite{Saito_Wright}*{Theorem~4.2.17}).  Therefore, it seems that there
is no analogue of the next proposition for essential groupoid
\(\Cst\)\nb-algebras.

\begin{definition}
  A Borel measure~\(\mu\) on~\(X\) is called
  \emph{\(\Gr\)\nb-invariant} if \(\mu(\s(U))=\mu(\rg(U))\) for
  every open bisection \(U\in \Bis(\Gr)\) (see
  \cite{Renault:Groupoid_Cstar}*{I.3.12} and
  \cite{Li-Renault:Cartan_subalgebras}).
\end{definition}

\begin{lemma}\label{lem:G-invariant_measure}
  Let~\(\mu\) be a regular Borel measure on~\(X\).
  Fix a family
  \(\B\subseteq \Bis(\Gr)\) that covers~\(\Gr\).
  The following conditions are equivalent:
  \begin{enumerate}
  \item \label{enu:G-invariant_measure1}%
    the measure~\(\mu\) is \(\Gr\)\nb-invariant;
  \item \label{enu:G-invariant_measure2}%
    \(\int_{\rg(U)} f \,\diff \mu = \int_{\s(U)} f\circ \theta_U
    \,\diff\mu\) for every \(f\in\Contc(\rg(U))^+\) and \(U\in \B\),
    where \(\theta_U\colon \s(U)\to \rg(U)\) is the homeomorphism
    given by \(\theta_U(\s(\gamma))=\rg(\gamma)\), \(\gamma\in U\);
  \item\label{enu:G-invariant_measure3} for any Borel function
    \(h\colon \Gr \to [0,\infty)\) we have
    \begin{equation}\label{eq:invariant_measure3}
      \int \sum_{\s(\gamma) = x} h(\gamma) \,\diff\mu(x)
      = \int \sum_{\rg(\gamma) = x} h(\gamma) \,\diff\mu(x).
    \end{equation}
  \end{enumerate}

\end{lemma}
\begin{proof}
  \ref{enu:G-invariant_measure1}\(\Rightarrow\)\ref{enu:G-invariant_measure2},\ref{enu:G-invariant_measure3}:
  Assume first that~\(\mu\) is \(\Gr\)\nb-invariant.
  Let \(U\in \Bis(\Gr)\).
  If \(V\subseteq \s(U)\) is open, then
  \(\mu(V)=\mu(\s(VU))=\mu(\rg(VU))=\mu(\theta_{U}(V))\).
  Since~\(\mu\) is regular, it follows that
  \(\mu(B)=\mu(\theta_{U}(B))\) for any Borel \(B\subseteq \s(U)\).
  Thus the standard change of variables gives
  \(\int_{\rg(U)} f \,\diff \mu = \int_{\s(U)} f\circ \theta_U
  \,\diff\mu\) for any positive Borel function
  \(f\in \Borel(r(U))^+\).
  This readily gives \ref{enu:G-invariant_measure2}.
  This also means that \eqref{eq:invariant_measure3} holds for all
  Borel functions~\(h\) vanishing outside~\(U\), as putting \(f\defeq
  h\circ \rg|_{U}^{-1}\in\Borel(r(U))^+\subseteq \Borel(X)\), we then
  get \(\sum_{\rg(\gamma) = x} h(\gamma)=f(x)\) and
  \(\sum_{\rg(\gamma) = x} h(\gamma)=f(\theta_{U}(x))\) for all \(x\in
  X\).
  This also obviously holds when~\(U\) is just a Borel
  subset of an open bisection.
  By choosing a decomposition \((U_i)_{i\in I}\) of~\(\Gr\) into
  pairwise disjoint Borel bisections for any Borel \(h\colon \Gr \to
  [0,\infty)\) we get \(h=\sum_{i\in I} 1_{U_i} h
  =\sup\setgiven{\sum_{i\in F} 1_{U_i}}{F\subseteq I\text{ finite}}\).
  Thus the previous step plus linearity (and monotone convergence)
  give~\eqref{eq:invariant_measure3} in general.

  \ref{enu:G-invariant_measure3}\(\Rightarrow\)\ref{enu:G-invariant_measure2}:
  For a given \(f\in\Contc(\rg(U))^+\) apply
  \eqref{eq:invariant_measure3} to~\(h\) given by
  \(h(\gamma)=f(r(\gamma))\) for \(\gamma\in U\) and zero elsewhere.

  \ref{enu:G-invariant_measure2}\(\Rightarrow\)\ref{enu:G-invariant_measure1}:
  We claim that
  \(\int_{\rg(U)} f \,\diff \mu = \int_{\s(U)} f\circ \theta_U
  \,\diff\mu\) for every \(f\in\Contc(\rg(U))^+\) and every
  \(U\in \Bis(\Gr)\) (not necessarily in \(\B\)).  Indeed, denote
  by~\(K\) the closed support of \(f\).  Then
  \(r|_{U}^{-1}(K)\subseteq U\) is compact and hence can be covered by
  a finite family \((V_i)_{i=1}^{n}\) of sets from the open
  cover~\(\B\).  Let \((g_i)_{i=1}^n\) be a partition of unity
  on~\(K\) subordinate to the cover \((r(V_i))_{i=1}^{n}\).  Then
  \(g_if\in \Contc(\rg(V_i))^+\cap \Contc(\rg(U))^+\) and
  \((g_if)\circ \theta_{V_i}=(g_if)\circ \theta_{U} \in
  \Contc(\s(V_i))^+\cap \Contc(\rg(U))^+\) for each~\(i\).
  Then~\ref{enu:G-invariant_measure2} implies
  \(\int_{\rg(U)} g_i f \,\diff\mu =\int_{\rg(V_i)} g_i f \,\diff \mu
  = \int_{\s(V_i)} (g_{i}f)\circ \theta_{V_i} \,\diff\mu=\int_{\s(U)}
  (g_{i}f)\circ \theta_{U} \,\diff\mu\).  Thus
  \[
    \int_{\rg(U)} f \,\diff \mu=\sum_{i=1}^n \int_{\rg(U)} g_i f
    \,\diff \mu
    = \sum_{i=1}^n \int_{\s(U)} (g_{i}f)\circ \theta_{U} \,\diff\mu
    =\int_{\s(U)} f\circ \theta_{U} \,\diff\mu.
  \]
  In other words, \(\tau(f)=\tau(f\circ \theta_{U})\).  This implies
  that \(\mu(\rg(U))=\mu(\s(U))\) since
  \begin{equation*}
    \mu(V)=\sup\setgiven{\tau(f)}{f\in \Contc(V)^+, \|f\|=1}.
  \end{equation*}
  for every open set \(V\subseteq X\).
\end{proof}

\begin{proposition}
  \label{prop:states_and_traces}
  Let \(\E\colon \Cst_\red(\Gr,\LL)\to \Borel(X)\) be the canonical
  generalised expectation on the reduced \(\Cst\)\nb-algebra of
  \((\Gr,\LL)\).  There is a natural bijection between lower
  semicontinuous traces~\(\tau\) on \(\Cst_\red(\Gr,\LL)\) that factor
  through~\(\E\), in the sense that \(\tau=\overline{\tau}\circ \E\),
  where \(\overline{\tau}\colon \Borel(X)^+\to [0,\infty]\) is the
  normal extension of \(\tau|_{\Cont_0(X)^+}\), and
  \(\Gr\)\nb-invariant regular Borel measures~\(\mu\) on~\(X\).

  For
  the corresponding objects, \(\tau\) is
  \textup{(}semi\textup{)}finite if and only if~\(\mu\) is
  \textup{(}locally\textup{)} finite; and if
  \(\Cst_\red(\Gr,\LL)=\Cst_\ess(\Gr,\LL)\), then~\(\tau\) is faithful
  if and only if~\(\mu\) has full support.
\end{proposition}

\begin{proof}
  Let~\(\tau\) be a lower semicontinuous trace
  on~\(\Cst_\red(\Gr,\LL)\) that factors through~\(\E\).  The
  restriction
  \(\tau|_{\Cont_0(X)^+} \colon \Cont_0(X)^+\to [0,\infty]\) is a
  positive lower semicontinuous trace and corresponds to a regular
  measure~\(\mu\) on~\(X\) by Riesz's Theorem.
  We must prove that~\(\mu\) is invariant.  Let \(U\in S(\LL)\) be a
  bisection such that the bundle~\(\LL|_U\) is trivial,
  see~\eqref{eq:SL_definition}.  So there is a continuous unitary
  section \(c\colon U\to \LL\) and
  \(\Contc(U)\ni g\mapsto g\cdot c\in \Contc(U,\LL)\subseteq
  \Cst_\red(\Gr,\LL)\) is an isometric isomorphism.  Take any
  \(f\in \Contc(\rg(U))^+\) and define
  \(g\in \Contc(U,\LL)\subseteq \Cst_\red(\Gr,\LL)\) by putting
  \(g(\gamma)\defeq \sqrt{f(\rg(\gamma))}\cdot c(\gamma)\).  Then
  \(g * g^*=f\) and \(g^* * g=f\circ \theta_U\).  Hence
  \[
    \int_{\rg(U)} f \,\diff\mu
    = \tau(f)
    = \tau(g * g^*)
    = \tau( g^* * g)
    = \tau(f\circ \theta_U)
    = \int_{\s(U)}f\circ \theta_U\,\diff\mu.
  \]
  As~\(S(\LL)\) covers~\(\Gr\), this implies that~\(\mu\) is
  \(\Gr\)\nb-invariant by its characterisation in
  Lemma~\ref{lem:G-invariant_measure}.\ref{enu:G-invariant_measure2}.

  Conversely, let~\(\mu\) be a \(\Gr\)\nb-invariant regular Borel
  measure on~\(X\).
  Let \(\tau\colon \Cont_0(X)^+ \to [0,\infty]\) be the corresponding
  lower semicontinuous weight and let \(\overline{\tau}\colon
  \Borel(X)^+ \to [0,\infty]\) be its normal extension.
  Then the composite \(\overline{\tau}\circ \E|_{\Cst_\red(\Gr,\LL)^+
  }\) is a lower semicontinuous extension of~\(\tau\) to
  \(\Cst_\red(\Gr,\LL)\), which we again denote by~\(\tau\).
  We need to show that~\(\tau\) is a trace.
  Consider the Banach space \(\Borel(\Gr,\LL)\) of bounded
  Borel sections of~\(\LL\), equipped with the supremum norm.
  The inclusion \(\Sect(\Gr , \LL) \subseteq \Borel(\Gr,\LL)\) extends
  to a contractive injective map \(j\colon \Cst_\red(\Gr,\LL)\to
  \Borel(\Gr,\LL)\), see
  \cite{Kwasniewski-Meyer:Essential}*{Proposition~7.10}, which turns
  the product and adjoint in \(\Cst_\red(\Gr,\LL)\) to the convolution
  and involution, see
  \cite{Bardadyn-Kwasniewski-McKee:Banach_algebras_simple_purely_infinite}*{Proposition~3.16}.
  Therefore, if \(f\in \Cst_\red(\Gr,\LL)\), then
  \[
    \E(f^* * f)(x) =
    \sum_{\s(\gamma)=x} j(f)(\gamma)^* j(f)(\gamma)
    = \sum_{\s(\gamma)=x} \norm{j(f)(\gamma)}^2,
  \]
  where \(\norm{j(f)(\gamma)}^2\) uses the norm on the fibre
  \(\LL_\gamma \cong\C\) (see also
  \cite{Kwasniewski-Meyer:Essential}*{Equation~(7.1)}).
  Similarly,
  \[
    \E(f * f^*)(x) =
    \sum_{\s(\gamma)=x} \norm{j(f)(\gamma^{-1})}^2
    = \sum_{\rg(\gamma)=x} \norm{j(f)(\gamma)}^2.
  \]
  The \(\mu\)\nb-integrals of these two functions are equal
  because of the characterisation in
  Lemma~\ref{lem:G-invariant_measure}.\ref{enu:G-invariant_measure3}.
  Thus \(\tau(f^* * f) = \overline{\tau}(\E(f * f^*)) =
  \overline{\tau}(\E(f * f^* ))= \tau(f^* * f)\) for all \(f\in
  \Cst_\red(\Gr,\LL)\).

  This proves the first part of the assertion.
  Now consider the corresponding \(\tau\) and~\(\mu\).
  It is immediate that~\(\tau\) is finite if and only if~\(\mu\) is
  finite, and that if~\(\tau\) is semifinite, then~\(\mu\) has to be
  locally finite.
  Conversely, if~\(\mu\) is locally finite, then taking any
  approximate unit~\((e_i)_i\) in \(\Contc(X)^+\) and any
  \(a\in\Cst_\red(\Gr,\LL)^+\), we compute \(\sqrt{a}e_i^2\sqrt{a}\le
  a\) and hence
  \[
    \tau(\sqrt{a}e_i^2\sqrt{a})
    = \tau(e_i a e_i )
    = \int_X \E(e_i a e_i)\,\diff\mu
    = \int_X e_i \E( a )e_i\,\diff\mu<\infty,
  \]
  where we used that~\(\tau\) is a trace, \(\E\) is a bimodule map
  and~\(\mu\) is a Radon measure.
  Since~\(\tau\) is lower semicontinuous,
  \(\tau(\sqrt{a}e_i^2\sqrt{a})=\tau(e_i a e_i )\to \tau(a)\).
  Hence~\(\tau\) is semifinite.

  If~\(\tau\) is faithful, then \(\tau|_{\Cont_0(X)}\) is faithful.
  This is equivalent to~\(\mu\) having full support.
  Conversely, assume that \(\Cst_\red(\Gr,\LL)=\Cst_\ess(\Gr,\LL)\)
  and that~\(\mu\) has full support.
  For any nonzero \(a\in\Cst_\red(\Gr,\LL)^+\), there is \(\varepsilon
  >0\) such that \(\setgiven*{x\in X}{\E(b)(x)>\varepsilon}\) has
  nonempty interior, see
  \cite{Kwasniewski-Meyer:Essential}*{Proposition~7.18}.
  Hence \(\tau(a)\ge \varepsilon \mu(\setgiven*{x\in
    X}{\E(b)(x)>\varepsilon})>0\).
  Thus~\(\tau\) is faithful.
\end{proof}

\begin{remark}
  The bijection in Proposition~\ref{prop:states_and_traces}
  restricts to a bijection between \(\Gr\)\nb-invariant regular
  Borel probability measures~\(\mu\) on~\(X\) and tracial
  states~\(\tau\) on \(\Cst_\red(\Gr,\LL)\) that factor
  through~\(\E\).  This is in essence proved in
  \cite{Li-Renault:Cartan_subalgebras}*{Lemmas 4.1 and~4.2}.  The
  literature gives at least two situations where all tracial states
  on \(\Cst_\red(\Gr,\LL)\) factor through~\(\E\).  Firstly, this
  happens if~\(\Gr\) is principal (see
  \cite{Li-Renault:Cartan_subalgebras}*{Lemma~4.3}, or
  \cite{Anderson:Extensions_states}*{Theorem~3.4} for a
  corresponding result for abstract \(\Cst\)\nb-inclusions).
  Secondly, this happens if~\(\Gr\) is ample and ``almost finite''
  (see
  \cite{Ara-Boenicke-Bosa-Lia:Comparison_almost_finite}*{Lemma~3.1}).
\end{remark}

\section{The type semigroup}
\label{sec:type_semigroup}

Let~\(\Gr\) be an \'etale groupoid, \(X\) its locally compact Hausdorff unit space and
\(\Bis(\Gr)\) its inverse semigroup of (open) bisections.  We define
a type semigroup for~\(\Gr\) that depends on an auxiliary choice,
which we call an \emph{inverse semigroup basis}.  We work in this
generality to unify definitions from \cites{Boenicke-Li:Ideal,
  Ma:Purely_infinite_groupoids, Rainone-Sims:Dichotomy}.  A
bisection \(W\subseteq \Gr\) with \(W^2=W\) is an open subset
of \(X\subseteq \Gr\).  The idempotent lattice
\(\setgiven{W\in\Bis(\Gr)}{W^2=W}\) coincides with the lattice of
open subsets of~\(X\).

\begin{definition}
  \label{def:isg_basis}
  An \emph{inverse semigroup basis} for~\(\Gr\) is a subset
  \(\B\subseteq \Bis(\Gr)\) such that
  \begin{itemize}
  \item \(\B\) is an inverse subsemigroup, that is, closed under
    multiplication and involution;
  \item \(\B\) is a basis for the topology of~\(\Gr\);
  \item
    \(\OO\defeq
    \setgiven{W\in\B}{W^2=W}=\setgiven{W\in\B}{W\subseteq X}\) is
    closed under finite unions.
  \end{itemize}
\end{definition}

The product of bisections contained in~\(X\) is just their
intersection.
Hence the above assumptions imply that~\(\OO\) is a lattice of sets
that generates the topology of~\(X\).

We are going to define a type semigroup for~\(\Gr\) that depends
on~\(\B\).  If \(\B=\Bis(\Gr)\), we recover the definition by
Ma~\cite{Ma:Purely_infinite_groupoids}.  If~\(X\) is not metrisable,
then it is useful for some results to let~\(\B\) be the set of all
\(\sigma\)\nb-compact open bisections, because the
\(\sigma\)\nb-compact open subsets are exactly the open supports of
\(\Cont_0\)\nb-functions.  Another natural choice for~\(\B\) is the
family of  bisections \(V\in \Bis(\Gr)\) that  are \emph{precompact}
in the sense that both \(\rg(V)\) and \(\s(V)\) are precompact subsets of~\(X\).
In other words, it is natural to assume that elements of the lattice
\(\OO= \B\cap 2^X\) are precompact \(\sigma\)\nb-compact subsets
of~\(X\).   When \((\Gr,\LL)\) is a twisted
groupoid, one may in addition assume that the bundle~\(\LL\) is
trivial on all bisections in~\(\B\).  If~\(\Gr\) is ample, then we may
let~\(\B\) be the set of all compact-open bisections.  This case is
studied in~\cites{Boenicke-Li:Ideal, Rainone-Sims:Dichotomy}.  The
type semigroup that we are going to construct will turn out to be a
quotient of the type semigroup studied in~\cites{Boenicke-Li:Ideal,
  Rainone-Sims:Dichotomy}.

\subsection{A type semigroup for a topological space and the
  way-below relation}

We first define a type semigroup for a locally compact Hausdorff topological space~\(X\) and a
lattice~\(\OO\) of open subsets that generates the topology of~\(X\).
Let
\begin{equation}
  \label{eq:def_FO}
  \F(\OO)\defeq
  \setgiven[\bigg]{\sum_{k=1}^n 1_{U_k}}{n \in\N,\ U_k \in \OO
    \textup{ for }k=1,\dotsc,n}.
\end{equation}
This set becomes an ordered monoid with the pointwise addition of
functions and the pointwise inequality~\(\le\).  It is generated as
a monoid by the characteristic functions of the subsets in~\(\OO\).

\begin{proposition}
  \label{pro:FO}
  Every \(f\in \F(\OO)\) is a bounded, lower semicontinuous functions
  \(X\to\N\) and can be written uniquely as \(f=\sum_{k=1}^n 1_{U_k}\)
  for a decreasing chain
  \(U_1\supseteq U_2\supseteq \dotsb \supseteq U_n\) of nonempty sets
  in~\(\OO\); then \(U_k \defeq f^{-1}(\N_{\ge k})\), for
  \(k=1,\dotsc,n\).
  \begin{enumerate}
  \item If~\(\OO\) is the whole topology of~\(X\), then~\(\F(\OO)\) is
    the set of all bounded, lower semicontinuous functions \(X\to\N\);
  \item If~\(\Gr\) is ample and~\(\OO\) is the family of compact-open
    subsets, then \(\F(\OO)=\Contc(X,\N)\) is the set of compactly
    supported continuous functions \(X\to\N\).
  \end{enumerate}
\end{proposition}

\begin{proof}
  Let \(f\colon X\to\N\) be any bounded function.  Let
  \(n\defeq \norm{f}_{\infty}\) and \(U_k \defeq f^{-1}(\N_{\ge k})\)
  for \(k=1,\dotsc,n\).  Then \((U_k)_{k=1}^n\) is a decreasing chain
  of nonempty subsets and \(f=\sum_{k=1}^n 1_{U_k}\) (\(U_k=U_{k+1}\)
  is allowed).  Conversely, if \(f=\sum_{k=1}^n 1_{U_k}\) for a
  decreasing chain of nonempty subsets \((U_k)_{k=1}^n\), then
  \(f(x) \ge k\) if and only if \(x\in U_k\).  In addition, let
  \(U_0\defeq X\) and \(U_{n+1}\defeq \emptyset\).  Then
  \(f|_{U_k\setminus U_{k+1}} = k\) for \(k=0,1,\dotsc,n\).

  A function~\(f\) of this form is lower semicontinuous if and only
  if the subsets \(U_1,\dotsc, U_n\) are open.  Thus all bounded, lower
  semicontinuous functions \(X\to\N\) belong to \(\F(\OO)\) if all
  open subsets are in~\(\OO\).  Conversely, functions in \(\F(\OO)\)
  must be bounded and lower semicontinuous because they are sums of
  bounded, lower semicontinuous functions.  Now let~\(\OO\) be the
  set of compact-open subsets.  Then the functions in~\(\F(\OO)\) are
  continuous with compact support.  Conversely, if~\(f\) is
  continuous with compact support, then the subsets~\(U_k\) defined
  above are compact and open for \(k\ge1\), so that \(f\in \F(\OO)\).
  Thus~\(\F(\OO)\) is the set of all continuous functions \(X\to\N\)
  with compact support.

  Now let~\(\OO\) be general and let \(f\in \F(\OO)\).  That is,
  \(f= \sum_{k=1}^m 1_{V_k}\) for some \(V_1,\dotsc,V_m\in \OO\).
  We claim that the associated decreasing nonempty subsets
  \(U_k \defeq f^{-1}(\N_{\ge k})\), \(k=1,\dotsc,n\), belong
  to~\(\OO\) as well.  By definition, \(x\in X\) belongs to~\(U_k\)
  if and only if~\(x\) belongs to at least~\(k\) of the
  subsets~\(V_i\).  Hence each~\(U_k\) is the union of the
  intersections
  \begin{equation}
    \label{eq:intersection_with_indices}
    V_I \defeq \bigcap_{i\in I} V_i
  \end{equation}
  for \(I\subseteq \{1,\dotsc,n\}\) with \(\abs{I} = k\).
  Since~\(\OO\) is closed under finite unions and intersections, it
  follows that \(U_k \in \OO\).  So each \(f\in \F(\OO)\) has the
  asserted special form.
\end{proof}

\begin{proposition}
  \label{pro:lub_glb}
  Any finite subset of~\(\F(\OO)\) has a least upper bound
  in~\(\F(\OO)\), namely, the pointwise maximum of these functions.
  Any nonempty finite subset of~\(\F(\OO)\) has a greatest lower
  bound in~\(\F(\OO)\), namely, the pointwise minimum of these
  functions.
\end{proposition}

\begin{proof}
  The least upper bound of the empty subset of~\(\F(\OO)\) is the
  minimal element of~\(\F(\OO)\), which is the constant
  function~\(0\).  There is no maximal element in~\(\F(\OO)\), so
  that the empty subset of~\(\F(\OO)\) has no greatest lower bound.
  The two statements for nonempty finite subsets follow once they
  are known for two elements.  Pick \(f,g\in \F(\OO)\).  Let
  \(f\vee g\colon X\to\N\) and \(f\wedge g\colon X\to\N\) be their
  pointwise maximum and minimum, respectively.  We claim that
  \(f\vee g\) and \(f\wedge g\) belong to~\(\F(\OO)\).  Then they
  clearly serve as a least upper bound and a greatest lower bound
  for \(\{f, g\}\) in~\(\F(\OO)\).  Let \(U_j = f^{-1}(\N_{\ge j})\)
  and \(V_j = g^{-1}(\N_{\ge j})\) for all \(j\in\N_{\ge1}\).  Then
  \(f = \sum_{j=1}^n 1_{U_j}\) and \(g = \sum_{j=1}^n 1_{V_j}\) for
  any sufficiently large~\(n\) by Proposition~\ref{pro:FO}.  For the
  pointwise maximum and minimum, we find
  \((f\vee g)^{-1}(\N_{\ge j}) = f^{-1}(\N_{\ge j}) \cup
  g^{-1}(\N_{\ge j}) = U_j \cup V_j \in \OO\) and
  \((f\wedge g)^{-1}(\N_{\ge j}) = f^{-1}(\N_{\ge j}) \cap
  g^{-1}(\N_{\ge j}) = U_j \cap V_j \in \OO\).  Thus
  \(f\vee g = \sum_{j=1}^n 1_{U_j \cup V_j}\in \F(\OO)\) and
  \(f\wedge g = \sum_{j=1}^n 1_{U_j \cap V_j}\in \F(\OO)\).
\end{proof}

We will now describe the way-below relation~\(\ll\)
in~\(\F(\OO)\), generalising
\cite{Gierz_Hofmann_Keimel_Lawson_Mislove_Scott:Continuous_lattices}*{Proposition~I-1.4}, and spend the rest of the subsection proving that~\((F(\OO),\ll)\) satisfies several properties exhibited by continuous posets.
The \emph{open support} of a function \(g\colon X\to\N\) is defined as
\[
  \supp(g)\defeq\setgiven{x\in X}{g(x)\neq 0}.
\]
For \(g\in \F(\OO)\), define an associated upper semicontinuous
function \(\cl{g}\colon X\to \N\) by
\[
  \cl{g}(x)\defeq \limsup_{y\to x} g(y).
\]
Then \(\cl{\supp(g)}=\supp(\cl{g})\).  If \(g=\sum_{i=1}^n 1_{U_i}\)
for a decreasing chain~\((U_i)\) as in Proposition~\ref{pro:FO}, then
\(\cl{g}=\sum_{i=1}^n 1_{\cl{U_i}}\).

\begin{proposition}
  \label{pro:compactly_contained}
  Let \(f,g\in \F(\OO)\).  Then \(g\ll f\) if and only if
  \(\cl{g}\le f\) and the support of~\(g\) is precompact.
\end{proposition}

\begin{proof}
  Assume first that \(g \ll f\).  We are going to prove that
  \(\cl{g} \le f\) and that the support of~\(g\) is precompact.  Write
  \(g = \sum_{j=1}^n 1_{V_j}\) and \(f = \sum_{j=1}^n 1_{U_j}\) for
  decreasing chains \((V_j)\) and~\((U_j)\) as in
  Proposition~\ref{pro:FO}.  We allow some \(U_j\) or~\(V_j\) to be
  empty to have the same upper index~\(n\) in the sums.  Let~\(N\) be
  the set of all \(n\)\nb-tuples of precompact subsets \(W_j \in \OO\)
  with \(\cl{W_j} \subseteq U_j\) for \(j=1,\dotsc,n\).  For each such
  \(n\)\nb-tuple~\(\alpha\), define a function in \(\F(\mathcal{O})\)
  as \(f_\alpha\defeq\sum_{j=1}^n 1_{W_j}\).  Then \(f_\alpha \le f\)
  because \(W_j \subseteq U_j\) for \(j=1,\dotsc,n\).  Let
  \(N_f\defeq \setgiven{f_\alpha}{\alpha \in N}\).  Every finite
  subset of~\(N\) has an upper bound, namely, the union of the
  corresponding subsets~\(W_j\).  The union is still in~\(\OO\)
  as~\(\OO\) is closed under finite unions.  So~\(N\) is a directed
  set.  This makes~\(N_f\) a directed set under~\(\le\).  Let
  \(x\in X\) with \(f(x)=k\).  Then~\(x\) is in \(U_1,\dotsc,U_k\) but
  not in \(U_{k+1},\dotsc,U_n\).  There is
  \(\alpha=(W_1,\dotsc,W_n)\in N\) with \(x\in W_1,\dotsc,W_k\).
  Since \(W_j\subseteq U_j\), the element~\(x\) cannot lie in
  \(W_{k+1},\dotsc,W_n\).  Thus \(f_\alpha\in N_f\) satisfies
  \(f_\alpha (x)=k\).  We can do this for every \(x\in X\).  Hence
  \(\sup_{\alpha\in N} f_\alpha=f\).  So~\(N_f\) fulfils the axioms on
  a directed set in the definition of the way-below relation.
  Therefore, \(g\ll f\) implies that there is \(\alpha = (Y_j) \in N\)
  with \(g\le f_\gamma\le f\).  As each \(Y_j\) is precompact, the
  support of \(f_\gamma\) and thus of~\(g\) is precompact.  Since
  \(Y_j\subseteq \cl{Y_j}\subseteq U_j\), we get \(\cl{g}\le f\).

  Conversely, assume that \(\cl{g} \le f\) and that the support
  of~\(g\) is precompact.  We are going to prove that \(g \ll f\).
  As in the proof of Proposition~\ref{pro:FO}, the subsets
  \(K_j \defeq \cl{g}^{-1}(\N_{\ge j})\) form a decreasing chain of
  subsets with
  \(X = K_0 \supseteq K_1 \supseteq \dotsb \supseteq K_\ell =
  \emptyset\) for some \(\ell\in\N\).  These subsets are closed
  because~\(\cl{g}\) is upper semicontinuous, and~\(K_1\) is compact
  because it is the closure of the support of~\(g\).  Let
  \((h_n)_{n\in N}\) be any increasing net in~\(\F(\OO)\) with
  \(\sup h_n \ge f\).  Write
  \(h_n = \sum_{j=1}^{\ell_n} 1_{V_{n,j}}\) as in
  Proposition~\ref{pro:FO}, that is,
  \(V_{n,j} = h_n^{-1}(\N_{\ge j})\).  Since~\((h_n)\) is an
  increasing net, so is the net of subsets~\((V_{n,j})\) for
  fixed~\(j\).  If \(x\in K_j\), then \(j \le \cl{g}(x) \le f(x)\).
  Then there is \(n\in N\) with \(j \le h_n(x)\), so that
  \(x \in V_{n,j}\).  It follows that
  \(K_j \subseteq \bigcup_{n\in N} V_{n,j}\).  Since~\(K_j\) is
  compact and the net~\((V_{n,j})\) is increasing, there is
  \(n_j\in N\) with \(K_j \subseteq V_{n_j,j}\).  Since~\(N\) is
  directed, there is \(n\in N\) with \(n \ge n_j\) for
  \(j=1,\dotsc,\ell\).  Then \(K_j \subseteq V_{n,j}\) for
  \(j=1,\dotsc,\ell\).  This says that \(g \le \cl{g} \le h_n\).
  Since the increasing net~\((h_n)\) was arbitrary, this says that
  \(g \ll f\).
\end{proof}

\begin{remark}\label{rem:way_below_sets}
  If \(U,V \in \OO\), then we briefly write \(V\ll U\) for
  \(1_V \ll 1_U\).  By the proposition, this is equivalent to~\(V\)
  being precompact with \(\cl{V}\subseteq U\).
\end{remark}

\begin{remark}
  If~\(\Gr\) is ample and~\(\B\) is the set of all compact-open
  bisections, then the proposition says that~\(\ll\) is the same
  relation as~\(\le\).
\end{remark}

\begin{remark}
  \label{rem:compact_in_FO}
  An element of a poset \(x\in L\) with \(x \ll x\) is called
  \emph{compact} or \emph{isolated from below}.  By
  Proposition~\ref{pro:compactly_contained}, \(f\in \F(\OO)\)
  satisfies \(f \ll f\) if and only if the
  support of~\(f\) is compact and \(f = \cl{f}\).
  Equivalently, \(f\) is continuous with compact support.
\end{remark}

The following proposition says that the poset \(\F(\OO)\) is
continuous as in Definition~\ref{def:way_below}:

\begin{proposition}
  \label{pro:FO_continuous}
  Let \(f\in \F(\OO)\).  The set of \(g\in \F(\OO)\) with \(g \ll f\)
  is directed and~\(f\) is its supremum.
\end{proposition}

\begin{proof}
  Proposition~\ref{pro:lub_glb} shows that~\(\F(\OO)\) has finite
  suprema.  Therefore, the set of \(g\in \F(\OO)\) with \(g \ll f\)
  is directed by Remark~\ref{rem:easy_ll_properties}.  Write
  \(f= \sum_{i=1}^n 1_{V_i}\) with a decreasing chain
  \(V_1 \supseteq V_2 \supseteq \dotsb \supseteq V_n\) as in
  Proposition~\ref{pro:FO}.  If \(W_j\in \OO\) are chosen so that
  \(\cl{W_j} \subseteq V_j\), \(W_j\) is precompact and~\((W_j)\) is a
  decreasing chain, then
  \(g\defeq \sum 1_{W_j} \in \F(\OO)\) and \(g\ll f\) by
  Proposition~\ref{pro:compactly_contained}.  The proof of
  Proposition~\ref{pro:compactly_contained} shows that the pointwise
  supremum of the set of \(g\in \F(\OO)\) of this form is equal
  to~\(f\).  Since \(f\in \F(\OO)\), this is also a supremum in the
  poset~\(\F(\OO)\).  Since \(g\ll f\) implies \(g\le f\), the
  supremum stays the same if we allow all \(g \ll f\).
\end{proof}

\begin{corollary}
  \label{cor:intermediate_way_below}
  If \(f,g\in \F(\OO)\) satisfy \(f \ll g\), then there is
  \(h\in \F(\OO)\) with \(f \ll h \ll g\).
\end{corollary}

\begin{proof}
  This is a general feature of continuous posets by
  \cite{Gierz_Hofmann_Keimel_Lawson_Mislove_Scott:Continuous_lattices}*{Theorem~I-1.9}.
  We will strengthen this result in
  Corollary~\ref{cor:intermediate_way_below_improved} below.
\end{proof}

Our monoid~\(\F(\OO)\) specialises to the monoid \(\Contc(X,\N)\) in
the ample case.  This monoid was used by Rainone and
Sims~\cite{Rainone-Sims:Dichotomy}.
B\"onicke--Li~\cite{Boenicke-Li:Ideal} and
Ma~\cite{Ma:Purely_infinite_groupoids} started instead with the free
commutative monoid generated by the sets in~\(\OO\).  We are going
to relate these two different starting points for the construction
of the type semigroup.  The free monoid on the set~\(\OO\) is the
set~\(\Free_\OO\) of words with letters in~\(\OO\),
\[
  \Free_\OO\defeq
  \setgiven{(U_1,U_2,\dotsc,U_n)}{n\in \N,\ U_i\in \OO
    \textup{ for } i=1,\dotsc,n}
\]
with concatenation as multiplication.  The characteristic function
map induces a canonical surjective homomorphism
\[
  \mathrm{can}\colon \Free_\OO \to \F(\OO),\qquad
  (U_1,U_2,\dotsc,U_n) \mapsto \sum_{j=1}^n 1_{U_j}.
\]
We are going to describe the pull-back of the order relation~\(\le\)
to~\(\Free_\OO\) using compact containment.

\begin{lemma}
  \label{lem:choose_intermediate}
  For any open subsets \(U_1,\dotsc,U_n\subseteq X\) and a compact
  subset \(K\subseteq \bigcup_{i=1}^n U_i\), there are precompact
  \(V_i\in \OO\) such that \(K\subseteq \bigcup_{i=1}^n V_i\), and
  \(\cl{V_i}\subseteq U_i\) for \(i=1\dotsc,n\).
\end{lemma}

\begin{proof}
  If \(x\in K\), then \(x\in U_i\) for some~\(i\), and then there is a
  compact neighbourhood of~\(x\) contained in~\(U_i\).  Since~\(\OO\)
  is a basis, there is \(V_x\in \OO\) that is contained in this
  neighbourhood.  A finite union of these subsets covers~\(K\)
  because~\(K\) is compact.  Let~\(V_i\) be the union of those~\(V_x\)
  contained in~\(U_i\).  This is a collection of subsets with the
  required properties.
\end{proof}

\begin{lemma}
  \label{lem:compare_sum_of_characteristic_functions}
  Let \(K_1,\dotsc,K_n\subseteq X\) be compact and let
  \(V_1,\dotsc,V_m \subseteq X\) be open subsets.  Assume that
  \(\sum_{i=1}^n 1_{K_i} \le \sum_{j=1}^m 1_{V_j}\).  Then there are
  precompact subsets \(W_{i,j} \in \OO\) for \(1\le i \le n\),
  \(1 \le j \le m\) such that
  \(K_i \subseteq \bigcup_{j=1}^m W_{i,j}\) for all~\(i\) and
  \(\bigsqcup_{i=1}^n \overline{W}_{i,j} \subseteq V_j\) for
  all~\(j\), that is, the subsets~\(\cl{W_{i,j}}\) for
  \(i=1,\dotsc,n\) are
  disjoint and contained in~\(V_j\).

  If we are given neighbourhoods \(U_i\supseteq K_i\), then we may
  arrange \(W_{i,j} \subseteq U_i\) for all \(i,j\).
\end{lemma}

\begin{proof}
  We prove this by induction on~\(m\).  The case \(m=0\) is clear:
  then \(\sum 1_{K_i} \le \sum 1_{V_j}\) says that all the
  subsets~\(K_i\) are empty.  We are going to prove the induction step
  to~\(m\) subsets~\(V_j\), assuming the assertion for \(m-1\) open
  subsets~\(V_j\).

  The function \(\sum 1_{V_j} - \sum 1_{K_i}\) is nonnegative and
  lower semicontinuous.  Therefore, its zero set~\(E\) is the preimage
  of \((-\infty,0]\), and this is a closed subset.  The functions
  \(f\defeq\sum 1_{V_j}\) and \(g\defeq\sum 1_{K_i}\) restrict to the
  same function on~\(E\), and they are lower and upper semicontinuous
  functions to~\(\N\), respectively.  Therefore, their common
  restriction to~\(E\) is continuous.  Thus \(f|_E=g|_E\) is locally
  constant.  Now let~\(E_k\) be the set of points in~\(E\) that belong
  to exactly~\(k\) of the subsets~\(V_j\).  These subsets are closed,
  and~\(E\) is their disjoint union: \(E = \bigsqcup_{k=0}^m E_k\).

  For subsets \(I\subseteq \{1,\dotsc,n\}\) and
  \(J\subseteq \{1,\dotsc,m\}\), define
  \(K_I\defeq \bigcap_{i\in I}K_i\) and
  \(V_J\defeq\bigcap_{j\in J}V_j\) as
  in~\eqref{eq:intersection_with_indices}.  Any point in~\(E_k\)
  belongs to \(E_k\cap K_I \cap V_J\) for some subsets \(I,J\) with
  exactly~\(k\) elements.  These subsets are disjoint because if
  \(x\in E_k \cap K_I \cap V_J\) and
  \(x\in E_k \cap K_{I'} \cap V_{J'}\), then
  \(x\in E_k \cap K_{I\cup I'} \cap V_{J\cup J'}\), and this is only
  nonempty if \(I\cup I'\) and \(J\cup J'\) again have exactly~\(k\)
  elements, so that \(I=I'\) and \(J=J'\).  Therefore,
  \[
    E = \bigsqcup_{\mathstrut k=0}^m \;
    \bigsqcup_{\mathstrut \abs{I}=\abs{J}=k} E_k \cap K_I \cap V_J.
  \]
  All these disjoint subsets are closed.  More specifically,
  \(\bigcup_{i=1}^n K_i\subseteq \bigcup_{j=1}^m V_j\) implies
  \(E_0\cap K_{\emptyset}\cap V_{\emptyset}=E_0=X\setminus
  \bigcup_{j=1}^m V_j\), and this set is closed as a complement of
  an open subset.  For \(k>0\) and \(I\subseteq \{1,\dotsc,n\}\),
  \(J\subseteq \{1,\dotsc,m\}\) with \(\abs{I}=\abs{J}=k\),
  \(E_k \cap K_I \cap V_J\) is equal to
  \(E_k \cap K_I\setminus\bigcup_{j\notin J} V_j\), and this set is
  compact as a closed subset of the compact set \(E_k \cap K_I\).

  These disjoint closed subsets may be enlarged to open
  neighbourhoods~\(W_{k,I,J}\) whose closures remain disjoint; here
  we use Urysohn's Lemma (on an appropriate compact neighbourhood
  of~\(W_{k,I,J}\), which may be obtained using
  Lemma~\ref{lem:choose_intermediate}) and that for \(k>0\), the
  subsets \(E_k \cap K_I \cap V_J\) are compact.
  Since the subsets~\(V_J\) are open, we may also arrange \(W_{k,I,J}
  \subseteq V_J\).
  For \(k>0\), Lemma~\ref{lem:choose_intermediate} allows us to
  arrange also that~\(W_{k,I,J}\) is precompact and belongs
  to~\(\OO\).
  If some open subsets \(U_i\supseteq K_i\) are given, then we may
  also arrange that \(W_{k,I,J} \subseteq U_I\defeq \bigcap_{i\in
    I}U_i\) for all \(k>0\).

  For each pair \((I,J)\) as above with \(\abs{I} = \abs{J}\), we
  fix a bijection \(\sigma_{I,J} \colon I \congto J\).  Now for
  \(i=1,\dotsc,n\), let
  \[
    W_{i,m} \defeq \bigsqcup {}
    \setgiven{W_{k,I,J}}{\abs{I}=\abs{J}=k>0,\
      i \in I,\ \sigma_{I,J}(i) = m}.
  \]
  Then \(W_{i,m} \subseteq V_m\) because \(\sigma_{I,J}(i) = m\)
  forces \(m\in J\) and then
  \(E_k \cap K_I \cap V_J \subseteq V_m\).  The subsets~\(W_{i,m}\)
  for different~\(i\) are disjoint because each~\(W_{k,I,J}\) may
  occur for at most one~\(i\).  Let
  \(K_i' \defeq K_i \setminus W_{i,1}\) for \(i=1\dotsc,n\).  These
  are still compact subsets.  We claim that
  \(\sum_{i=1}^n 1_{K_i'} \le \sum_{j=1}^{m-1} 1_{V_j}\).  If
  \(x\notin E\), then
  \[
    \sum_{i=1}^n 1_{K_i'}(x)
    \le \sum_{i=1}^n 1_{K_i}(x)
    \le \sum_{j=1}^{m-1} 1_{V_j}(x) + (1_{V_m}(x) -1)
    \le \sum_{j=1}^{m-1} 1_{V_j}(x).
  \]
  If \(x\in E\), then \(x\in E_k \cap V_I \cap W_J\) for some
  \(I,J\) with \(\abs{I} = \abs{J}=k\).  If \(m\notin J\), then we
  are outside~\(V_m\), so that
  \(\sum_{j=1}^{m-1} 1_{V_j}(x) = \sum_{j=1}^m 1_{V_j}(x)\), and our
  inequality follows.  If \(m\in J\), then \(m= \sigma_{I,J}(i)\)
  for a unique \(i\in I\).  Then \(x\in K_i \cap W_{1,i}\) for
  exactly this~\(i\), and we compute
  \[
    \sum_{i=1}^n 1_{K_i'}(x)
    = \sum_{i=1}^n 1_{K_i}(x) -1
    = \sum_{j=1}^{m-1} 1_{V_j}(x) + (1_{V_m}(x) -1)
    = \sum_{j=1}^{m-1} 1_{V_j}(x).
  \]
  This proves the claim in all cases.  Now we apply the induction
  hypothesis to the subsets \(K_1',\dotsc,K_n'\) and
  \(V_1,\dotsc,V_{m-1}\).  It gives us open subsets
  \(W_{i,j}\in\OO\) for \(1\le i \le n\), \(1\le j\le m-1\) with
  \(K_i' \subseteq \bigcup_{j=1}^{m-1} W_{i,j}\) for all~\(i\) and
  \(\bigsqcup_{i=1}^n W_{i,j} \subseteq V_j\) for all~\(j\).  The
  first inclusion implies \(K_i \subseteq \bigcup_{j=1}^m W_{i,j}\)
  as needed.
\end{proof}

\begin{corollary}
  \label{cor:intermediate_way_below_improved}
  If \(k, f,g\in \F(\OO)\) satisfy \(k \ll f+g\), then there are
  \(k_1,k_2\in \F(\OO)\) with \(k_1\ll f\), \(k_2\ll g\) and
  \(k \ll k_1+k_2 \ll f+g\).
\end{corollary}

\begin{proof}
  Write \(k=\sum_{i=1}^n 1_{K_i}\), \(f=\sum_{i=1}^l 1_{V_i}\) and
  \(g=\sum_{i=l+1}^m1_{V_i}\) with \(K_i, V_i\in\OO\) and~\((K_i)\) a
  decreasing chain.  By Proposition~\ref{pro:compactly_contained},
  \(k \ll f+g\) means that
  \(\sum_{i=1}^n 1_{\overline{K_i}} \le \sum_{j=1}^m 1_{V_j}\).  By
  Lemma~\ref{lem:compare_sum_of_characteristic_functions}, this
  implies that there are precompact subsets \(W_{i,j}\in \OO\) with
  \(\overline{K_i} \subseteq \bigcup_{j=1}^m W_{i,j}\) for all~\(i\)
  and \(\bigsqcup_{i=1}^n \overline{W}_{i,j} \subseteq V_j\) for
  all~\(j\).  Then \(k_1\defeq \sum_{j=1}^l\sum_{i=1}^n 1_{W_{i,j}}\)
  and \(k_2\defeq \sum_{j=l+1}^m\sum_{i=1}^n 1_{W_{i,j}}\) are
  elements in~\(\F(\OO)\) with the desired properties, again by
  Proposition~\ref{pro:compactly_contained}.
\end{proof}

\begin{remark}
  \label{rem:Lsc_Cuntz_semigroup}
  By \cite{Elliott-Im:Colored}*{Proposition 2.13}, the set
  \(\text{Lsc}_{\sigma}(X,\overline{\N})\) of lower semicontinuous
  extended positive integer-valued functions~\(f\) such that
  \(f^{-1}(k,\infty]\) is \(\sigma\)\nb-compact for each \(k\in \N\),
  is an abstract Cuntz semigroup.  Thus
  \(\text{Lsc}_{\sigma}(X,\overline{\N})\) has the properties
  described in Proposition~\ref{pro:FO_continuous} and
  Corollary~\ref{cor:intermediate_way_below_improved}, which are
  axioms of an abstract Cuntz semigroup.  In particular, if we
  choose~\(\OO\) to consist of all \(\sigma\)\nb-compact open subsets
  of~\(X\), then \(F(\OO)\) is a subsemigroup of all bounded
  functions in \(\text{Lsc}_{\sigma}(X,\overline{\N})\), and hence
  inherits the afforementioned properties from
  \(\text{Lsc}_{\sigma}(X,\overline{\N})\).
\end{remark}

\subsection{Definition of the type semigroup for a groupoid}

Throughout this subsection, let~\(\Gr\) be an \'etale
groupoid with a locally compact Hausdorff unit space~\(X\), let \(\B\subseteq \Bis(\Gr)\)
be an inverse semigroup basis for~\(\Gr\), and let
\(\OO\defeq \setgiven{U\in\B}{U\subseteq X}\) as in
Definition~\ref{def:isg_basis}.  We are going to define a
preorder~\(\precsim_\B\) on~\(\F(\OO)\) that takes into
account~\(\B\).  To this end, we let
\[
  \F(\B)\defeq \setgiven[\bigg]{\sum_{k=1}^n 1_{W_k}}
  {n \in\N,\  W_k \in \B
    \textup{ for }k=1,\dotsc,n}.
\]
Notice that, unlike~\(\OO\), the set~\(\B\) is not closed under
finite unions.  We cannot arrange for this because unions of
bisections may fail to be bisections.  The source and range maps
\(\s,\rg\colon \Gr\rightrightarrows X\) induce maps
\(\B \rightrightarrows \OO\) because if \(W\in\B\), then
\(\s(W) = W^* W\) and \(\rg(W) = W W^*\) are idempotents in~\(\B\),
so that they belong to~\(\OO\).  Therefore, if
\(b=\sum_{k=1}^n 1_{W_k} \in \F(\B)\), then
\(\s_*(b) \defeq \sum_{k=1}^n 1_{\s(W_k)}\) and
\(\rg_*(b) \defeq \sum_{k=1}^n 1_{\rg(W_k)}\) belong to~\(\F(\OO)\).
This gives well-defined homomorphisms
\(\s_*, \rg_*\colon \F(\B)\rightrightarrows \F(\OO)\) because
\[
  (\s_*(b))(x) = \sum_{\s(\gamma)=x} b(\gamma),\qquad
  (\rg_*(b))(x)= \sum_{\rg(\gamma)=x} b(\gamma).
\]

\begin{definition}
  \label{def:preorder_relation}
  For \(f,g\in \F(\OO)\), we write \(f\precsim_\B g\) or just
  \(f\precsim g\) if, for all \(k\in \F(\OO)\) with \(k \ll f\),
  there is \(b\in \F(\B)\) with \(k \le \s_*(b)\) and
  \(\rg_*(b) \le g\).
\end{definition}

\begin{remark}
  \label{rem:compactly_below_relation2}
  By construction, the relation~\(\precsim\) is ``regular'' in the
  sense that \(f\precsim g\) holds if and only if \(k\precsim g\)
  for all \(k\ll f\).  Proposition~\ref{pro:precsim_relation} below
  shows that we get the same relation~\(\precsim\) if we require
  \(k \ll \s_*(b)\) and \(\rg_*(b) \ll g\) instead.  It is easier,
  however, to work with the definition above.
\end{remark}

The relation~\(\precsim_\B\) is compatible with the semigroup law
in \(\F(\OO)\):

\begin{lemma}
  \label{lem:additive_monoid}
  Let \(f,g, f',g'\in \F(\OO)\).  If \(f\precsim_\B f'\) and
  \(g\precsim_\B g'\), then \(f+g\precsim_\B f'+g'\).
\end{lemma}

\begin{proof}
  Take any \(k\in \F(\OO)\) with \(k \ll f+g\).  By
  Corollary~\ref{cor:intermediate_way_below_improved}, there are
  \(k_1,k_2\in \F(\OO)\) with \(k_1\ll f\), \(k_2\ll g\) and
  \(k \ll k_1+k_2 \ll f+g\).  Since \(f\precsim f'\) and
  \(g\precsim g'\), there are \(b_1,b_2\in \F(\B)\) with
  \(k_1 \le \s_*(b_1)\), \(\rg_*(b_1) \le f'\), and
  \(k_2 \le \s_*(b_2)\), \(\rg_*(b_2) \le g'\).  Let
  \(b\defeq b_1+b_2\).  Then
  \[
    k\le k_1+k_2\le \s_*(b_1)+ \s_*(b_2)=\s_*(b)
    \quad \text{and}\quad
    \rg_*(b)=\rg_*(b_1)+\rg_*(b_2)\le f'+g'.
  \]
  This shows that \(f+g\precsim f'+g'\).
\end{proof}

The following proposition relates~\(\precsim_\B\) to the relation
introduced in \cite{Ma:Purely_infinite_groupoids}*{Definitions~4.3}.

\begin{proposition}
  \label{pro:precsim_relation}
  Let \(f,g\in \F(\OO)\) and write \(f = \sum_{i=1}^n 1_{U_i}\) and
  \(g = \sum_{j=1}^m 1_{V_j}\) for \(U_i,V_j \in \OO\).  Then
  \(f\precsim_\B g\) if and only if for any compact subsets
  \(K_i\subseteq U_i\) for \(i=1,\dotsc,n\), there are a finite
  set~\(A\), maps \(\alpha\colon A \to \{1,\dotsc,n\}\) and
  \(\beta\colon A \to \{1,\dotsc,m\}\), and bisections \(B_a\in\B\)
  for \(a\in A\) such that
  \(K_i \subseteq \bigcup_{\alpha(a)=i} \s(B_a)\) for all~\(i\) and
  \(\bigsqcup_{\beta(a)=j} \rg(B_a) \subseteq V_j\) for all~\(j\).

  In addition, we may arrange that
  \(\bigsqcup_{\beta(a)=j} \cl{\rg(B_a)} \subseteq V_j\) and
  \(\bigsqcup_{\beta(a)=j} \cl{\rg(B_a)}\) is compact for all~\(j\).
\end{proposition}

\begin{proof}
  Assume first that \(f\precsim_\B g\).  Choose compact subsets
  \(K_i \subseteq U_i\).  Lemma~\ref{lem:choose_intermediate} gives
  us precompact \(L_i,L_i' \in \OO\) with
  \(K_i \subseteq L_i \subseteq \cl{L_i} \subseteq L_i' \subseteq
  \cl{L_i'} \subseteq U_i\).  Then
  \(k \defeq \sum_{i=1}^n 1_{L_i'} \in \F(\OO)\) satisfies
  \(k \ll f\).  So \(f\precsim_\B g\) gives
  \(b = \sum_{p=1}^\ell 1_{B_p} \in \F(\B)\) with \(k \le \s_*(b)\)
  and \(\rg_*(b) \le g\).  The first inequality implies
  \(\sum_{i=1}^n 1_{\cl{L_i}} \le \sum_{p=1}^\ell 1_{\s(B_p)}\).
  Then Lemma~\ref{lem:compare_sum_of_characteristic_functions} gives
  precompact sets \(W^0_{i,p} \in \OO\) for \(1\le i \le n\),
  \(1 \le p \le \ell\) such that
  \(\cl{L_i} \subseteq \bigcup_{p=1}^\ell W^0_{i,p}\) for all~\(i\)
  and \(\bigsqcup_{i=1}^n W^0_{i,p} \subseteq \s(B_p)\) for
  all~\(p\).  Lemma~\ref{lem:choose_intermediate} gives us
  precompact \(W^1_{i,p} \in \OO\) with
  \(\cl{W^1_{i,p}} \subseteq W^0_{i,p}\) and
  \(K_i \subseteq \bigcup_{p=1}^\ell W^1_{i,p}\) for all~\(i\).  Now
  \(\rg_*(b) \le g\) implies
  \(\sum_{i=1}^n \sum_{p=1}^\ell 1_{\cl{B_p(W^1_{i,p}})} \le
  \sum_{j=1}^m 1_{V_j}\); here \(B_p(W^1_{i,p})\) is the set of
  \(\rg(g)\) for \(g\in B_p\) with \(\s(g) \in W^1_{i,p}\), which is
  precompact.  By
  Lemma~\ref{lem:compare_sum_of_characteristic_functions} or by a
  much simpler argument, there are precompact \(V'_i\in\OO\) with
  \(\cl{V'_i} \subseteq V_i\) and
  \(\sum_{i=1}^n \sum_{p=1}^\ell 1_{\cl{B_p(W^1_{i,p}})} \le
  \sum_{j=1}^m 1_{V'_j}\).  We may apply
  Lemma~\ref{lem:compare_sum_of_characteristic_functions} again to
  find precompact \(W^2_{i,p,j} \in \OO\) such that
  \(\overline{B_p(W^1_{i,p})} \subseteq \bigcup_{j=1}^m
  \overline{W^2_{i,p,j}}\),
  \(\bigsqcup_{i=1}^n \bigsqcup_{p=1}^\ell W^2_{i,p,j} \subseteq
  V'_j\), and \(\cl{W^2_{i,p,j}} \subseteq \rg(B_p)\) for all
  \((i,p,j)\).

  Let~\(A\) be the set of triples \((i,p,j)\) with \(1 \le i
  \le n\), \(1\le p \le \ell\), \(1\le j \le m\), let
  \(\alpha(i,p,j) = i\), \(\beta(i,p,j) = j\), and let
  \[
    B_{(i,p,j)}
    = W^2_{i,p,j} \cdot B_p
    \defeq \setgiven{g\in B_p}{\rg(g) \in W^2_{i,p,j}}
    \subseteq B_p.
  \]
  Then \(\rg(B_{(i,p,j)}) = W^2_{i,p,j}\), so that
  \[
    \bigsqcup_{\beta(i,p,j')=j} \cl{\rg(B_{(i,p,j')})}
    = \bigsqcup_{i,p} \cl{W^2_{i,p,j}}
    \subseteq V'_j \subseteq \cl{V'_j} \subseteq V_i
  \]
  for all~\(j\).  And
  \[
    \bigcup_{\alpha(i',p,j)=i} \s(B_{(i',p,j)})
    \supseteq \bigcup_{p=1}^\ell W^1_{i,p}
    \supseteq K_i.
  \]
  Since \(\cl{V'_j} \subseteq V_j\) is compact, so is
  \(\bigsqcup_{\beta(i,p,j')=j} \cl{\rg(B_{(i,p,j')})} \subseteq
  V_j\).

  Now we assume the criterion in the proposition and prove,
  conversely, that~\(f\precsim_\B g\).  So we pick \(k \ll f\).
  Proposition~\ref{pro:compactly_contained} gives a decreasing
  chain~\((X_p)\) of precompact sets in~\(\OO\) with
  \(k = \sum_{p=1}^\ell 1_{X_p}\) and
  \(\sum_{p=1}^\ell 1_{\cl{X_p}} \le \sum_{i=1}^n 1_{U_i}\).
  Lemma~\ref{lem:compare_sum_of_characteristic_functions} gives us
  precompact subsets \(W_{p,i}\in\OO\) with
  \(\cl{X_p} \subseteq \bigcup_i W_{p,i}\) and
  \(\bigsqcup W_{p,i} \subseteq U_i\).
  Lemma~\ref{lem:choose_intermediate} gives precompact
  \(W'_{p,i}\in\OO\) with \(\cl{W'_{p,i}} \subseteq W_{p,i}\) and
  \(\cl{X_p} \subseteq \bigcup_i W'_{p,i}\).  Let
  \(K_i \defeq \bigsqcup W'_{p,i}\).  Then
  \(k \le \sum 1_{K_i} \ll f\).  Now our criterion applied to
  \(\cl{K_i} \subseteq U_i\) gives us a family of bisections
  \(B_a\in \B\).  Putting \(b \defeq \sum_{a\in A} 1_{B_a}\) we get
  \(\s_*(b) \ge \sum_{i=1}^n 1_{\cl{K_i}}\) and
  \(\rg_*(b) \le \sum_{j=1}^m 1_{V_j}\).  The first inequality implies
  \(k \ll \s_*(b)\).  So~\(b\) witnesses that indeed
  \(f\precsim_\B g\).
\end{proof}

Roughly speaking, the last proposition says that
\(f\precsim_\B g\) holds if and only if any~\(k\) that is way
below~\(f\) may be decomposed into finitely many, possibly
overlapping pieces which may then be moved around by suitable
bisections so as to assemble into something that is below~\(g\).

\begin{lemma}
  The relation~\(\precsim_\B\) is a preorder.
\end{lemma}

\begin{proof}
  It is clear that~\(\precsim\) is reflexive, using a unit bisection
  for~\(b\).  To prove that~\(\precsim\) is transitive, we use the
  characterisation of \(f \precsim g\) in
  Proposition~\ref{pro:precsim_relation}.  Let \(f,g,h\in \F(\OO)\)
  satisfy \(f\precsim g\) and \(g \precsim h\).  Write
  \(f= \sum_{i=1}^n 1_{U_i}\), \(g= \sum_{j=1}^m 1_{V_j}\),
  \(h= \sum_{k=1}^\ell 1_{W_k}\).  Choose compact subsets
  \(K_i \subseteq U_i\).  Then
  Proposition~\ref{pro:precsim_relation} gives us a family of
  bisections \((B_a)_{a\in A}\) in~\(\B\) and maps
  \(\alpha\colon A \to \{1,\dotsc,n\}\),
  \(\beta\colon A\to \{1,\dotsc,m\}\) such that
  \(K_i \subseteq \bigcup_{\alpha(a) = i} \s(B_a)\) and
  \(\bigsqcup_{\beta(a) = j} \cl{\rg(B_a)} \subseteq V_j\).
  Applying Proposition~\ref{pro:precsim_relation} to the compact
  subsets \(\bigsqcup_{\beta(a) = j} \cl{\rg(B_a)}\) for
  \(1\le j \le m\) gives a further family of bisections
  \((C_d)_{d\in D}\) in~\(\B\) with maps
  \(\gamma\colon D \to \{1,\dotsc,m\}\),
  \(\delta\colon D\to \{1,\dotsc,\ell\}\) such that
  \(\bigsqcup_{\beta(a) = j} \cl{\rg(B_a)} \subseteq
  \bigcup_{\gamma(d) = j} \s(C_d)\) for all~\(j\) and
  \(\bigsqcup_{\delta(d) = k} \cl{\rg(C_d)} \subseteq W_k\) for
  all~\(k\).  Now we form the bisections \(C_d\cdot B_a\) for all
  \((a,d)\in A\times D\) with \(\beta(a) = \gamma(d)\).  When we
  fix~\(a\), then the sources of~\(C_d\) for
  \(\gamma(d) = \beta(a)\) cover
  \(\bigsqcup_{\beta(a') = j} \cl{\rg(B_{a'})}\), which
  contains~\(\rg(B_a)\).  Therefore, the union of the sources of
  \(C_d\cdot B_a\) for all such~\(d\) contains \(\s(B_a)\).
  Letting~\(a\) run through \(\alpha^{-1}(i)\), these sources cover
  all of~\(K_i\).  The range of~\(C_d\) is contained
  in~\(W_{\delta(d)}\), and these ranges for different~\(d\) with
  fixed~\(\delta(d)\) are disjoint.  Since the ranges of~\(B_a\) for
  different~\(a\) with \(\beta(a) = \gamma(d)\) are disjoint as
  well, it follows that all \(C_d\cdot B_a\) with fixed
  \(\delta(d)\) and \(\beta(a) = \gamma(d)\) have disjoint ranges.
  Therefore, the family of bisections \(C_d\cdot B_a\) witnesses
  that \(f\precsim h\).
\end{proof}

\begin{definition}
  Let~\(\approx_\B\) denote the equivalence relation
  on~\(\F(\OO)\) defined by the preorder~\(\precsim_\B\), that is,
  \(f \approx_\B g\) for \(f, g\in \F(\OO)\) if and only if
  \(f\precsim_\B g\) and \(g\precsim_\B f\).  The \emph{type
    semigroup} of~\(\Gr\) with respect to the inverse semigroup
  basis~\(\B\) is defined as the quotient
  \[
    S_\B(\Gr)\defeq \F(\OO)/{\approx_\B}
  \]
  with the partial order \([f]\precsim_\B [g]\) if \(f\precsim_\B g\)
  and the addition \([f]+[g]\defeq [f+g]\) for \(f, g\in \F(\OO)\).
  The type semigroup is a well-defined, partially ordered abelian
  monoid by Lemma~\ref{lem:additive_monoid}.
\end{definition}

\begin{remark}
  When \(\Gr=X\), then \(S_\B(\Gr)=\F(\OO)\) and \(\precsim_\B\)
  is \(\le\).
\end{remark}

\begin{remark}
  Let~\(\B\) consist of all bisections of~\(\Gr\).  The relation
  described in Proposition~\ref{pro:precsim_relation} is exactly the
  one used by Ma in \cite{Ma:Purely_infinite_groupoids}*{Definitions
    4.3 and~4.4} to define the groupoid semigroup of~\(\Gr\).
  Therefore, our type semigroup specialises to Ma's ordered
  semigroup \(\mathcal{W}(\Gr)\) in this case.
\end{remark}

\begin{remark}\label{rem:ample_case_type_semigroup}
  Let~\(\Gr\) be ample and let~\(\B\) be the family of compact open
  bisections.  In that case, the criterion in
  Proposition~\ref{pro:precsim_relation} simplifies because we may
  take \(K_i = U_i\).  Even more, we may shrink the bisections so
  that their sources become disjoint as well.  With the equivalence
  relation~\(\approx\) defined by Rainone and Sims in
  \cite{Rainone-Sims:Dichotomy}*{Definition~5.1}, this says that
  \(f\precsim g\) holds if and only if there are \(f_2,h\in \F(\OO)\)
  with \(f \approx f_2\) and \(f_2+h = g\).  There is a
  difference between the two type semigroups, however, because we
  identify \(f,g\) if \(f \precsim g\) and \(g \precsim f\), while
  Rainone and Sims use the potentially finer relation~\(\approx\).
\end{remark}

In the ample
case, if we insist to use only compact open bisections, then the
type semigroup does not depend on the choice of~\(\B\) any more:
\begin{lemma}
  \label{lem:ample_change_B}
  Let~\(\Gr\) be an ample groupoid.  Let \(\B_0\subseteq \Bis(\Gr)\)
  be the inverse semigroup basis for~\(\Gr\) that consists of all
  compact open bisections, and let \(\B\subseteq\B_0\) be another
  inverse semigroup basis for~\(\Gr\).  Then
  \(S_{\B_0}(\Gr)\cong S_\B(\Gr)\).
\end{lemma}

\begin{proof}
  The family~\(\OO\) of subsets of~\(X\) belonging to~\(\B\) is a basis
  for the topology and consists of compact open sets by
  assumption.  Therefore, any compact open subset of~\(X\) is a finite union of
  sets in~\(\OO\).  Since~\(\OO\) is closed under finite
  unions, this means that~\(\OO\)  consists of all compact open subsets of \(X\).
  If \(U\in\B\),
  then any compact open subset \(V\subseteq U\) is of the form
  \(U\cdot \s(V)\).  Here \(\s(V)\in\OO\) because it is compact and
  open.  So \(V = U\cdot \s(V) \in \B\). Hence~\(\B\) is closed under taking compact open subsets.

  Next, we claim that any \(U\in \B_0\) is a disjoint union of
  bisections in~\(\B\). Indeed, since~\(\B\) is a basis and~\(U\) is
  compact, \(U\) is a finite union of bisections in~\(\B\), and we can
  refine this to a disjoint union because~\(\B\) is closed under
  compact open subsets.  By this claim, \(\F(\B)=\F(\B_0)\).  This
  implies that the relations \(\precsim_\B\) and~\(\precsim_{\B_0}\)
  are the same.  Then so are the type semigroups defined by \(\B\)
  and~\(\B_0\).
\end{proof}

The following lemma generalises
\cite{Ma:Purely_infinite_groupoids}*{Proposition~6.1} to twisted
non-Hausdorff groupoids.  This is stated
in~\cite{Ma:Purely_infinite_groupoids} without a proof.

\begin{lemma}
  \label{lem:from_type_semigroup_to_Cuntz_semigroup}
  Let \(D=\Cst(\Gr,\LL)\) be the \(\Cst\)\nb-algebra for the twisted
  groupoid \((\Gr,\LL)\) and let
  \(\B\subseteq S(\LL)\subseteq \Bis(\Gr)\) be an inverse semigroup
  basis for~\(\Gr\) that consists of bisections that trivialise the
  bundle~\(\LL\).  Suppose that the open supports of
  \((a_i)_{i=1}^n\), \((b_j)_{j=1}^m\subseteq \Cont_0(X)^+\) belong
  to~\(\OO\).  If
  \(\sum_{i=1}^n [1_{\supp(a_i)}] \precsim \sum_{j=1}^m
  [1_{\supp(b_j)}]\) in \(S_\B(\Gr)\), then
  \(\sum_{i=1}^n [a_i] \precsim \sum_{j=1}^m [b_j]\) in
  \(W(D)\subseteq \Cu(D)\).
\end{lemma}

\begin{proof}
  We reduce to the case \(n=m=1\) by passing to the stabilised
  groupoid \((\Gr\times \RR,\LL\otimes \C)\) and the stabilised
  algebra \(D\otimes\Comp\) (see Lemma
  \ref{lem:stabilisation_groupoids}
  and Remark~\ref{remark:stabilisation_dynamical}).  Thus let us assume
  that \(a, b\subseteq \Cont_0(X)^+\) are such that
  \(\supp(a),\supp(b)\in \OO\) and \(\supp(a) \precsim \supp(b)\)
  in~\(\F(\OO)\).  Let \(\varepsilon >0\) and put
  \(K\defeq \cl{\supp(a-\varepsilon)_+}\).  There are bisections
  \(W_1,\dotsc,W_N\in\B\) such that
  \(K\subseteq \bigcup_{k=1}^N \s(W_k)\) and
  \(\bigsqcup_{k=1}^N \rg(W_k) \subseteq \supp(b)\).  Let
  \((w_k)_{k=1}^N\subseteq \Contc(X)\) be a partition of unity
  subordinate to the open covering
  \(K\subseteq \bigcup_{k=1}^N \s(W_k)\).  Put
  \(z_k\defeq (a-\varepsilon)_+^{\frac{1}{2}}\cdot
  w_k^{\frac{1}{2}}\circ \s|_{W_i}^{-1}\in \Contc(W_i)\), for
  \(k=1,\dotsc, N\).  Since~\(W_i\) trivialises the bundle~\(\LL\),
  we may identify \(\Contc(W_i)\) with
  \(\Contc(\LL|_{W_i})\subseteq D\).  Then
  \(z\defeq \sum_{k=1}^N z_k\in D\).  Since
  \(\rg(W_i)\cap \rg(W_j)=\emptyset\) we get \(z_i^*z_j=0\) in~\(D\)
  for \(i\neq j\).  Therefore,
  \[
    z^* z
    = \sum_{k=1}^N z_k^*z_k
    = \sum_{k=1}^N (a-\varepsilon)_+\cdot w_k
    = (a-\varepsilon)_+.
  \]
  Since \(\bigsqcup_{k=1}^N \rg(W_k) \subseteq \supp(b)\), we get
  \(zz^* \in b Bb\).  Thus \(a\precsim b\) in~\(D\)
  by~\eqref{eq:algebraic_Cuntz_comparison}.
\end{proof}

\begin{corollary}
  \label{cor:from_type_semigroup_to_Cuntz_semigroup}
  Let~\(D\) be an exotic \(\Cst\)\nb-algebra for the twisted
  groupoid \((\Gr,\LL)\) and let \(\B\subseteq\Bis(\Gr)\) be an
  inverse semigroup basis for~\(\Gr\).  Assume that the bisections
  in~\(\B\) are \(\sigma\)\nb-compact and trivialise the
  bundle~\(\LL\).  There is an order-preserving homomorphism
  \(\Psi\colon S_\B(\Gr)\to W(D)\) with \(\Psi([1_U])=[a]\) for any
  \(a\in \Cont_0(X)^+\) with \(\supp(a)=U\).
\end{corollary}

\begin{remark}
  For an ample groupoid, the line bundle~\(\LL\) is trivialisable on
  any $\sigma$\nb-compact Hausdorff subset~\(Y\) of~\(\Gr\) --~and
  so the assumption in
  Corollary~\ref{cor:from_type_semigroup_to_Cuntz_semigroup} that
  the $\sigma$\nb-compact bisections in~\(\B\) trivialise the
  bundle~\(\LL\) is empty.  Indeed, since~\(\LL\) is locally
  trivial, we may cover~\(Y\) by open subsets on which~\(\LL\) is
  trivialisable.  We may refine this cover to one consisting of
  compact open subsets because the latter form a basis for the
  topology.  Since~\(Y\) is $\sigma$\nb-compact, this cover has a
  countable subcover \((V_n)_{n=1}^\infty\).  The compact open
  subsets~\(V_n\) are closed and open because~\(Y\) is Hausdorff.
  Hence the subsets
  \(U_n\defeq V_n\setminus\bigcup_{k=1}^{n-1} V_k\) form an open
  cover \((U_n)_{n=1}^\infty\) of~\(Y\) consisting of pairwise
  disjoint subsets.  Gluing together trivialisiations of~\(\LL\) on
  each piece~\(U_n\) gives a trivialisaion of~\(\LL\) on~\(Y\).

  As a consequence, if~\(\Gr\) itself is ample, Hausdorff and
  \(\sigma\)\nb-compact, then~\(\LL\) is trivial globally and so the
  twist must come from a groupoid cocycle.
\end{remark}

\begin{remark}
  Assume that the twist~\(\LL\) restricts to a nontrivial line bundle
  on a bisection \(W\subseteq \Gr\).
  Then there is no reason to expect the relation \(1_{\s(W)}\simeq
  1_{\rg(W)}\) in the Cuntz semigroup of \(\Cst(\Gr,\LL)\).
  Instead, we explain how this leads to an equivalence
  between~\(1_{\s(W)}\) and another element that does not come from the
  type semigroup of~\(\Gr\), namely, the class of the line
  bundle~\(\LL|_{\rg(W)}\).
  Any section in \(\Cont_0(W,\LL)\) must have zeros because
  otherwise~\(\LL\) would be trivial.
  There is, however, usually an embedding of~\(\LL\) into a trivial
  bundle of rank~\(k\).
  Multiply it with a scalar-valued function \(\varphi\in\Cont_0(W)\)
  without zeros to get~\(k\) sections \(a_1,\dotsc,a_k\in\Cont_0(W,\LL)\) so that
  the column matrix \(a=(a_1,\dotsc,a_k)^t \in \mathbb{M}_m(\Cont_0(W,\LL))
  \subseteq \mathbb{M}_m(\Cst(\Gr,\LL))\) has no zeros in~\(W\).
  Then~\(a\) provides an equivalence between the class of
  \(1_{\s(W)}\) and the class of the line bundle~\(\LL\) viewed as a
  vector bundle over~\(\rg(W)\).
\end{remark}

\subsection{A definition using stabilisation and Morita equivalence}

We will discuss another construction of the type semigroup, which
for ample groupoids appeared (somewhat implicitly) in
\cites{Rainone-Sims:Dichotomy, Ara_Bonicke_Bosa_Li:type_semigroup}.
One consequence is that after stabilising~\(\Gr\), we may arrange
that \(S_\B(\Gr) =\setgiven{1_U}{U\in \OO}/{\approx}\), so the type
semigroup consists only of generators~\(\OO\).

Let~\(\RR\) be the full equivalence relation
\(\RR\defeq \N\times \N\) regarded as a principal discrete groupoid
with unit space \(\RR^{(0)}\defeq \setgiven{(n,n)}{n\in \N}\)
identified with~\(\N\).  The \emph{stabilisation of the
  groupoid}~\(\Gr\) is the product groupoid
\(\K(\Gr) \defeq \Gr\times\RR\).  For our fixed basis~\(\B\) and its
sources~\(\OO\), let \(\K(\OO)\) be the set of finite unions of sets
\(U\times\{n\}\), \(U\in \OO\), \(n\times \N\), and let
\[
  \K(\B)\defeq\B\times \setgiven{\{(n,m)\} }{(n,m)\in \RR}\cup \K(\OO).
\]
Then \(\K(\B)\) is an inverse semigroup basis for \(\K(\Gr)\)
whose lattice of idempotents is \(\K(\OO)\).  The following is an
analogue of \cite{Rainone-Sims:Dichotomy}*{Proposition~5.7}:

\begin{proposition}
  \label{prop:stabilisation_vs_type_semigroup}
  The map \(S_\B(\Gr)\to S_{\K(\B)}(\Gr \times \RR)\),
  \([f]\mapsto [f\times 1_{(0,0)}]\), is an isomorphism of ordered
  monoids, and so
  \[
    S_\B(\Gr)
    \cong S_{\K(\B)}(\Gr \times \RR)
    = \setgiven{[1_U]}{U\in \K(\OO)}/{\approx}_{\K(\B)}.
  \]
  Let \(p\colon X\times \N\to X\) be the canonical projection.  The
  inverse of the isomorphism above sends~\([\tilde{f}]\) for
  \(\tilde{f}\in \F(\K(\OO))\) to \([p_* \tilde{f}]\) where
  \((p_* \tilde{f})(x) \defeq \sum_{n\in \N} \tilde{f}(x,n)\) for
  all \(x\in X=\Gr^{(0)}\).
\end{proposition}

\begin{proof}
  Assume first that \(f\precsim g\) in~\(\F(\OO)\).
  If \(\tilde{k}\in \F(\K(\OO))\) satisfies \(\tilde{k}\ll (f\times 1_{(0,0)})\),
  then \(\tilde{k}=k\times 1_{(0,0)}\) for some \(k\in \F(\OO)\)
  with \(k\ll f\).
  Thus there is \(b\in \F(\B)\) with \(k \ll \s_* (b)\) and \(\rg_*(b)
  \ll g\).
  This implies \(\tilde{k} \ll \s_*(b \times 1_{(0,0)})\) and
  \(\rg_*(b \times 1_{(0,0)}) \ll (g \times 1_{(0,0)})\).
  Hence \(f\times 1_{(0,0)}\precsim g\times 1_{(0,0)}\) in
  \(\F(\K(\OO))\).
  Conversely, if \(f\times 1_{(0,0)}\precsim g\times 1_{(0,0)}\) in
  \(\F(\K(\OO))\) and \(k\ll f\), then there are
  \(\tilde{b}\in\F(\K(\B))\) with \(k\times 1_{(0,0)}\ll \s_*
  (\tilde{b})\) and \(\rg_*(\tilde{b}) \ll g\times 1_{(0,0)}\).
  The relation \(\rg_*(\tilde{b}) \ll g\times 1_{(0,0)}\) implies that
  \(\tilde{b}(\gamma \times (n,m))=0\) whenever \(n\neq 0\).
  Thus putting \(b(\gamma)\defeq \tilde{b}(\gamma \times (0,0))\) we
  get \(b\in \F(\B)\) satisfying \(k \ll \s_* (b)\) and \(\rg_*(b) \ll
  g\).
  Hence \(f\precsim g\) in~\(\F(\OO)\).

  This implies that the map
  \(\F(\OO)\ni f \mapsto f\times 1_{(0,0)} \in \F(\K(\OO))\) factors
  through an injective homomorphism of ordered monoids
  \(\Psi\colon S_\B(\Gr)\to S_{\K(\B)}(\Gr \times \N^2)\).
  Moreover, if \(U\in \OO\) and \(n\in \N\), then
  \(1_U\times 1_{(0,0)}=1_{\s(U\times \{(n,0)\})}\approx
  1_{\rg(U\times \{(n,0)\})}=1_U\times 1_{(n,n)}\).  By additivity,
  this equivalence extends to
  \[
    \sum_{k=0}^N f_k\times 1_{(k,k)}
    \approx \left(\sum_{k=0}^N f_k\right)\times 1_{(0,0)}
    = p_*\left(\sum_{k=0}^N f_k\times 1_{(k,k)}\right) \times 1_{(0,0)},
  \]
  for any \(f_k\in \F(\OO)\), \(k=1,\dotsc, N\), \(N\in \N\).  As
  every element in \(\F(\K(\OO))\) is of the form
  \(\sum_{k=0}^N f_k\times 1_{(k,k)}\), the map~\(\Psi\) is
  surjective and its inverse is induced by~\(p_*\).

  In particular, it follows that
  \([1_U\times 1_{(n,n)}] +[1_V\times 1_{(m,m)}]= [1_U\times
  1_{(n,n)}+1_V\times 1_{(n+m,n+m)}]\) for all \(U,V\in \OO\),
  \(n,m\in \N\).  This implies
  \(S_{\K(\B)}(\Gr \times \N^2)=\setgiven{[1_U]}{U\in \K(\OO)}\).
\end{proof}

\begin{remark}
  \label{remark:stabilisation_dynamical}
  It follows that one may define \(S_\B(\Gr)\) by introducing the
  preorder relation on~\(\K(\OO)\) where \(U\precsim V\) if and only
  if for every \(K\in \K(\OO)\) with \(K\ll U\) there are bisections
  \(W_1,\dotsc,W_N\in\K(\B)\), such that
  \(K\subseteq \bigcup_{k=1}^N \s(W_k)\) and
  \(\bigsqcup_{k=1}^N \rg(W_k) \subseteq V\).  Passing to the
  quotient by the equivalence relation defined by~\(\precsim\) we
  get an ordered abelian monoid \(\K(\OO)/{\approx}\), where
  \[
    [U] + [V]\defeq [U\oplus V],
  \]
  and
  \(\sum_{i=1}^n U_i\times \{i\} \oplus\sum_{j=1}^m V_j\times \{j\}
  \defeq \sum_{i=1}^n U_i\times \{i\} +\sum_{i=n+1}^m V_{i-n}\times
  \{i\}\).  Proposition~\ref{prop:stabilisation_vs_type_semigroup}
  implies an isomorphism \(S_\B(\Gr)\cong \K(\OO)/{\approx}\) as
  ordered monoids.
\end{remark}

\begin{remark}
  \label{remark:stabilisation_dynamical2}
  We chose to define \(\K(\B)\) as above, as in a sense it is the
  smallest inverse semigroup basis for \(\K(\Gr)\) for which the
  natural isomorphism \(S_\B(\Gr) \cong S_{\K(\B)}(\Gr \times \RR)\)
  holds.  Another good choice is
  \[
    \widetilde{\K}(\B)\defeq \setgiven{U\times V\in\B\times \Bis(\RR)}{V\text{ is finite}}\cup \K(\OO).
  \]
  The proof of Proposition~\ref{prop:stabilisation_vs_type_semigroup}
  also works for that and shows that the map
  \([f]\mapsto [f\times 1_{(0,0)}]\) yields an isomorphism
  \(S_\B(\Gr) \cong S_{\widetilde{\K}(\B)}(\Gr \times \RR)\).
\end{remark}

We get the following version of
\cite{Rainone-Sims:Dichotomy}*{Corollary~5.8}, which holds for not
necessarily Hausdorff groupoids, and with our slightly different
definition of the type semigroup.

\begin{corollary}
  \label{cor:equivalence_implies_the_same_type_semigroups}
  Let \(\Gr_1\) and~\(\Gr_2\) be ample groupoids with
  \(\sigma\)\nb-compact unit spaces, and let \(\B_1\), \(\B_2\) be the
  bases of compact open subsets in \(\Gr_1\) and~\(\Gr_2\),
  respectively.  If \(\Gr_1\) and~\(\Gr_2\) are groupoid equivalent,
  then the type semigroups \(S_{\B_1}(\Gr_1)\) and~\(S_{\B_2}(\Gr_2)\)
  are order isomorphic.
\end{corollary}

\begin{proof}
  Follow the proof of \cite{Rainone-Sims:Dichotomy}*{Corollary~5.8}.
\end{proof}

So in the ample case the type semigroup is invariant under Morita
equivalence.  The following example shows, however, that this fails
for non-ample groupoids, at least if we require an isomorphism of
type semigroups that preserves the canonical map~\(\Sigma\) defined
below.  Here it does not matter which inverse semigroup basis we
choose for our groupoids.

Let \(\Gr^0/\Gr\) be the orbit space for the canonical
\(\Gr\)\nb-action on~\(\Gr^0\), that is, for the equivalence
relation defined by \(\s(g) \sim \rg(g)\) for all \(g\in\Gr\).  If
\(f\in \F(\OO)\), then we define a function on \(\Gr^0/\Gr\) by
summing over the \(\Gr\)\nb-orbits in~\(\Gr^0\):
\[
  \Sigma f\colon \Gr^0/\Gr\to\N \cup \{\infty\},\qquad
  (\Sigma f)([x]) = \sum_{\s(g)=x} f(\rg(g)).
\]

\begin{lemma}
  \label{lem:sum_over_orbits}
  The map~\(\Sigma\) descends to a well-defined order-preserving map
  on \(S_\B(\Gr)\), that is, \(\Sigma f \le \Sigma g\) if
  \(f \precsim g\).
\end{lemma}

\begin{proof}
  If \(U\in\B\), then \(\Sigma 1_{\s(U)} = \Sigma 1_{\rg(U)}\).  This
  implies \(\Sigma \s_*(b) = \Sigma \rg_*(b)\) for all \(b\in\B\).
  Therefore, if \(f \precsim g\), then \(\Sigma k \le \Sigma g\) for
  all \(k \ll f\).  Since~\(f\) is the supremum of \(k \ll f\) by
  Proposition~\ref{pro:FO_continuous} and~\(\Sigma\) preserves suprema
  for the order relation~\(\le\) on \(\F(\OO)\), \(\Sigma f\) is the
  supremum of~\(\Sigma k\) for \(k \ll f\).  Therefore,
  \(\Sigma f \le \Sigma g\).  Then \(\Sigma f = \Sigma g\) follows if
  both \(f \precsim g\) and \(g \precsim f\).  That is, \(\Sigma\) is
  well defined on \(S_\B(\Gr)\).
\end{proof}

\begin{example}
  \label{exa:counterexample_Morita}
  Let \(\Gr=\Sphere^1\) be the circle, viewed as an \'etale groupoid
  with only identity arrows.
  Let \(\GrH= \R\times\Z\) be the transformation groupoid for the
  action of~\(\Z\) on~\(\R\) by translations.
  Since this action is free and proper and its orbit space is
  identified with~\(\Sphere^1\), these two groupoids are Morita
  equivalent.
  We claim that there cannot be any isomorphism between their type
  semigroups that intertwines the maps~\(\Sigma\) to functions on the
  orbit spaces.
  Here it does not matter which inverse semigroup bases we pick for
  \(\Gr\) and~\(\GrH\).
  For~\(\Gr\), the relation~\(\precsim\) simplifies to~\(\le\) because
  \(\s_*(b) = \rg_*(b)\) for all \(b\in\B\).
  Therefore, \(S_\B(\Gr) \cong \F(\OO_\Gr)\), and this consists of
  lower semicontinuous functions \(\Sphere^1 \to \N\) by
  Proposition~\ref{pro:FO}.
  Clearly, the map~\(\Sigma\) for~\(\Gr\) is just the identity map.
  The constant function~\(1\) on the circle belongs to~\(\OO_\Gr\) for
  any choice of basis because~\(\Sphere^1\) is compact.

  We claim, however, that the range of the map~\(\Sigma\) for the
  groupoid \(\R\rtimes\Z\) does not contain the constant
  function~\(1\).
  Assume that there were \(f\in \F(\OO_\R)\) with \(\Sigma f = 1\).
  If \(f(x) \ge 2\) for some \(x\in\R\), then also \(\Sigma f([x]) \ge
  2\), so that \(\Sigma f \neq 1\).
  Therefore, \(0 \le f \le 1\).
  This makes~\(f\) the characteristic function of some open subset
  \(U\subseteq \R\).
  Since \(\Sigma f(x)=1\) for all \(x\in \Sphere^1\), this open subset
  contains exactly one point from each orbit.
  So the orbit space projection \(\R\to\Sphere^1\) restricts to a
  bijection \(U\to \Sphere^1\).
  This bijection is also a local homeomorphism, hence a homeomorphism.
  So its inverse is a continuous section for the covering map
  \(\R\to\Sphere^1\).
  But this does not exist.
  So our function~\(f\) cannot exist.
\end{example}

This example suggests to look for another definition for a type
semigroup for general \'etale groupoids, which would fix the problem
in Example~\ref{exa:counterexample_Morita}.  We could, of course,
replace open subsets by locally closed subsets, allowing half-open
intervals.  This is not a good choice, however, because it would
destroy the connection to the Cuntz semigroup of the groupoid
\(\Cst\)\nb-algebra, which is the main motivation to study the type
semigroup.

\section{Regular ideals and regular states}
\label{sec:ideals_states}

Throughout this section, \(\B\subseteq \Bis(\Gr)\) is an inverse
semigroup basis for an \'etale groupoid~\(\Gr\) with locally compact
Hausdorff unit space~\(X\), and
\(\OO=\setgiven{V\in \B}{V\subseteq X}\) as in
Definition~\ref{def:isg_basis}.  We study analogues of closed ideals
and lower semicontinuous states for \((S_\B(\Gr),{\precsim})\) (see
Definition~\ref{def:closed_lower_semicontinuous}).  A technical issue
here is that \((S_\B(\Gr),{\precsim})\) need not be continuous.  In
concrete examples (like Example~\ref{Ex:irrational_rotation}) the map
from \((\F(\OO),{\le})\) to \((S_\B(\Gr),{\precsim})\) preserves
directed suprema, which relates the way-below relations on these
monoids.  We do not know whether this always holds.  Since we want to
use nice descriptions of closed ideals and lower semicontinuous states
for the continuous monoid \((\F(\OO),{\le})\), we use the way-below
relation~\(\ll\) in \(\F(\OO)\) instead of working intrinsically in
\((S_\B(\Gr),{\precsim})\).

\subsection{Regular ideals}

\begin{definition}
  \label{def:regular_ideal}
  An ideal~\(I\) in the type semigroup \(S_\B(\Gr)\) is
  \emph{regular} if \([f]\in S_\B(\Gr)\) belongs to~\(I\) whenever
  $[k]\in I$ for all $k\in \F(\OO)$ with \(k\ll f\).
\end{definition}

\begin{remark}
  \label{rem:regular_ideal}
  When~\(\OO\) is the set of compact open subsets in an ample groupoid,
  then every ideal is regular because \(f \ll f\) (see
  Remark~\ref{rem:compact_in_FO}).  When \(\Gr=X\), then
  \(S_\B(\Gr)=\F(\OO)\) and regular ideals are the same as closed
  ideals.
\end{remark}

\begin{example}
  Let \(X=\N\) and let~\(\OO\) consist of all finite subsets and the
  whole space~\(\N\).  The subset of~\(\F(\OO)\) spanned by
  all~\(1_U\) for \(U\subseteq\N\) finite is an ideal which is not
  regular.  Its ``closure'' is the whole \(\F(\OO)\).
\end{example}

If the condition in Definition~\ref{def:regular_ideal} holds for one
representative of~\([f]\), then it holds for all:

\begin{lemma}
  \label{prop:regular_ideal_characterisations}
  Let~\(I\) be an ideal in \(S_\B(\Gr)\) and let \(x\in S_\B(\Gr)\).
  The following are equivalent:
  \begin{enumerate}
  \item \label{prop:regular ideal_characterisations_1}%
    there is \(f\in \F(\OO)\) with \(x=[f]\) such that if
    \(k\ll f\), then \([k]\in I\).
  \item \label{prop:regular ideal_characterisations_2}%
    for all \(f\in \F(\OO)\) with \(x=[f]\), if \(k\ll f\), then
    \([k]\in I\).
  \end{enumerate}
\end{lemma}

\begin{proof}
  We only need to show that \ref{prop:regular
    ideal_characterisations_1} implies \ref{prop:regular
    ideal_characterisations_2}, as the converse implication is
  obvious.  So assume that \(x=[f]=[f']\) and that \([k]\in I\)
  whenever \(k\ll f\).  Let \(k'\ll f'\).  We want to show that
  \([k']\in I\).  Since \(k'\ll f'\precsim f\), there is
  \(b\in \F(\B)\) such that \(k' \ll \s_*(b)\) and
  \(\rg_*(b) \ll f\) (see
  Remark~\ref{rem:compactly_below_relation2}).  Let
  \(k\defeq \rg_*(b)\).  Then \(k'\precsim k \ll f\).  By
  assumption, \([k]\in I\).  Since~\(I\) is \(\precsim\)\nb-downward
  directed, this implies \([k']\in I\).
\end{proof}

When~\(\OO\) is the set of compact open subsets in an ample groupoid,
the following lemma reduces to
\cite{Ara_Bonicke_Bosa_Li:type_semigroup}*{Lemma~2.3}.

\begin{lemma}
  \label{lem:order_ideals}
  There is a bijection between regular ideals in~\(S_\B(\Gr)\) and
  open \(\Gr\)\nb-\alb{}invariant subsets in~\(X\).  It maps a regular
  ideal \(I\subseteq S_\B(\Gr)\) to
  \(U_I\defeq \bigcup_{f\in \F(\OO), [f]\in I} {} \supp(f)\), and it
  maps an open \(\Gr\)\nb-invariant subset \(U\subseteq X\) to the
  submonoid~\(I_U\) of \(S_\B(\Gr)\) generated by the image of
  \(\OO_U\defeq \setgiven{V\in \OO}{V\subseteq U}\).  Moreover,
  \(I_U\) is isomorphic to \(S_{\B_U}(\Gr_U)\) for
  \(\B_U\defeq\setgiven{V\in \B}{V\subseteq \Gr_U}\).
\end{lemma}

\begin{proof}
  Let~\(U\) be a \(\Gr\)\nb-invariant open subset.  Let~\(I\) be the
  submonoid of~\(S_\B(\Gr)\) generated by~\(\OO_U\).  It consists of
  all the classes of \(\sum_{i=1}^n 1_{V_i}\) with
  \(V_i \subseteq U\).  It is easy to see that this is a regular
  ideal and that \(U=\bigcup_{f\in \F(\OO), [f]\in I}\supp(f)\).
  The identical inclusion \(\F(\OO_U)\subseteq \F(\OO)\) induces an
  order isomorphism \(I\cong S_{\B_U}(\Gr_U)\) because the
  \(\precsim\)\nb-relation among elements of~\(I\) that holds
  in~\(\Gr\) is already implemented using bisections in~\(\Gr_U\).

  Conversely, let~\(I\) be any regular ideal in \(S_\B(\Gr)\) and put
  \[
    U_I\defeq \bigcup_{f\in \F(\OO),[f]\in I}\supp(f).
  \]
  We show that~\(U_I\) is \(\Gr\)\nb-invariant.
  Let \(\gamma\in \Gr\) with \(\s(\gamma)\in U\).
  Take any \([f]\in I\) with \(\s(\gamma)\in \supp(f)\).
  Write \(f=\sum_{k=1}^n 1_{U_k}\) with \(U_k \in \OO\).
  Since~\(I\) is an ideal, it follows that \([1_{U_k}]\in I\) for
  \(k=1,\dotsc,n\).
  There is an index~\(k_0\) with \(\s(\gamma)\in U_{k_0}\).
  Using our assumptions on~\(\B\), we get a bisection \(V\in \B\)
  containing~\(\gamma\) with \(\s(V) \subseteq U_{k_0}\).
  Then \([1_{\rg(V)}]=[1_{\s(V)}]\le [1_{U_{k_0}}]\in I\).
  So \([1_{\rg(V)}]\in I\).
  Consequently, \(\rg(\gamma)\in \rg(V)\subseteq U_I\).
  Thus~\(U_I\) is \(\Gr\)\nb-invariant.
  Now let \(f\in \F(\OO)\) be any element such that \(\supp(f)
  \subseteq U_I\).
  Take any \(k\in \F(\OO)\) with \(k\ll f\).
  Then \(\cl{\supp(k)}\) is compact and contained in~\(U_I\).
  Thus there are \([f_1], \dotsc, [f_n]\in I\) with
  \(\supp(k)\subseteq \bigcup_{i=1}^n\supp(f_i)\).
  Then \(k\le \sum_{i=1}^n m f_i\) for sufficiently large~\(m\), and
  this implies \([k]\in I\).
  Since~\(I\) is regular, this implies \([f]\in I\).
  Thus the whole submonoid generated by~\(\OO_{U_I}\) is also
  contained in~\(I\).
  Since any~\(f\) with \([f]\in I\) is supported in~\(U_I\), the
  ideal~\(I\) cannot be bigger than that either.
  This proves the bijection.
\end{proof}

\begin{corollary}
  \label{cor:continuous_ideal}
  There is a bijection between closed ideals in~\(\F(\OO)\) and open
  subsets in~\(X\).  It maps a closed ideal~\(I\) in~\(\F(\OO)\) to the
  open subset \(U_I\defeq \bigcup_{f\in I}\supp(f)\) and an open subset
  \(U\subseteq X\) to the closed ideal
  \(I_U\defeq \setgiven{f\in \F(\OO)}{\supp(f)\subseteq U}\).
\end{corollary}

\begin{proof}
  Apply Lemma~\ref{lem:order_ideals} to \(\Gr=X\).
\end{proof}

\begin{corollary}
  \label{cor:minimal_implies_simple}
  Assume that~\(\OO\) consists of precompact subsets of~\(X\).  The
  following are equivalent:
  \begin{enumerate}
  \item\label{enu:minimal_implies_simple1}%
    \(S_\B(\Gr)\) is simple;
  \item\label{enu:minimal_implies_simple2}%
    \(S_\B(\Gr)\) has no nontrivial regular ideals;
  \item\label{enu:minimal_implies_simple3}%
    \(\Gr\)~is minimal.
  \end{enumerate}
\end{corollary}

\begin{proof}
  The implication
  \ref{enu:minimal_implies_simple1}\(\Rightarrow\)\ref{enu:minimal_implies_simple2}
  is obvious, and \ref{enu:minimal_implies_simple2}
  and~\ref{enu:minimal_implies_simple3} are equivalent by
  Lemma~\ref{lem:order_ideals}.
  So it suffices to show
  \ref{enu:minimal_implies_simple3}\(\Rightarrow\)\ref{enu:minimal_implies_simple1}.
  Assume that~\(\Gr\) is minimal.
  Let \(u,\vartheta\in S_\B(\Gr)\setminus\{0\}\).
  We need to show that \(\vartheta\precsim n\cdot u\) for some \(n\in
  \N\).
  Assume first that \(u=[1_U]\) and \(\vartheta=[1_V]\) for some
  \(U,V\in \OO\).
  As~\(\Gr\) is minimal and~\(\B\) is a basis for~\(\Gr\), we have
  \(\bigcup_{W\in \B} \s(U W)=X\).
  As~\(\overline{V}\) is compact, there are bisections
  \(W_1,\dotsc,W_n\in \B\) such that \(\overline{V}\subseteq
  \bigcup_{i=1}^n \s(W_i)\) and \(\rg(W_i)\subseteq U\) for
  \(i=1,\dotsc,n\).
  Thus \(\vartheta \precsim n [1_U]=n u\).
  If \(\vartheta \in S_\B(\Gr)\) is arbitrary, then
  \(\vartheta=\sum_{j=1}^m c_j [1_{V_j}]\) for some \(c_j\in \N\) and
  \(V_j\in \OO\), \(j=1,\dotsc,m\).
  By the above argument, for every \(j=1,\dotsc,m\) there is \(n_j\in
  \N\) with \([1_{V_j}] \precsim n_j u\).
  This implies that \(\vartheta \precsim \sum_{j=1}^m c_j n_j u\).
  Now, let \(u \in S_\B(\Gr)\setminus\{0\}\) and \(\vartheta \in
  S_\B(\Gr)\) be arbitrary.
  Then \(u=\sum_{j=1}^m c_j [1_{U_j}]\) for some
  \(c_j\ge 1\) and \(U_j\in\OO \setminus \{\emptyset\}\) for
  \(j=1,\dotsc,m\).
  By the above, there is \(n_1\in\N\) with \(\vartheta \precsim n_1
  [1_{U_1}] \le n\cdot u\).
\end{proof}

\begin{example}
  \label{exm:precompactness_important}
  The precompactness assumption in
  Corollary~\ref{cor:minimal_implies_simple} is important.  For
  instance, the equivalence relation \(\RR= \N\times \N\) is a
  minimal discrete groupoid with the unit space \(X\cong \N\).  Let
  \(\B=\Bis(\RR)\) contain all bisections.  The span of~\([1_U]\)
  where \(U\subseteq \N\) is finite yields a nontrivial (and
  irregular) ideal in \(S_\B(\RR)\), as it does not
  contain~\(1_{[\N]}\).
\end{example}

For any open \(\Gr\)\nb-invariant subset \(U\subseteq X\), there are
restricted groupoids \(\Gr_U\) and~\(\Gr_{X\setminus U}\) and
inverse semigroups
\[
  \B_U\defeq\setgiven{V\in \B}{V\subseteq \Gr_U}\qquad
  \B_{X\setminus U}\defeq\setgiven{V\setminus \Gr_U}{V\in \B}
\]
of bisections of \(\Gr_U\) and~\(\Gr_{X\setminus U}\), respectively.
Here \(\B_U\) and~\(\B_{X\setminus U}\) satisfy analogues of our
standing assumption on~\(\B\).  We already related~\(\B_U\) with
ideals in the type semigroup \(S_\B(\Gr)\) in Lemma \ref{lem:order_ideals}.    Now we turn to
quotients.

\begin{lemma}
  \label{lem:restriction_to_closed_sets_quotients}
  Let \(C\subseteq X\) be a closed \(\Gr\)\nb-invariant subset.  Let
  \(f, g\in \F(\OO)\) and
  \(\B_C \defeq \setgiven{V\cap \Gr_C}{V\in \B}\).  Put
  \(\OO_C\defeq \setgiven{V\cap C}{V\in \OO}\).  The following are
  equivalent:
  \begin{enumerate}
  \item \label{enu:restriction_to_closed_sets_quotients1}%
    \(f|_C\precsim_{\B_C} g|_C\) in \(\F(\OO_C)\) ;
  \item \label{enu:restriction_to_closed_sets_quotients2}%
    for any \(k\ll f\) there is \(h\in \F(\OO_{X\setminus C})\) with
    \(k\precsim_\B g + h\).
  \end{enumerate}
  If~\(\OO\) is the set of all compact open subsets or if
  \(X\setminus C\) is a multiplier of~\(\OO\), that is, if
  \(V\in \OO\) implies \(V\setminus C\in \OO\), then the above are
  further equivalent to
  \begin{enumerate}[resume]
  \item \label{enu:restriction_to_closed_sets_quotients3}%
    there is \(h\in \F(\OO_{X\setminus C})\) with
    \(f\precsim_\B g + h\).
  \end{enumerate}
\end{lemma}

\begin{proof}
  Clearly, \ref{enu:restriction_to_closed_sets_quotients3} always
  implies~\ref{enu:restriction_to_closed_sets_quotients2}.  If~\(\OO\)
  consists of compact open subsets,
  then~\ref{enu:restriction_to_closed_sets_quotients2}
  implies~\ref{enu:restriction_to_closed_sets_quotients3} because then
  we may take \(k=f\).  To prove
  that~\ref{enu:restriction_to_closed_sets_quotients2}
  implies~\ref{enu:restriction_to_closed_sets_quotients1} let
  \(l\in \F(\OO_C)\) be such that \(l\ll f|_C\).  Then \(\cl{l}\le f\)
  and we may find \(k\in \F(\OO)\) with a precompact support and such
  that \(\cl{l}\le k \le \cl{k}\le f\).  In particular, \(k\ll f\).
  Then \ref{enu:restriction_to_closed_sets_quotients2} gives
  \(h\in \F(\OO_{X\setminus C})\) with \(k\precsim_\B g + h\).  Hence
  \(k|_C\precsim_{\B_C} g|_C\), and so there is \(b\in \F(\B_C)\) with
  \(\cl{l} \le \s_* (b)\) and \(\cl{\rg_*(b)} \le g|_C\).  This proves
  that \ref{enu:restriction_to_closed_sets_quotients2}
  implies~\ref{enu:restriction_to_closed_sets_quotients1}.

  Now assume~\ref{enu:restriction_to_closed_sets_quotients1}.  Take
  any \(k\ll f\).  Then \(k|_C\ll g|_C\), so there is
  \(b\in \F(\B)\) such that \(\cl{k}|_C \le \s_* (b)\) and
  \(\cl{\rg_*(b)}|_C \le g\), and the support of~\(b\) is precompact.
  The function \(g - \cl{\rg_*(b)}\) is lower continuous, nonnegative
  on~\(C\) and possibly negative on the compact subset
  \(\supp(\cl{\rg_*(b)})\).  Hence
  \(D\defeq (g - \cl{\rg_*(b)})^{-1}(\Z_{<0})\) is a compact subset
  disjoint from~\(C\).  Pick any precompact subset \(U\in \OO\) with
  \(D\subseteq U\subseteq \cl{U}\subseteq X\setminus C\).  Let
  \(\tilde{b}\defeq 1_U \cdot b=b|_{\rg^{-1}(U)}\in \F(\B)\).  Then
  \(\cl{k}|_C \le \s_* (b)|_C= \s_*(\tilde{b})|_C\) and
  \(\cl{\rg_*(\tilde{b})} \le g\).  Hence, replacing~\(b\)
  by~\(\tilde{b}\), we have arranged that \(\cl{k}|_C \le \s_* (b)\) and
  \(\cl{\rg_*(b)} \le g\).

  The function \(\s_*(b)-\cl{k}\) is lower semicontinuous, and it takes
  nonnegative values on~\(C\) and no negative values outside
  \(\supp(\cl{k})\).  Hence \(E\defeq (\s_*(b)-\cl{k})^{-1}(\Z_{<0})\)
  is a compact subset disjoint from~\(C\).
  Lemma~\ref{lem:choose_intermediate} gives a precompact
  \(V\in \OO\) so that
  \(E\subseteq V\subseteq \cl{V}\subseteq X\setminus C\).  Let~\(n\)
  be any number not smaller than the maximum of \(\cl{k} -\s_*(b)\)
  and replace~\(b\) by \(b+n 1_U\).  Then \(\cl{k} \le \s_* (b)\) and
  \(\cl{\rg_*(b)} \le g +n 1_V\).  Thus \(h\defeq n 1_V\) satisfies
  \(h\in \F(\OO_{X\setminus C})\) and \(k\precsim_\B g + h\).  This
  proves~\ref{enu:restriction_to_closed_sets_quotients2}.

  Now let \(X\setminus C\) be a multiplier of~\(\OO\).  Then
  \(\supp(f) \setminus C \in \OO\) and so we may change the last
  step above by taking \(V\defeq \supp(f)\setminus C\) and
  letting~\(n\) be the maximum of~\(f\).  The function
  \(h=n 1_V\in \F(\OO_{X\setminus C})\) does not depend on~\(k\),
  and our proof shows \(f\precsim_\B g + h\).  That is, in this
  case~\ref{enu:restriction_to_closed_sets_quotients1}
  implies~\ref{enu:restriction_to_closed_sets_quotients3}.
\end{proof}

\begin{corollary}
  \label{cor:quotient_monoids}
  Let \(U\subseteq X\) be an open, \(\Gr\)\nb-invariant subset and let
  \(I\defeq I_U\) be the corresponding regular ideal
  in~\(S_\B(\Gr)\).  The restriction of functions induces an order
  preserving surjective homomorphism
  \(S_\B(\Gr)/I\onto S_{\B_{X\setminus U}}(\Gr_{X\setminus U})\).
  This is an isomorphism if~\(\OO\) is the set of compact-open
  subsets or if~\(U\) is a multiplier of~\(\OO\).  In particular,
  this happens if~\(\OO\) is the whole topology or if~\(\OO\) is the
  set of all precompact subsets.
\end{corollary}

\begin{proof}
  Clearly, \(f\precsim_\B g\) implies
  \(f|_{X\setminus U}\precsim_{\B_{X\setminus U}} g|_{X\setminus
    U}\).  Hence the restriction of functions induces a well-defined
  order-preserving, surjective homomorphism
  \(S_\B(\Gr)\onto S_{\B_{X\setminus U}}(\Gr_{X\setminus U})\).
  Its kernel contains the congruence defined by the ideal
  \(I\cong S_{\B_U}(\Gr_U)\).  Therefore, it factors through a
  surjective homomorphism
  \(S_\B(\Gr)/I\onto S_{\B_{X\setminus U}}(\Gr_{X\setminus U})\).
  This is an isomorphism if and only if these two congruences
  coincide.  By
  Lemma~\ref{lem:restriction_to_closed_sets_quotients}, this always
  happens when~\(\OO\) consists of compact open subsets or when~\(U\) is
  a multiplier of~\(\OO\).
\end{proof}

\subsection{States}

We define analogues of lower semicontinuous states for the type
semigroup \(S_\B(\Gr)\), generalising
\cite{Ma:Purely_infinite_groupoids}*{Definition~4.7}.  We relate
these with traces on groupoid \(\Cst\)\nb-algebras.

\begin{definition}
  \label{def:regular_state}
  A state~\(\nu\) on \(S_\B(\Gr)\) is \emph{regular} if
  \(\nu([f])=\sup {}\setgiven{\nu([k])}{k \ll f }\) for
  all \(f\in \F(\OO)\).
\end{definition}

\begin{remark}
  \label{rem:regular_states}
  When~\(\OO\) is the set of compact open subsets, then every state is
  regular.  When \(\Gr=X\), then \(S_\B(\Gr)=\F(\OO)\) and regular
  states are the same as lower semicontinuous states.
\end{remark}

\begin{lemma}
  \label{lem:regular_state}
  Let~\(\nu\) be a state on \(S_\B(\Gr)\).  The following are
  equivalent:
  \begin{enumerate}
  \item \label{en:regular_state_1}%
    the state~\(\nu\) is regular;
  \item \label{en:regular_state_2}%
    \(\nu([1_U]) = \sup {}\setgiven{\nu([1_V])}{V\in \OO \text{ is
        precompact and }\cl{V} \subseteq U }\) for all \(U\in \OO\);
  \item \label{en:regular_state_3}%
    \(\nu([1_{\bigcup U_\alpha}]) = \sup \nu(1_{U_\alpha})\) for any
    increasing net of open subsets \(U_\alpha\in\OO\) with
    \(\bigcup U_\alpha \in \OO\).
  \end{enumerate}
\end{lemma}

\begin{proof}
  By Proposition~\ref{pro:compactly_contained}, \(1_V \ll 1_U\)
  holds if and only if~\(V\) is precompact and
  \(\cl{V} \subseteq U\).  Therefore, \ref{en:regular_state_1}
  implies~\ref{en:regular_state_2}.  The converse also holds because
  of Proposition~\ref{pro:FO} and the description of~\(\ll\) in
  Proposition~\ref{pro:compactly_contained}.  The approximation
  property in Proposition~\ref{pro:FO_continuous} shows that~\(1_U\)
  for any \(U\in \F(\OO)\) is the supremum of the directed net of
  \(1_V\ll 1_U\).  Since this supremum is just the union,
  \ref{en:regular_state_3} implies~\ref{en:regular_state_2}.
  Conversely, assume~\ref{en:regular_state_2} and let~\((U_\alpha)\)
  be a net as in~\ref{en:regular_state_3}.  Let \(x < \nu[1_U]\).
  First, \ref{en:regular_state_2} gives \(V\ll U\) with
  \(\nu[1_V] >x\).  Secondly, the definition of~\(\ll\)
  gives~\(\alpha\) with \(V \subseteq U_\alpha\) and hence
  \(\nu[1_{U_\alpha}] \ge \nu[1_V] > x\).
  Thus~\ref{en:regular_state_2} implies~\ref{en:regular_state_3}.
\end{proof}

When~\(\OO\) is the whole topology of \(X\), Ma established a
bijection between regular states on \(S_\B(\Gr)\) and what he called
\emph{groupoid dimension functions} on~\(X\) in
\cite{Ma:Purely_infinite_groupoids}*{Theorem~4.11}.
Here, instead of ``dimension function'' we prefer to use the term
``content'' as it already exists in measure theory, see
\cite{Halmos:Measure}*{p.~231}.

We will generalise the result of Ma by allowing arbitrary~\(\OO\).
More importantly, we show that regular contents always extend to
regular measures, which is a missing step
in~\cite{Ma:Purely_infinite_groupoids}.
This allows us to connect regular states on \(S_\B(\Gr)\) with traces
on the reduced groupoid \(\Cst\)\nb-algebra.
To extend contents to measures we use a result on deficient
topological measures
from~\cite{Butler:Deficient_topological_measures}, which was
communicated to us by a referee.\footnote{In the initial version of
  this article we gave a self-contained elementary proof of this based
  on an adaptation of the proof of Riesz' Theorem in \cite{Rudin}, see
  Appendix~B in arXiv:2502.17190\,v2.}

\begin{definition}
  \label{def:dimension_function}
  Let~\(\OO\) be a basis for the topology of~\(X\), which is closed
  under finite unions and contains~\(\emptyset\) (so it is a
  \(\cup\)\nb-semilattice with zero).  A \emph{content}
  on~\(\OO\) is a map \(\lambda\colon \OO\to [0,\infty]\), such that
  for all \(U,V\in \OO\)
  \begin{enumerate}
  \item \label{it:dim_function1}%
    \(\lambda(\emptyset)=0\);
  \item \label{it:dim_function2}%
    \(\lambda(U)\le \lambda(V)\) if \(U\subseteq V\);
  \item \label{it:dim_function3}%
    \(\lambda(U\cup V)\le \lambda(U)+\lambda(V)\), with equality if
    \(U\cap V=\emptyset\).
  \end{enumerate}
  We call~\(\lambda\) \emph{regular},
  if \(\lambda(V)= \sup {}\setgiven{\lambda(U)}{ U\ll V,\, U\in \OO }\) for
  every \(V\in \OO\). We say ~\(\lambda\) is \emph{\(\Gr\)-invariant} if \(\lambda(\s(U))=\lambda(\rg(U))\) for
  every open bisection \(U\in \Bis(\Gr)\) with \(\s(U)\), \(\rg(U)\in \OO\).
\end{definition}

We do not assume that regular contents are finite on precompact sets as this is the right notion to relate them to regular measures, which by our definition also need not be finite on precompact sets.

\begin{lemma}
  \label{lem:regular_measure_extension}
  Let~\(X\) be a locally compact Hausdorff space and let~\(\OO\) be
  a basis for the topology of~\(X\) closed under finite unions.
  Restriction gives a bijection between regular Borel measures
  on~\(X\) and regular contents on~\(\OO\).
\end{lemma}

\begin{proof}
  A regular Borel measure \(\mu\colon \B(X)\to [0,\infty]\)
  clearly restricts to a content    \(\lambda\colon\OO\to[0,\infty]\).
  Moreover,  for
  every  \(V\subseteq X\) open and compact \(K\subseteq V\)
  there is \(U\in \OO\) with \(K\subseteq U\ll V\) (see
  Lemma~\ref{lem:choose_intermediate}).  Thus
  \(\mu(V)=\sup {}\setgiven{\mu(U)}{U\ll V,\ U\in \OO}\) by inner regularity of~\(\mu\).  This
  implies both that~\(\lambda=\mu|_{\OO}\) is regular and that it
  determines~\(\mu\).

  Thus it suffices to show that a regular content \(\lambda\colon \OO\to [0,\infty]\) extends
  to a regular Borel measure \(\mu\) on \(X\).
  We first extend \(\lambda\)  to the whole topology \(\OO(X)\) of \(X\) by putting
  ~\(\mu_o(V)\defeq\sup\{\lambda(V):U\ll V, \, U\in \OO\}\) for
  any~\(V\in \OO(X)\).
  This is again a regular content. Indeed, regularity, monotonicity
  and additivity are straightforward.
  To see subadditivity, let \(V_1, V_2\in \OO(X)\)  and
  \(U\in \OO\) be such that \(U\ll V_1\cup V_2\). There is a
  finite cover \(\{W_i\}_{i=1}^n \subseteq \OO\) of~\(\overline{U}\)
  such that each~\(U_i\) is compactly contained either in~\(V_1\) or
  in~\(V_2\).
  Then \(U_1\defeq \bigcup_{W_i \subseteq V_1} W_i\)
  and \(U_2\defeq \bigcup_{W_i \subseteq V_2} W_i\) are elements
  of~\(\OO\) such that \(U\ll U_1\cup U_2\) and \(U_1\ll V_1\) and
  \(U_2\ll V_2\).
  Thus
  \(
  \mu(U)\le \mu(U_1\cup U_2)
  \le \mu(U_1) + \mu(U_2)
  \le\mu_o(V_1)+\mu_o(V_2).
  \)
  This implies
  \(\mu_o(V_1\cup V_2) \le
  \mu_o(V_1)+\mu_o(V_2)\), as claimed.
  Now denote by \(\mathcal{C}(X)\)  the family of all closed sets in \(X\) and
  define~\(\mu_c(F)\defeq\inf\{\mu_o(U):F\subseteq U\in \OO(X)\}\)
  for~\(F\in \mathcal{C}(X)\).
  Then \(\mu_o\) and~\(\mu_c\) agree on clopen subsets of~\(X\), and
  the map \(\mu_o\cup\mu_c\colon \OO(X)\cup  \mathcal{C}(X)\to
  [0,\infty]\) is a deficient topological measure in the sense of
  \cite{Butler:Deficient_topological_measures}*{Definition 3.1}, which
  extends~\(\lambda\).
  Since~\(\mu_o\) is subadditive, \(\mu_o\cup\mu_c\) extends to a
  regular Borel measure on~\(X\) by
  \cite{Butler:Deficient_topological_measures}*{Theorem 4.7}.
\end{proof}

\begin{theorem}
  \label{thm:Riesz_for_monoids}
  Let \(\E\colon \Cst_\red(\Gr,\LL)\to \Borel(X)\) be the canonical
  generalised expectation for the reduced \(\Cst\)\nb-algebra of a
  twisted \'etale groupoid \((\Gr,\LL)\).  Let~\(\B\) be any inverse
  semigroup basis for~\(\Gr\) and let \(\OO\defeq \B\cap 2^X\).  There are bijections between the
  sets of
  \begin{enumerate}
  \item \label{it:thm_traces}%
    lower semicontinuous traces~\(\tau\) on~\(\Cst_\red(\Gr,\LL)\)
    that factor through~\(\E\) \textup{(}see Proposition
    \textup{\ref{prop:states_and_traces})};
  \item \label{it:thm_measures}%
    \(\Gr\)\nb-invariant regular Borel measures~\(\mu\) on~\(X\);
  \item \label{it:thm_contents}%
    \(\Gr\)\nb-invariant regular contents~\(\lambda\) on~\(\OO\);

  \item \label{it:thm_states}%
    regular states~\(\nu\) on~\(S_\B(\Gr)\).
  \end{enumerate}
  The bijections are characterised by the conditions
  \(\nu([1_U])=\mu(U)\) for \(U\in \OO\),
  \(\tau(f)=\int_X \E(f) \,\diff\mu\) for
  \(f\in \Cst_\red(\Gr,\LL)^+\), and $\lambda=\mu|_{\OO}$.
  For the corresponding objects
  \begin{enumerate}[label=\textup{(\alph*)}]
  \item \label{it:a}%
    \(\tau\) is a tracial state if and only if~\(\mu\) is a
    probability measure if and only if
    \(\sup {}\setgiven{\nu([1_U])}{U\in \OO}=1\);
  \item\label{it:b}%
    \(\tau\) is semifinite if and only if~\(\mu\) is locally finite
    \textup{(}a Radon measure\textup{)} if and only if
    \(\nu([1_U])<\infty\) for every precompact \(U\in \OO\);
  \item\label{it:c}%
    if \(\Cst_\red(\Gr,\LL)=\Cst_\ess(\Gr,\LL)\), then~\(\tau\) is
    faithful if and only if~\(\mu\) has full support if and only
    if~\(\nu\) is faithful.
  \end{enumerate}
\end{theorem}

\begin{proof}
  The correspondence between the objects in \ref{it:thm_traces}
  and~\ref{it:thm_measures} is proved in
  Proposition~\ref{prop:states_and_traces}.
  Lemma \ref{lem:regular_measure_extension} establishes a bijection
  between regular Borel measures~\(\mu\) and regular
  contents~\(\lambda\).
  If~\(\mu\) is \(\Gr\)\nb-invariant, then so is~\(\lambda\).
  The converse follows from Lemma~\ref{lem:G-invariant_measure}
  because~\(\B\) covers~\(\Gr\).
  This explains the bijection between the objects in
  \ref{it:thm_measures} and~\ref{it:thm_contents}.
  We show now that a regular Borel measure~\(\mu\) defines a regular
  state on the type semigroup.
  Let \(f=\sum_{i=1}^n 1_{U_i}\) and
  \(g=\sum_{j=1}^m 1_{V_j}\in \F(\OO)\) with \(f\precsim_\B g\).
  Let \(k=\sum_{i=1}^n 1_{K_i}\) with compact \(K_i\subseteq U_i\)
  for \(i=1,\dotsc,n\).
  Then there are bisections \(W_1,\dotsc,W_N\in\B\) such that \(k \le
  \sum_{i=1}^N 1_{\s(W_i)}\) and \(\sum_{i=1}^N 1_{\rg(W_i)} \le g\).
  Since~\(\mu\) is \(\Gr\)\nb-invariant,
  \[
    \int k \,\diff\mu
    \le \sum_{i=1}^N \mu(\s(W_i))
    = \sum_{i=1}^N \mu(\rg(W_i))
    \le \int g \,\diff\mu.
  \]
  This implies \(\int f \,\diff\mu \le \int g\,\diff\mu\)
  because~\(\mu\) is regular.
  This shows that the formula \(\nu([f])\defeq \int f \,\diff\mu\)
  gives a well-defined, order-preserving map \(\nu\colon S_\B(\Gr)\to
  [0,+\infty]\).
  So~\(\nu\) is a state that satisfies \(\mu(U)=\nu([1_U])\) for all
  \(U\in \OO\).
  Since~\(\mu\) is regular, so is~\(\nu\) by
  Lemma~\ref{lem:regular_state}.

  Conversely, let~\(\nu\) be a state on \(S_\B(\Gr)\).
  Define \(\lambda\colon \OO\to [0,\infty]\) by \(\lambda(U)\defeq\nu([1_U])\) for
  \(U\in\OO\).
  Clearly, \(\lambda(\emptyset)=\nu([0])=0\) and \(\lambda(U_1)\le
  \lambda(U_2)\) if \(U_1\subseteq U_2\), because then \(1_{U_1}\le
  1_{U_2}\), so that \(1_{U_1}\precsim_\B 1_{U_2}\).
  Similarly, \(1_{U_1\cup U_2}\le 1_{U_1}+1_{U_2}\) implies
  \(\lambda(U_1\cup U_2)\le \lambda(U_1)+\lambda(U_2)\), and both of
  these are equalities when \(U_1\cap U_2=\emptyset\).
  Hence~\(\lambda\) is a content.
  Since~\(\nu\) is regular as a state, \(\lambda\) is regular as a
  content.
  Let~\(\mu\) be the corresponding regular measure on~\(X\).
  If \(V\in \B\), then \(1_{\s(V)} \approx_\B 1_{\rg(V)}\) and thus
  \(\mu(\s(V))=\mu(\rg(V))\).
  Since~\(\B\) covers~\(\Gr\), the measure~\(\mu\) is
  \(\Gr\)\nb-invariant by Lemma~\ref{lem:G-invariant_measure}.

  This proves the main part of the assertion.
  The statements \ref{it:a}--\ref{it:c} follow from
  Proposition~\ref{prop:states_and_traces}.
\end{proof}

\begin{remark}
  \label{rem:Riesz_for_monoids}
  Theorem~\ref{thm:Riesz_for_monoids} applied to \(\Gr=X\) gives a
  bijection between lower semicontinuous weights on \(\Cont_0(X)\),
  regular Borel measures on~\(X\), and regular (or, equivalently,
  lower semicontinuous) states on~\(\F(\OO)\).
\end{remark}

\begin{corollary}
  \label{cor:non_purely_infinite}
  Assume that~\(S_\B(\Gr)\) admits a nontrivial regular state and
  that~\(\OO\) consists of \(\sigma\)\nb-compact subsets.  Then
  there is a nonzero hereditary subalgebra of
  \(\Cst_\red(\Gr,\LL)\) that has a tracial state.  If
  \(\Cst_\red(\Gr,\LL)\) is simple, it has a faithful,
  semifinite, lower semicontinuous trace.
  If~\(X\) is compact and \(\Cst_\red(\Gr,\LL)\) is simple, then
  \(\Cst_\red(\Gr,\LL)\) has a faithful tracial state.
\end{corollary}

\begin{proof}
  Let~\(\nu\) be a nontrivial regular state on~\(S_\B(\Gr)\).  Then
  there is \(U\in\OO\) with \(\nu(1_U)\in(0,\infty)\).
  Dividing~\(\nu\) by \(\nu(1_U)\), we may assume that
  \(\nu(1_U)=1\).  There is \(f_U\in \Cont_0(X)^+\) with open
  support~\(U\) because~\(U\) is \(\sigma\)\nb-compact.  By
  Lemma~\ref{lem:hereditary_subalgebras_essential},
  \(f_U\Cst_\red(\Gr,\LL)f_U\cong \Cst_\red(\Gr_U,\LL_U)\) where
  \((\Gr_U,\LL_U)\) is the twisted groupoid restricted to~\(U\).
  Since \(U\in \OO\subseteq \B\) and~\(\B\) is an inverse semigroup,
  \(\B_U\defeq \setgiven{W\in \B}{W\subseteq \Gr_U}\) is equal to
  \(\setgiven{UWU}{W\in\B}\).  Thus~\(\B_U\) is an inverse semigroup
  basis for~\(\Gr_U\).  Its set of idempotents is
  \(\OO_U\defeq\setgiven{V\in \OO }{V\subseteq U}\).  We have a
  natural isomorphism of ordered monoids from \(S_{\B_U}(\Gr_U)\)
  onto the submonoid of \(S_\B(\Gr)\) generated by~\(\OO_U\) (this
  is the isomorphism from Lemma~\ref{lem:order_ideals} if~\(U\) is
  \(\Gr\)\nb-invariant).  Assuming the identification
  \(S_{\B_U}(\Gr_U)\subseteq S_\B(\Gr)\), \(\nu\) restricts to a
  regular state on \(S_{\B_U}(\Gr_U)\) with
  \(\sup {}\{\nu(1_V) | V\in\OO_U \}=\nu(1_U)=1\).  Now
  Theorem~\ref{thm:Riesz_for_monoids} gives a tracial state on
  \(\Cst_\red(\Gr_U,\LL_U)\cong f_U \Cst_\red(\Gr,\LL)f_U\).  If
  \(\Cst_\red(\Gr,\LL)\) is simple, then
  \(f_U\Cst_\red(\Gr,\LL)f_U\) is Morita--Rieffel equivalent
  to~\(\Cst_\red(\Gr,\LL)\).  So the tracial state on
  \(f_U \Cst_\red(\Gr,\LL)f_U\) is necessarily faithful and it
  transfers to a faithful semifinite lower semicontinuous trace
  on~\(\Cst_\red(\Gr,\LL)\) (see~\cite{Combes-Zettl:Order_traces}).
  The latter trace is finite if \(\Cst_\red(\Gr,\LL)\) is unital.
\end{proof}

Next we describe a ``regularisation procedure'' that extends
\cite{Ma:Purely_infinite_groupoids}*{Proposition~4.8}.

\begin{proposition}
  \label{prop:regularization_of_states}
  Every state~\(\nu\) on the monoid \(S_\B(\Gr)\) induces a regular
  state determined by
  \[
    \cl{\nu}([1_U])=\sup {}\setgiven{\nu([1_V])}{V\in \OO
      \text{ is precompact and }\cl{V} \subseteq U },
  \]
  for all \(U\in \OO\).  In particular, \(\nu\) is regular if and
  only if \(\nu=\cl{\nu}\).
\end{proposition}

\begin{proof}
  We define the map \(\cl{\nu}\colon S_\B(\Gr) \to [0,+\infty]\) by
  the formula
  \[
    \cl{\nu}([f])
    \defeq \sup {}\setgiven{\nu([k ])}{ k\in \F(\OO) \text{ and } k\ll  f }.
  \]
  Clearly, \(\cl{\nu}([f])+ \cl{\nu}([g]) \le \cl{\nu}([f]+[g])\)
  for any \(f,g\in \F(\OO)\).  To show the reverse inequality take
  any \(k\ll f+g\).  By
  Corollary~\ref{cor:intermediate_way_below_improved} there are
  \(k_1\ll f\) and \(k_2\ll g\) such that \(k\ll k_1 +k_2\)
  in~\(\F(\OO)\).  Then
  \[
    \nu([k])\le \nu([k_1 +k_2])
    = \nu([k_1]) +\nu([k_2])
    \le \cl{\nu}([f])+ \cl{\nu}([g]).
  \]
  Therefore, \(\cl{\nu}([f]+[g])\le \cl{\nu}([f])+ \cl{\nu}([g])\).
  This proves that~\(\cl{\nu}\) is additive.

  Let \(f\precsim g\).  Then for any \(k\ll f\) there is
  \(b\in \F(\B)\) with \(k \ll \s_* (b)\) and \(\rg_*(b) \ll g\).  Then
  \[
    \nu([k])
    \le \nu([\s_* (b)])
    = \nu([\rg_* (b)])
    \le \cl{\nu}([g]).
  \]
  Hence \(\cl{\nu}([f])\le \cl{\nu}([g])\).  That is, \(\cl{\nu}\)
  is monotone (order-preserving).  So~\(\cl{\nu}\) is a state.  To
  see that~\(\cl{\nu}\) is regular, take any \(U\in\OO\).  For every
  \(V\in\OO\) with \(V\ll U\) there is \(W\in\OO\) with
  \(V\ll W\ll U\) by Corollary~\ref{cor:intermediate_way_below}.
  Then \(\nu(V)\le \cl{\nu}(1_W)\), and hence
  \(\cl{\nu}([1_U])=\sup_{V\ll U}\nu([1_V])\le \sup_{W\ll U}
  \cl{\nu}([1_W])\).  Thus
  \(\cl{\nu}([1_V])=\sup_{W\ll U}\cl{\nu}([1_W])\) and
  so~\(\cl{\nu}\) is regular by Lemma~\ref{lem:regular_state}.
\end{proof}

The induced state~\(\cl{\nu}\) may become trivial even if~\(\nu\) is
not.  The next two examples show that this may happen even when~\(X\)
is compact or when~\(\Gr\) is minimal.  The next lemma says, however,
that it cannot happen if we have both of these properties.

\begin{example}
  \label{exa:trivial_regular_state}
  Let \(\Gr=X=\mathbb{T}\) and put
  \[
    \OO\defeq
    \setgiven{U\subseteq\mathbb{T} \text{ open}}
    {\cl{U}\subset \mathbb{T}\setminus\{1\} \text{ or }
      U=\mathbb{T}\setminus\{1\} \text{ or } 1\in U}.
  \]
  Define
  \(\nu([1_U])\) to be~\(0\) if
  \(\cl{U}\subseteq\mathbb{T}\setminus\{1\}\), \(1\) if
  \(U=\mathbb{T}\setminus\{1\}\), and~\(\infty\) otherwise.
  Then~\(\cl{\nu}\) is~\(0\) if
  \(\cl{U}\subset\mathbb{T}\setminus\{1\}\) and~\(\infty\)
  otherwise.  Hence~\(\overline{\nu}\) is trivial, although~\(\nu\)
  is not.  Since all open subsets in~\(X\) are \(\Gr\)\nb-invariant,
  there are infinitely many of them.
\end{example}

\begin{example}
  Let \(\Gr=\RR= \N\times \N\) be the full equivalence relation (a
  minimal principal discrete groupoid) and let~\(\B\) consist of all
  finite bisections and the unit space \(X\cong \N\).  Let~\(\nu\)
  be the (finite) state on \(S_\B(\RR)\) defined by \(\nu([1_X])=1\)
  and \(\nu([1_U])=0\) when \(U\subseteq \N\) is finite.  Then
  \(\overline{\nu}\equiv 0\).
\end{example}

\begin{lemma}
  \label{lem:nontrivial_state_finitely_many_invariant}
  Assume that~\(\OO\) consists of precompact subsets and that~\(X\)
  has only finitely many \(\Gr\)\nb-invariant open subsets.  Then the
  regular state associated to a nontrivial state on~\(S_\B(\Gr)\) as
  in Proposition~\textup{\ref{prop:regularization_of_states}} is again
  nontrivial.
\end{lemma}

\begin{proof}
  Let~\(\nu\) be a nontrivial state for~\(S_\B(\Gr)\) and
  let~\(\cl{\nu}\) be the regular state it induces.  By the
  assumption there is an ascending sequence
  \(\emptyset=X_0\subseteq X_1\subseteq X_2\subseteq \dotsb
  \subseteq X_n=X\) of \(\Gr\)\nb-invariant open subsets of~\(X\)
  such that for each \(k=1,\dotsc,n-1\) there are no open
  \(\Gr\)\nb-invariant subsets between \(X_k\) and~\(X_{k+1}\);
  equivalently, the restriction of~\(\Gr\) to
  \(X_{k+1}\setminus X_k\) is minimal.  The proof is by induction
  on~\(n\).

  For \(n=1\), the groupoid~\(\Gr\) is minimal, and then
  \(S_\B(\Gr)\) is simple by
  Corollary~\ref{cor:minimal_implies_simple}.  Thus~\(\nu\) is
  faithful and finite by
  Lemma~\ref{lem:simplicity_implies_faithfulness_of_states}.  This
  readily implies that~\(\cl{\nu}\) is
  nontrivial.

  Now let \(n>1\) and assume that the assertion holds for \(n-1\).
  Put \(\OO_0\defeq\setgiven{U\in \OO}{\nu([1_U])=0}\) and
  \(\OO_{n-1}\defeq\setgiven{U\in \OO}{U\subseteq X_{n-1}}\).  If
  \(\OO_{n-1}\not\subseteq \OO_0\), then~\(\cl{\nu}\) is nontrivial
  by the induction hypotheses applied to \(\Gr_{X_{n-1}}\)
  and~\(\OO_{n-1}\) (we may identify
  \(S_{\B_{X_{n-1}}}(\Gr_{X_{n-1}})\) with an ideal in \(S_\B(\Gr)\)
  by Lemma~\ref{lem:order_ideals}).  Thus we may assume that
  \(\OO_{n-1}\subseteq \OO_0\).  Then \(X_{n-1}=\bigcup \OO_0\)
  because \(\bigcup \OO_0\) is a nontrivial open
  \(\Gr\)\nb-invariant subset containing~\(X_{n-1}\) and~\(X_{n-1}\)
  is a maximal proper subset with these properties.  Hence
  \(\OO_{n-1}=\OO_0\).  Accordingly, for each \(U\in \OO\) we get
  \[
    U\not\subseteq X_{n-1} \quad \Longleftrightarrow\quad 0<\nu([U]).
  \]
  Now take any \(U\in \OO\) with \(0<\nu([U])<\infty\).  By the
  above equivalence, there is a compact subset \(K\subseteq U\) with
  \(K\not\subseteq X_{n-1}\), and hence also a precompact
  \(V\in \OO\) such that
  \(K\subseteq V \subseteq \cl{V}\subseteq U\) (see
  Lemma~\ref{lem:choose_intermediate}).  Then
  \(0<\nu([1_V])\le \cl{\nu}([1_U]) \le \nu([1_U])<\infty\).
  Hence~\(\cl{\nu}\) is nontrivial.
\end{proof}

\begin{proposition}[Tarski's Theorem for regular states]
  \label{prop:regular_Tarski}
  Assume that the type semigroup \(S_\B(\Gr)\) is almost
  unperforated.  An element \([f]\in S\) is not paradoxical if and
  only if there is a regular state
  \(\nu\colon S_\B(\Gr)\to [0,\infty]\) with \(\nu([f])=1\).
\end{proposition}

\begin{proof}
  We first show a sufficient condition for an element to be
  paradoxical.  Assume that for all \(k\ll f\) and all
  states~\(\nu\) on \(S_\B(\Gr)\) with \(\nu([f])=1\) we have
  \(\nu(2[k])<\nu([f])\).  Then by
  Theorem~\ref{thm:Handelman's_tarski}, for each \(k\ll f\) there is
  \(n\in \N\) such that \((n+1)2 [k]\precsim n [f]\).  This implies
  \(2k\precsim f\) because \(S_\B(\Gr)\) is almost unperforated.
  This implies that \(2f\precsim f\) (see
  Remark~\ref{rem:compactly_below_relation2}), so~\(f\) is
  paradoxical.

  Thus if \([f]\in S_\B(\Gr)\) is not paradoxical, then there is
  \(k\ll f\) and a state~\(\nu\) with \(\nu([f])=1\) such that
  \(\nu([f])\le 2 \nu([k])\).  Since \(k\precsim f\), this implies
  that \(0<\nu([f])/2 \le \nu([k])\le \nu([f])=1\).  Therefore,
  the regular state~\(\cl{\nu}\) defined in
  Proposition~\ref{prop:regularization_of_states} satisfies
  \(0< \cl{\nu}([f]) <\infty\).  Thus normalising~\(\cl{\nu}\)
  in~\([f]\) gives the desired state.
\end{proof}

\begin{corollary}
  \label{cor:regular_states_for_unperforated}
  If the type semigroup \(S_\B(\Gr)\) is almost unperforated, then
  \(S_\B(\Gr)\) admits a nontrivial state if and only if it admits a
  regular nontrivial state.
\end{corollary}

\begin{proof}
  Combine Corollary~\ref{cor:original_Tarski} and
  Proposition~\ref{prop:regular_Tarski}.
\end{proof}

\begin{corollary}
  \label{cor:pure_infinite_implies_paradoxical}
  Assume that~\(\OO\) consists of \(\sigma\)\nb-compact subsets and
  that \(S_\B(\Gr)\) is almost unperforated.
  If for some twist~\(\LL\) the algebra \(\Cst_\red(\Gr,\LL)\) is
  purely infinite, then \(S_\B(\Gr)\) is purely infinite.
\end{corollary}

\begin{proof}
  Assume \(S_\B(\Gr)\) is not purely infinite.  By almost
  unperforation, \(S_\B(\Gr)\setminus\{0\}\) contains an element
  that is not paradoxical.  Thus by
  Proposition~\ref{prop:regular_Tarski}, \(S_\B(\Gr)\) has a
  nontrivial state.  Hence by
  Corollary~\ref{cor:non_purely_infinite} a hereditary subalgebra of
  \(\Cst_\red(\Gr,\LL)\) has a tracial state, which implies that
  \(\Cst_\red(\Gr,\LL)\) is not purely infinite.
\end{proof}

\begin{corollary}
  Let~\(\B\) be an inverse semigroup basis for~\(\Gr\) such
  that~\(\OO\) consists of compact open subsets or~\(\OO\) is the whole
  topology or~\(\OO\) consists of all precompact subsets.  An element
  \([f]\in S_\B(\Gr)\) is properly infinite if and only
  if~\([f|_C]\) is infinite in \(S_{\B_C}(\Gr_C)\) for every closed
  \(\Gr\)\nb-invariant \(C\subseteq X\) with
  \(\supp(f)\cap C \neq \emptyset\).
\end{corollary}

\begin{proof}
  Let \(I\defeq I_{X\setminus C}\) be the ideal in \(S_\B(\Gr)\)
  corresponding to the open \(\Gr\)\nb-invariant subset
  \(X\setminus C\).  Thus \([f]\not\in I\) if and only if
  \(\supp(f)\cap C \neq \emptyset\).
  Corollary~\ref{cor:quotient_monoids} provides a surjective
  homomorphism
  \(S_\B(\Gr)/I\onto S_{\B_C}(\Gr_C)\).  So if
  \(\supp(f)\cap C \neq \emptyset\) and~\([f]\) is properly infinite
  in \(S_\B(\Gr)\), then its image in \(S_{\B_C}(\Gr_C)\) is
  properly infinite.  Conversely, if \([f]\) is not properly
  infinite, then by Lemma~\ref{lem:infiniteness_ideal},
  \([f]\not\in I([f])\) and the image of~\([f]\) is finite in
  \(S_\B(\Gr)/I([f])\).  It follows from the definition of
  \(I([f])\) that this is a regular ideal.  Hence
  \(I([f])\cong S_{\B_U}(\Gr_U)\) for an open \(\Gr\)\nb-invariant
  subset \(U\subseteq X\) by Lemma~\ref{lem:order_ideals}.  Put
  \(C\defeq X\setminus U\).  Then \(\supp(f)\cap C \neq \emptyset\)
  because \([f]\not\in I([f])\), and~\([f|_C]\) is finite in
  \(S_{\B_C}(\Gr_C)\) because
  \(S_\B(\Gr)/I([f])\cong S_{\B_C}(\Gr_C)\) by
  Corollary~\ref{cor:quotient_monoids}.
\end{proof}

\subsection{Type semigroups for ample groupoids}

For a while, we restrict attention to the case where~\(\Gr\) is a
(not necessarily Hausdorff) \emph{ample groupoid}.  Equivalently,
\(\Gr\) is an \'etale groupoid with totally disconnected unit space
\(X\defeq \Gr^0\).  We also assume the inverse semigroup
basis~\(\B\) to be the family of all compact open bisections.  Thus
\[
  \OO\defeq\setgiven{U\subseteq X}{U\text{ is compact open}}
\]
forms a basis for topology of~\(X\) (and a ring of sets).  We are
going to relate our type semigroup \(S_\B(\Gr)\) to the type semigroup
\(S(\Gr)\) introduced in \cites{Boenicke-Li:Ideal,
  Rainone-Sims:Dichotomy} (where~\(\Gr\) is needlessly assumed
Hausdorff).
We note that this specific ``ample'' inverse semigroup~\(\B\) is an
instance of a Boolean inverse semigroup, and the type semigroup
\(S(\Gr)\) coincides with the type monoid associated to the
semigroup~\(\B\), see~\cite{Wehrung:Monoids_Boolean}.

In~\cite{Boenicke-Li:Ideal}, \(S(\Gr)\) is defined as the quotient of
the free semigroup \(\Free_\OO\) by the equivalence relation~\(\sim\),
where \((U_1,U_2,\dotsc,U_n) \sim (V_1,\dotsc,V_m)\) if there is a
collection of compact open bisections \(W_1,\dotsc, W_l\) in~\(\Gr\),
\(l\in \N\), and natural numbers \(n_1,\dotsc, n_l\),
\(m_1,\dotsc, m_l\) such that
\[
  \bigsqcup_{j=1}^n U_i\times \{j\}
  = \bigsqcup_{i=1}^l\s(W_i)\times \{n_i\}
  \text{ and }
  \bigsqcup_{i=1}^m V_i\times \{i\}
  = \bigsqcup_{i=1}^nr(W_i)\times \{m_i\}.
\]
This is a congruence relation on~\(\Free_\OO\), and the quotient
semigroup \(S(\Gr)\defeq \Free_\OO/{\sim}\) is an abelian monoid
called the type semigroup of~\(\Gr\) in
\cite{Boenicke-Li:Ideal}*{Definition~5.1}.  An equivalent structure
was defined in~\cite{Rainone-Sims:Dichotomy} as the quotient of the
abelian monoid \(\F(\OO)=\Contc(X,\Z)^+\) by the equivalence
relation~\(\sim_\Gr\), where
\begin{equation}
  \label{eq:Rainone_Sims_equiv}
  f\sim_\Gr g \quad \Longleftrightarrow\quad \exists_{b\in \F(\B)}\qquad  f= \s_* (b), \qquad    \rg_*(b) =g.
\end{equation}
It is shown in~\cite{Rainone-Sims:Dichotomy} that this is a
congruence relation on the monoid \(\Contc(X,\Z)^+\) and that
\[
  S(\Gr)\cong \Contc(X,\Z)^+/{\sim_\Gr}.
\]
with the isomorphism sending the equivalence class of~\(1_U\) to the
equivalence class of~\(U\), for every \(U\in \OO\).
Write~\([f]\) for the \(\sim_\Gr\)\nb-equivalence class of
\(f\in \Contc(X,\Z)^+\).  Then the addition on~\(S(\Gr)\) is defined
by \([f]+[g] \defeq [f+g]\).  The monoid structure induces an
algebraic preorder: write \([f] \le [g]\) if there is
\(h \in \Contc(X,\N)\) with \([f]+[h]=[g]\).

Example~\ref{ex:type_semigroup_for_Cuntz_algebras} shows that
\(S_\B(\Gr)\) and~\(S(\Gr)\) may fail to be isomorphic.  Nevertheless,
they are closely related:

\begin{proposition}
  \label{prop:type_vs_groupoid_semigroup}
  Assume that~\(\Gr\) is ample and let~\(\B\) be the family of
  compact open bisections.
  \begin{enumerate}
  \item \label{item:type_vs_groupoid_semigroup1}%
    The identity map on \(\Contc(X,\Z)^+\) induces a surjective
    monoid homomorphism \(\Psi\colon S(\Gr)\onto S_\B(\Gr)\).
    In fact, \(S_\B(\Gr)\cong \widetilde{S(\Gr)}\) and~\(\Psi\) is the
    quotient map \(S(\Gr)\onto \widetilde{S(\Gr)}\).
    In particular, \([f]\le [g]\) if and only if \(\Psi([f])\precsim
    \Psi([g])\) for all \(f,g\in \Contc(X,\Z)^+\).
  \item \label{item:type_vs_groupoid_semigroup1.33}%
    \(S(\Gr)\) is (almost) unperforated or has plain paradoxes if
    and only if
    \(S_\B(\Gr)\) has this property.  Moreover, \([f]\) is
    paradoxical, infinite or properly infinite in~\(S(\Gr)\) if and
    only if \(\Psi([f])\) has this property in~\(S_\B(\Gr)\).
  \item \label{item:type_vs_groupoid_semigroup1.66}%
    There are natural bijections between ideals in~\(S(\Gr)\) and
    in~\(S_\B(\Gr)\).
  \item \label{item:type_vs_groupoid_semigroup2}%
    There are natural bijections between states on \(S(\Gr)\)
    and~\(S_\B(\Gr)\) and regular \(\Gr\)\nb-invariant Borel
    measures on~\(X\).
  \item \label{item:type_vs_groupoid_semigroup3}%
    \(S(\Gr)\) admits a nontrivial, faithful or finite state,
    respectively, if and only if \(S_\B(\Gr)\) admits such a state,
    if and only if there is a nontrivial, fully supported or locally
    finite regular \(\Gr\)\nb-invariant Borel measure on~\(X\).

  \item \label{item:type_vs_groupoid_semigroup4}%
    If there is a \(\Gr\)\nb-invariant Radon measure on~\(X\) with
    full support \textup{(}equivalently, \(S(\Gr)\) or \(S_\B(\Gr)\)
    admits a faithul finite state\textup{)}, then~\(\Psi\) is an
    isomorphism: \(S(\Gr)\cong S_\B(\Gr)\).
  \end{enumerate}
\end{proposition}

\begin{proof}
  \ref{item:type_vs_groupoid_semigroup1}: Clearly, \(f\sim_\Gr g\)
  implies \(f\approx_\B g\) for \(f, g\in \Contc(X,\Z)^+\).  This
  gives the surjective homomorphism \(\Psi\colon S(\Gr)\onto
  S_\B(\Gr)\).  The
  equivalence of \([f]\le [g]\) and \(\Psi([f])\precsim \Psi([g])\)
  is proved in
  \cite{Ma:Purely_infinite_groupoids}*{Proposition~5.11}.  This
  implies \(S_\B(\Gr)\cong \widetilde{S(\Gr)}\).

  \ref{item:type_vs_groupoid_semigroup1.33}: This readily follows
  from~\ref{item:type_vs_groupoid_semigroup1}, see Remarks
  \ref{rem:about_infnite_elements}
  and~\ref{rem:almost_unperforation_for_ordered_quotient} and recall
  that \(S(\Gr)\) is conical.

  \ref{item:type_vs_groupoid_semigroup1.66}: This follows from
  \(S_\B(\Gr)\cong \widetilde{S(\Gr)}\) and a general fact, namely,
  Remark~\ref{rem:trvial_states_and_ideals}.

  \ref{item:type_vs_groupoid_semigroup2}: The bijection between the
  states on \(S_\B(\Gr)\) and regular \(\Gr\)\nb-invariant Borel
  measures on~\(X\) is given by Theorem~\ref{thm:Riesz_for_monoids}.
  The bijection between states on \(S(\Gr)\) and~\(S_\B(\Gr)\) is
  described in Remark~\ref{rem:trvial_states_and_ideals}.

  \ref{item:type_vs_groupoid_semigroup3}: The correspondence between
  nontrivial and finite states readily follows
  from~\ref{item:type_vs_groupoid_semigroup2}.  When it comes to
  faithful states, one also needs to use the second part
  of~\ref{item:type_vs_groupoid_semigroup1} and that~\(S(\Gr)\) is
  conical (see Remark~\ref{rem:faithful_implies_conical}).

  \ref{item:type_vs_groupoid_semigroup4}: This follows from
  Remark~\ref{rem:faithful_states_algebraic_order}, as
  by~\ref{item:type_vs_groupoid_semigroup3} the algebraically ordered
  monoid \(S(\Gr)\) admits a faithful finite state.
\end{proof}

\begin{remark}
  It follows that in the ample case one may replace the semigroup
  \(S_\B(\Gr)\) with~\(S(\Gr)\) in many results.  This may be useful
  because~\(S(\Gr)\) is always a conical refinement monoid (see
  \cite{Wehrung:Monoids_Boolean}*{Corollary 4.1.4}
  and~\cite{Ara_Bonicke_Bosa_Li:type_semigroup}).  There is a vast
  literature on such monoids.  In particular, it is known that all
  finitely generated conical refinement monoids are of the form
  \(S(\Gr)\) where~\(\Gr\) is the transformation groupoid of a
  discrete group action on a totally disconnected space (see
  \cite{Wehrung:Monoids_Boolean}*{Theorem 4.8.9}), or where~\(\Gr\) is
  the groupoid of an adaptable separated graph (see
  \cite{Ara-Bosa-Pardo-Sims:Separated_graphs}*{Corollary 7.6}).
  We do not know whether \(S_\B(\Gr)\) is a refinement monoid in general.
\end{remark}

Finally, we comment on the groupoid version of the strict comparison
mentioned in Remark~\ref{rem:dynamical_comparison_for_semigroups}.
The regular ideal in \(S_\B(\Gr)\) generated by~\([g]\) is the ideal
corresponding to the smallest open \(\Gr\)\nb-invariant subset
containing \(\supp(g)\).  This set is \(\rg(\Gr\supp(g))\).
So~\([f]\) is in this ideal if and only if
\(\supp(f)\subseteq \rg(\Gr\supp(g))\).  We generalise the
comparison property defined for ample groupoids
in~\cite{Ara_Bonicke_Bosa_Li:type_semigroup} (also called
\emph{groupoid comparison} in~\cite{Ma:Purely_infinite_groupoids},
which works well only in the minimal case) to general \'etale
groupoids as follows:

\begin{definition}
  \label{def:regular_comparison}
  We say that the groupoid~\(\Gr\) has \emph{dynamical comparison}
  if the following happens for all open subsets \(U,V\subseteq X\):
  if \(U\subseteq \rg(\Gr V)\) and \(\mu(U)<\mu(V)\) for all regular
  \(\Gr\)\nb-invariant measures~\(\mu\) on~\(X\) with
  \(\mu(V)< \infty\), then \(1_U\precsim 1_V\), that is, for any
  \(K\ll U\) there are open bisections \(W_1,\dotsc,W_N\), such that
  \(K\subseteq \bigcup_{k=1}^N \s(W_k)\) and
  \(\bigsqcup_{k=1}^N \rg(W_k) \subseteq V\).  We say that~\(\Gr\)
  has \emph{stable dynamical comparison} if its stabilisation
  \(\K(\Gr)=\Gr\times \N\times \N\) has dynamical comparison.
\end{definition}

\begin{remark}
  \label{rem:dynamical_comparison_almost_unperforated}
  By the preceding discussion and
  Theorem~\ref{thm:Riesz_for_monoids}, \(\Gr\) has dynamical
  comparison if and only if the following happens for any
  \([1_U],[1_V]\in S_\B(\Gr)\): if \([1_U]\) is contained in the
  regular ideal generated by \([1_V]\) and \(\nu([1_U])<\nu([1_V])\)
  for all regular states~\(\nu\) on~\(S_\B(\Gr)\) with
  \(\nu([1_V])<\infty\), then \([1_U]\precsim [1_V]\).  Since
  \(S_\B(\Gr)\cong S_{\K(\B)}(\Gr \times \N^2)\) (by
  Proposition~\ref{prop:stabilisation_vs_type_semigroup}), stable
  dynamical comparison for~\(\Gr\) is equivalent to the above
  condition for all elements in~\(S_\B(\Gr)\).  So this is a
  ``regular almost unperforation''.
\end{remark}

\begin{proposition}
  \label{prop:dynamical_comparison_ample}
  Assume~\(\Gr\) is ample and let~\(\B\) be the family of all
  compact open bisections.  The following are equivalent:
  \begin{enumerate}
  \item\label{enu:dynamical_comparison_ample1}%
    the groupoid \(\Gr\) has stable dynamical comparison;
  \item\label{enu:dynamical_comparison_ample2}%
    the monoid \(S_\B(\Gr)\) is almost unperforated;
  \item\label{enu:dynamical_comparison_ample3}%
    the monoid \(S(\Gr)\) is almost unperforated.
  \end{enumerate}
  If~\(\Gr\) is minimal and \(\sigma\)\nb-compact, then the above
  are further equivalent to
  \begin{enumerate}[resume]
  \item\label{enu:dynamical_comparison_ample4}%
    the groupoid \(\Gr\) has dynamical comparison.
  \end{enumerate}
\end{proposition}

\begin{proof}
  For ample groupoids, the adjective ``regular'' is vacuous.  Thus
  Remark~\ref{rem:dynamical_comparison_almost_unperforated} implies
  \ref{enu:dynamical_comparison_ample1}\(\Leftrightarrow\)\ref{enu:dynamical_comparison_ample2}
  and Proposition
  \ref{prop:type_vs_groupoid_semigroup}.\ref{item:type_vs_groupoid_semigroup1.33}
  implies
  \ref{enu:dynamical_comparison_ample2}\(\Leftrightarrow\)\ref{enu:dynamical_comparison_ample3}.
  The implication
  \ref{enu:dynamical_comparison_ample1}\(\Rightarrow\)\ref{enu:dynamical_comparison_ample4}
  is obvious.

  Assume now that~\(\Gr\) is minimal and \(\sigma\)\nb-compact.  If
  the unit space~\(X\) is compact, the implication
  \ref{enu:dynamical_comparison_ample4}\(\Rightarrow\)\ref{enu:dynamical_comparison_ample1}
  is proved in
  \cite{Ara_Bonicke_Bosa_Li:type_semigroup}*{Proposition~3.10}
  (assuming the groupoids to be Hausdorff, but this assumption is not
  used to prove this proposition).  We now explain why the compactness
  assumption is not needed.  So
  assume~\ref{enu:dynamical_comparison_ample4}, where~\(X\) need not
  be compact.  Then we may pick any nonempty compact open subset
  \(K\subseteq X\).  The restriction~\(\Gr_K\) is minimal,
  \(\sigma\)\nb-compact itself and still has dynamical comparison.
  Thus all conditions
  \ref{enu:dynamical_comparison_ample1}--\ref{enu:dynamical_comparison_ample4}
  hold for~\(\Gr_K\).  However, since~\(\Gr\) is minimal, the
  groupoids \(\Gr\) and~\(\Gr_K\) are equivalent.  Hence the type
  semigroups for \(\Gr_K\) and~\(\Gr\) are isomorphic by
  Corollary~\ref{cor:equivalence_implies_the_same_type_semigroups}.
  Thus all conditions
  \ref{enu:dynamical_comparison_ample1}--\ref{enu:dynamical_comparison_ample4}
  hold also for~\(\Gr\).
\end{proof}

\begin{remark}
  The equivalence
  \ref{enu:dynamical_comparison_ample3}\(\Leftrightarrow\)\ref{enu:dynamical_comparison_ample4}
  in the minimal and second countable case is
  \cite{Ara_Bonicke_Bosa_Li:type_semigroup}*{Theorem~A}.  It is
  an analogue of a celebrated result by R\o{}rdam \cite{Rordam:stable_and_real_rank} on the equivalence
  between strict comparison and almost unperforation of the Cuntz
  semigroup for unital simple separable exact \(\Cst\)\nb-algebras.
  By \cite{Ara_Bonicke_Bosa_Li:type_semigroup}*{Theorem~C}, if every
  restriction of~\(\Gr\) to a closed invariant subset is almost
  finite, then the equivalent conditions
  \ref{enu:dynamical_comparison_ample1}--\ref{enu:dynamical_comparison_ample4}
  hold.
\end{remark}

\section{Pure infiniteness and stable finiteness}
\label{sec:pure_infiniteness}

We first rephrase the pure infiniteness criteria for
\(\Cst_\ess(\Gr,\LL)\) obtained
in~\cite{Kwasniewski-Meyer:Pure_infiniteness} in terms of the type
semigroup:

\begin{theorem}
  \label{the:purely_infinite_semigroup}
  Let~\(\Gr\) be an \'etale, residually topologically free locally
  compact groupoid with a locally compact Hausdorff unit space~\(X\)
  and let \((\Gr,\LL)\) be an essentially exact twist over~\(\Gr\).
  Let \(\B\subseteq \Bis(\Gr)\) be an inverse semigroup basis
  for~\(\Gr\) that consists of \(\sigma\)\nb-compact bisections that
  trivialise the twist~\(\LL\) and let
  \(\OO \defeq\setgiven{\s(W)}{W\in \B}\).  Assume one of the
  following conditions:
  \begin{enumerate}[label=\textup{(\roman*)}]
  \item \label{enu:pure_infiniteness_groupoid1}%
    there are only finitely many \(\Gr\)\nb-invariant open subsets
    of~\(X\);
  \item \label{enu:pure_infiniteness_groupoid2}%
    \(\OO\)~consists of compact open subsets;
  \item \label{enu:pure_infiniteness_groupoid3}%
    the compact open subsets in~\(\OO\) separate the
    \(\Gr\)\nb-invariant open subsets of~\(X\).
  \end{enumerate}
  If the monoid \(S_\B(\Gr)\) is purely infinite, then the
  \(\Cst\)\nb-algebra \(\Cst_\ess(\Gr,\LL)\) is purely infinite and
  has the ideal property.
\end{theorem}

\begin{proof}
  For each \(\sigma\)\nb-compact subset \(U\in \OO\), there is
  \(f_U \in \Cont_0(X)^+\) with open support~\(U\).  The family
  \(\mathcal{F}= (f_U)_U\subseteq \Cont_0(X)^+\) fills \(\Cont_0(X)\)
  by \cite{Kwasniewski-Meyer:Pure_infiniteness}*{Example~2.30}.  If
  \(U\in \OO\) is \(\sigma\)\nb-compact and~\([1_U]\) is properly
  infinite in \(S_\B(\Gr)\), then~\(f_U\) is properly infinite in
  \(\Cst_\ess(\Gr,\LL)\) by
  Corollary~\ref{cor:from_type_semigroup_to_Cuntz_semigroup}.  A
  similar conclusion can be inferred from
  \cite{Kwasniewski-Meyer:Pure_infiniteness}*{Lemmas 5.19 and~5.26}.
  So every element in~\(\mathcal{F}\) is properly infinite in
  \(\Cst_\ess(\Gr,\LL)\).  Thus the assertion follows from
  \cite{Kwasniewski-Meyer:Pure_infiniteness}*{Theorem~1.(3)} (see also
  \cite{Kwasniewski-Meyer:Pure_infiniteness}*{Theorems 2.37 and~5.8}).
\end{proof}

\begin{theorem}
  \label{the:purely_infinite_semigroup_dichotomy}
  Let \((\Gr,\LL)\) be an exact, twisted groupoid such that~\(\Gr\)
  is residually topologically free and
  \(\Cst_\red(\Gr,\LL)=\Cst_\ess(\Gr,\LL)\) \textup{(}see
  Proposition~\textup{\ref{prop:essential_reduced_coincide})}.
  Suppose that~\(\B\) satisfies the same assumptions as in
  Theorem~\textup{\ref{the:purely_infinite_semigroup}}.
  If~\(S_\B(\Gr)\) is almost unperforated \textup{(}under the
  assumptions \ref{enu:pure_infiniteness_groupoid1},
  \ref{enu:pure_infiniteness_groupoid2}, it suffices for~\(S_\B(\Gr)\)
  to have plain paradoxes\textup{)}, then the following conditions are
  equivalent:
  \begin{enumerate}
  \item \label{item:pure_infiniteness_groupoid2}%
    the \(\Cst\)\nb-algebra \(\Cst_\red(\Gr,\LL)\) is purely
    infinite;
  \item \label{item:pure_infiniteness_groupoid2b}%
    the \(\Cst\)\nb-algebra \(\Cst_\red(\Gr,\LL)\) is purely
    infinite and has the ideal property;
  \item \label{item:pure_infiniteness_groupoid3}%
    no hereditary \(\Cst\)\nb-subalgebra of \(\Cst_\red(\Gr,\LL)\)
    admits a tracial state;
  \item \label{item:pure_infiniteness_groupoid4}%
    there are no nontrivial regular \(\Gr\)\nb-invariant Borel measures
    on~\(X\);
  \item \label{item:pure_infiniteness_groupoid5}%
    the semigroup~\(S_\B(\Gr)\) admits no nontrivial regular state;
  \item \label{item:pure_infiniteness_groupoid6}%
    the semigroup~\(S_\B(\Gr)\) admits no nontrivial state;
  \item \label{item:pure_infiniteness_groupoid7}%
    every nonzero element in~\(S_\B(\Gr)\) is paradoxical;
  \item \label{item:pure_infiniteness_groupoid1}%
    the monoid \(S_\B(\Gr)\) is purely infinite.
  \end{enumerate}
\end{theorem}

\begin{proof}
  Hereditary subalgebras of purely infinite algebras are purely
  infinite by
  \cite{Kirchberg-Rordam:Non-simple_pi}*{Proposition~4.17}, and
  hence traceless.  Thus
  \ref{item:pure_infiniteness_groupoid2}\(\Rightarrow\)%
  \ref{item:pure_infiniteness_groupoid3}.  The implication
  \ref{item:pure_infiniteness_groupoid3}\(\Rightarrow\)%
  \ref{item:pure_infiniteness_groupoid5} follows from
  Corollary~\ref{cor:non_purely_infinite}, and
  \ref{item:pure_infiniteness_groupoid4}\(\Leftrightarrow\)%
  \ref{item:pure_infiniteness_groupoid5} by
  Theorem~\ref{thm:Riesz_for_monoids}.  If~\(S_\B(\Gr)\) is almost
  unperforated, then the equivalence
  \ref{item:pure_infiniteness_groupoid5}\(\Leftrightarrow\)%
  \ref{item:pure_infiniteness_groupoid6} is
  Corollary~\ref{cor:regular_states_for_unperforated}.  If
  \ref{enu:pure_infiniteness_groupoid2} holds, the equivalence
  \ref{item:pure_infiniteness_groupoid5}\(\Leftrightarrow\)%
  \ref{item:pure_infiniteness_groupoid6} is trivial and
  if~\ref{enu:pure_infiniteness_groupoid1} holds, this follows from
  Lemma~\ref{lem:nontrivial_state_finitely_many_invariant}.  This
  covers this equivalence under the assumptions in the theorem.  The
  equivalence
  \ref{item:pure_infiniteness_groupoid6}\(\Leftrightarrow\)%
  \ref{item:pure_infiniteness_groupoid7} is Tarski's Theorem.
  And~\ref{item:pure_infiniteness_groupoid7}
  implies~\ref{item:pure_infiniteness_groupoid1} if and only if \(S_\B(\Gr)\)
  has plain paradoxes, which we assume.  Finally,
  Theorem~\ref{the:purely_infinite_semigroup} gives
  \ref{item:pure_infiniteness_groupoid1}\(\Rightarrow\)%
  \ref{item:pure_infiniteness_groupoid2b}, and this
  implies~\ref{item:pure_infiniteness_groupoid2}.
\end{proof}

\begin{remark}
  If~\(\Gr\) is ample, then
  Theorem~\ref{the:purely_infinite_semigroup_dichotomy} holds with
  \(S_\B(\Gr)\) replaced by the refinement monoid~\(S(\Gr)\).  Then
  one can replace the assumption of almost unperforation with some
  weaker notions described in
  Remark~\ref{rem:assumptions_needed_for_dichotomy}.
  In~\cite{Pask-Sierakowski-Sims:Quasitraces_finite_infinite}, a
  number of cases were established when the type semigroup
  \(S(\Gr_\Lambda)\) of the groupoid coming from a higher-rank
  graph~\(\Lambda\) is almost unperforated or even unperforated.
  This holds, for instance, when~\(\Lambda\) is strongly connected
  and row-finite with no sources.  Ara and Exel constructed an
  action of a free group on a Cantor set with an element in the type
  semigroup~\(S(\Gr)\) which is paradoxical but not properly
  infinite (see \cite{Ara-Exel:Dynamical_systems}*{Corollary~7.12}).
  It is not known if this phenomenon may occur also when the action
  is residually topologically free, exact and every element in the
  associated type semigroup is paradoxical.  In other words, it is
  not known whether the conditions
  \ref{item:pure_infiniteness_groupoid1}
  and~\ref{item:pure_infiniteness_groupoid2} in
  Theorem~\ref{the:purely_infinite_semigroup_dichotomy} remain
  equivalent without assumptions on \(S_\B(\Gr)\) or~\(S(\Gr)\).
  This question is particularly relevant when~\(\Gr\) is minimal
  (see Proposition~\ref{prop:unperforation_for_minimal_exel_pardo}
  below).
\end{remark}

Let~\(D\) be an exotic \(\Cst\)\nb-algebra for the twisted groupoid
\((\Gr,\LL)\), that is, the canonical map factors into quotient
maps \(\Cst(\Gr,\LL) \onto D \onto \Cst_\ess(\Gr,\LL)\).
As noted above, if some \(U\in \OO\) is properly infinite in
\(S_\B(\Gr)\), then there is a properly infinite element in~\(D\).
Now we will notice that if there are paradoxical elements
in~\(S_\B(\Gr)\), then there are properly infinite elements in the
stabilised \(\Cst\)\nb-algebra \(D\otimes \Comp\):

\begin{proposition}
  \label{prop:Blackadar_Cuntz_stable_finite}
  Suppose that~\(D\) is an exotic \(\Cst\)\nb-algebra for the
  twisted groupoid \((\Gr,\LL)\) and let \(\B\subseteq\Bis(\Gr)\) be
  any inverse semigroup basis for~\(\Gr\) that consists of
  \(\sigma\)\nb-compact subsets that trivialise the bundle~\(\LL\).
  If~\(S_\B(\Gr)\) contains a paradoxical element, then
  \(D\otimes \Comp\) contains a properly infinite element.

  If, in addition, \(D\) is simple and there is a paradoxical
  \([f]\in S_\B(\Gr)\) with precompact \(\supp(f)\), then~\(D\) is
  not stably finite, that is, \(D\otimes \Comp\) contains an
  infinite projection.
\end{proposition}

\begin{proof}
  If \([f]\in S_\B(\Gr)\) is paradoxical, then~\(l[f]\) is properly
  infinite in \(S_\B(\Gr)\) for some \(l\in \N\) (see Lemma \ref{lem:paradoxical_implies_properly_infinite}).  This corresponds
  to \(U\in K(\OO)\) such that~\([1_U]\) is properly infinite in
  \(S_{\K(\B)}(\Gr \times \N^2)\cong S_\B(\Gr)\) (see
  Proposition~\ref{prop:stabilisation_vs_type_semigroup}).  Taking
  any
  \(a\in \Cont_0(X\times \N)\cong \Cont_0(X)\otimes c_0 \subseteq
  D\otimes \Comp\) with \(\supp(a)=U\) we conclude that~\(a\) is
  properly infinite in \(D\otimes \Comp\) by
  Corollary~\ref{cor:from_type_semigroup_to_Cuntz_semigroup}.
  Assume, in addition, that~\(D\) is simple and that~\(f\) has
  precompact support.  Then \(D\otimes \Comp\) is a stable simple
  \(\Cst\)\nb-algebra and~\(a\) is in Pedersen's ideal of
  \(D\otimes \Comp\) because \(\supp(a)=U\) is precompact.  So
  \(D\otimes \Comp\) contains an infinite projection by
  \cite{Blackadar-Cuntz:The_structure_of_stable_simple}*{Theorem~1.2}.
\end{proof}

This proposition together with the dichotomy for ordered monoids leads
to the following characterisation when a simple twisted groupoid
\(\Cst\)\nb-algebra is stably finite:

\begin{theorem}
  \label{thm:stably_finite_twisted}
  Let \((\Gr,\LL)\) be a twisted groupoid such that the
  \(\Cst\)\nb-algebra \(\Cst_\red(\Gr,\LL)\) is simple.  Let~\(\B\)
  be any inverse semigroup basis for~\(\Gr\) that consists of
  precompact, \(\sigma\)\nb-compact bisections that trivialise the
  bundle~\(\LL\).  The following are equivalent:
  \begin{enumerate}
  \item \label{item:stably_finite_Cartan1}%
    \(\Cst_\red(\Gr,\LL)\) admits a faithful semifinite lower
    semicontinuous trace;
  \item \label{item:stably_finite_Cartan2}%
    \(\Cst_\red(\Gr,\LL)\) is stably finite;
  \item \label{item:stably_finite_Cartan2.5}%
    there are no paradoxical elements in \(S_\B(\Gr)\);
  \item \label{item:stably_finite_Cartan2.8}%
    there is a nonzero element in \(S_\B(\Gr)\) which is not
    paradoxical;
  \item \label{item:stably_finite_Cartan5}%
    \(S_\B(\Gr)\) admits a nontrivial state;
  \item \label{item:stably_finite_Cartan5b}%
    \(S_\B(\Gr)\) admits a finite
    and faithful state;
  \item \label{item:stably_finite_Cartan5.5}%
    \(S_\B(\Gr)\) admits a nontrivial regular state;
  \item \label{item:stably_finite_Cartan5.5b}%
    \(S_\B(\Gr)\) admits a regular, finite and faithful
    state;
  \item \label{item:stably_finite_Cartan6}%
    there is a nontrivial regular \(\Gr\)\nb-invariant Borel
    measure on~\(X\);
  \item \label{item:stably_finite_Cartan6b}%
    there is a locally finite regular
    \(\Gr\)\nb-invariant Borel measure on~\(X\) with full support.
  \end{enumerate}
  If, in addition, \(\Gr\) is second countable, amenable and Hausdorff, then the properties
  \ref{item:stably_finite_Cartan1}--\ref{item:stably_finite_Cartan6b}
  are all equivalent to
  \begin{enumerate}[resume]
  \item \label{item:stably_finite_Cartan8}%
    the \(\Cst\)\nb-algebra \(\Cst_\red(\Gr,\LL)\) is quasidiagonal.
  \end{enumerate}
\end{theorem}

\begin{proof}
  If \(\Cst_\red(\Gr,\LL)\) is simple, then~\(\Gr\) is minimal.  Hence
  the monoid~\(S_\B(\Gr)\) is simple by
  Corollary~\ref{cor:minimal_implies_simple} (here we use that
  elements in~\(\OO\) are precompact).  Therefore, every nontrivial
  state on~\(S_\B(\Gr)\) has to be finite and faithful by
  Lemma~\ref{lem:simplicity_implies_faithfulness_of_states}.  So
  \ref{item:stably_finite_Cartan5}$\iff$\ref{item:stably_finite_Cartan5b}
  and
  \ref{item:stably_finite_Cartan5.5}$\iff$\ref{item:stably_finite_Cartan5.5b}.
  Similarly, \ref{item:stably_finite_Cartan6}
  and~\ref{item:stably_finite_Cartan6b} are equivalent because any
  \(\Gr\)\nb-invariant Borel measure on~\(X\) is locally finite and
  has full support because~\(\Gr\) is minimal.
  Conditions \ref{item:stably_finite_Cartan5}
  and~\ref{item:stably_finite_Cartan5.5} are equivalent by
  Lemma~\ref{lem:nontrivial_state_finitely_many_invariant}, while
  \ref{item:stably_finite_Cartan5.5}
  and~\ref{item:stably_finite_Cartan6} are equivalent by
  Theorem~\ref{thm:Riesz_for_monoids}.

  Corollary~\ref{cor:non_purely_infinite} shows that
  \ref{item:stably_finite_Cartan5.5}
  implies~\ref{item:stably_finite_Cartan1}, and it is well-known
  that \ref{item:stably_finite_Cartan1}
  implies~\ref{item:stably_finite_Cartan2} (see
  \cite{Blanchard-Kirchberg:Non-simple_infinite}*{Remark~2.27.(viii)}).
  That~\ref{item:stably_finite_Cartan2}
  implies~\ref{item:stably_finite_Cartan2.5} follows from
  Proposition \ref{prop:Blackadar_Cuntz_stable_finite},
  and~\ref{item:stably_finite_Cartan2.5} obviously
  implies~\ref{item:stably_finite_Cartan2.8}.
  Condition~\ref{item:stably_finite_Cartan2.8}
  implies~\ref{item:stably_finite_Cartan5} by Tarski's Theorem,
  Corollary~\ref{cor:original_Tarski}.

  It is well known that~\ref{item:stably_finite_Cartan8}
  implies~\ref{item:stably_finite_Cartan2}.  Conversely,
  \ref{item:stably_finite_Cartan5.5} and the proof of
  Corollary~\ref{cor:non_purely_infinite} imply that
  \(\Cst_\red(\Gr,\LL)\) contains a nonzero hereditary
  subalgebra~\(H\) with a faithful trace which is isomorphic to
  \(\Cst_\red(\Gr_U,\LL_U)\), where~\(U\) is an open subset of~\(X\).
  The properties of \(\Gr\) that we assume here pass to~\(\Gr_U\)
  and so the \(\Cst\)\nb-algebra \(H\cong \Cst_\red(\Gr_U,\LL_U)\) is
  separable nuclear and satisfies the UCT (see
  \cite{Takeishi:Nuclearity}*{Theorem 5.4} and
  \cite{Barlak-Li:Cartan_UCT}*{Theorem 3.1}). Hence~\(H\) is
  quasidiagonal by
  \cite{Tikuisis-White-Winter:Quasidiagonality_nuclear}*{Corollary~B}.
  As~\(H\) is stably isomorphic to \(\Cst_\red(\Gr,\LL)\), we conclude
  that \(\Cst_\red(\Gr,\LL)\) is also quasidiagonal.
\end{proof}

\begin{remark}
  If~\(\Gr\) is ample and~\(\B\) is the family of compact open
  bisections, then Theorem~\ref{thm:stably_finite_twisted} holds
  with \(S_\B(\Gr)\) replaced by~\(S(\Gr)\) (see
  Proposition~\ref{prop:type_vs_groupoid_semigroup}).  Also, if the
  equivalent conditions in this theorem hold, then
  \(S_\B(\Gr)\cong S(\Gr)\).  In this way,
  Theorem~\ref{thm:stably_finite_twisted} generalises the
  corresponding results in \cites{Boenicke-Li:Ideal,
    Rainone-Sims:Dichotomy}.
\end{remark}

In the stably finite case, the type semigroup is usually
complicated, especially if one takes~\(\OO\) to be the whole
topology (like in~\cite{Ma:Purely_infinite_groupoids}).  An
appropriate choice of~\(\OO\) may sometimes simplify things:

\begin{example}
  \label{Ex:irrational_rotation}
  Consider the irrational rotation algebra
  \(A_\theta\defeq \Cont(\T)\rtimes\Z\) for some
  \(\theta\notin \Q\), obtained from the action of~\(\Z\) on the
  unit circle~\(\T\) by
  \(x\to x \mathrm{e}^{2\pi \mathrm{i}\theta}\).  It is modeled by
  the transformation groupoid \(\Gr_{\theta}\defeq \T\rtimes \Z\).
  Let \(\OO\defeq \setgiven{(a,b)\subseteq \mathbb{T}}{a,b\in\Q}\)
  be the set of all arcs with rational endpoints.  If
  \(\alpha, \beta\in \OO\), then \(1_\alpha\precsim_\B 1_\beta\) if
  and only if \(\abs{\alpha}<\abs{\beta}\) or \(\alpha=\beta\).
  Thus \(\alpha\approx_\B\beta\) if and only if \(\alpha=\beta\),
  and so \(S_\B(\Gr_{\theta})=\F(\OO)\).  The order on
  \(S_\B(\Gr_{\theta})\), however, is induced by the length of arcs,
  as described above for generators.  In particular, for any nonzero
  \(\alpha\in \OO\) there is \(n\in \N\) with
  \(1\le n \abs{\alpha}\) and so
  \(1_{\T}\precsim_\B n 1_{\alpha}\).  Hence \(S_\B(\Gr_{\theta})\)
  is simple.  Up to scaling, there is a unique nontrivial state on
  \(S_\B(\Gr_{\theta})\), which is determined by
  \(\nu(1_{\alpha})=\abs{\alpha}\) for \(\alpha\in \OO\).  In
  accordance with Theorem~\ref{thm:stably_finite_twisted}, this
  state corresponds to the unique tracial state on
  \(A_\theta\cong \Cst_\red(\Gr_{\theta})\).
\end{example}

We now give a sample series of dichotomy results that follow from the above theorems.

\begin{corollary}
  \label{cor:dichotomy_for_topologically_free_groupoids}
  Let \((\Gr,\LL)\) be a twisted groupoid with minimal and
  topologically free~\(\Gr\).
  Let~\(\B\) be an inverse semigroup basis consisting of precompact,
  \(\sigma\)\nb-compact subsets that trivialise the twist~\(\LL\).
  Assume that there is no \(f\in \Cst_\red(\Gr,\LL)^+\setminus\{0\}\)
  for which \(\setgiven*{x\in X}{\E(f)(x)\neq 0}\) is meagre.
  Assume that the type semigroup~\(S_\B(\Gr)\) has plain paradoxes.
  Then \(\Cst_\red(\Gr,\LL)\) is simple and either purely infinite or
  stably finite.
  In the latter case, it admits a faithful semifinite lower
  semicontinuous trace, and even a tracial state if~\(X\) is compact.
\end{corollary}

\begin{proof}
  By Propositions \ref{prop:essential_reduced_coincide}
  and~\ref{prop:ideals_twisted_groupoids},
  \(\Cst_\red(\Gr,\LL)=\Cst_\ess(\Gr,\LL)\) is simple.  Hence it is
  either purely infinite or stably finite by
  Corollary~\ref{cor:dichotomy_for_monoids} and Theorems
  \ref{the:purely_infinite_semigroup_dichotomy}
  and~\ref{thm:stably_finite_twisted}.
\end{proof}

\begin{corollary}
  \label{cor:dichotomy_for_universal}
  Assume that~\(\Gr\) is an \'etale groupoid such that~\(\Cst(\Gr)\)
  is simple and there is an inverse semigroup basis~\(\B\) of
  precompact \(\sigma\)\nb-compact bisections such
  that~\(S_\B(\Gr)\) has plain paradoxes.
  Then~\(\Cst(\Gr)\) is either purely infinite or stably finite.
\end{corollary}

\begin{proof}
  If~\(\Cst(\Gr)\) is simple, then
  \(\Cst(\Gr)=\Cst_\red(\Gr)=\Cst_\ess(\Gr)\), and the groupoid is
  minimal and topologically free by
  \cite{Kwasniewski-Meyer:Essential}*{Theorem~7.29} and
  Lemma~\ref{lem:hereditary_subalgebras_essential}.  Hence the
  assertion is a special case of
  Corollary~\ref{cor:dichotomy_for_topologically_free_groupoids}.
\end{proof}

\begin{corollary}
  \label{cor:dichotomy_for_ample}
  Let~\(\Gr\) be a minimal, topologically free, ample groupoid.  Let
  \(\B\subseteq \Bis(\Gr)\) be an inverse semigroup basis consisting
  of compact, regular open subsets as in
  Proposition~\textup{\ref{prop:essential_reduced_coincide}}.
  Assume that the type semigroup~\(S_\B(\Gr)\) has plain paradoxes.
  Then \(\Cst_\red(\Gr,\LL)\) is simple for any
  twist~\(\LL\) over~\(\Gr\).  Either \(\Cst_\red(\Gr,\LL)\) is
  purely infinite for all twists~\(\LL\) or it is stably finite for
  all twists~\(\LL\).
\end{corollary}

\begin{proof}
  Since~\(\B\) consists of compact, regular open subsets, every
  compact open subset in~\(\Gr\) is regular open.  Hence
  \(\Cst_\red(\Gr,\LL)=\Cst_\ess(\Gr,\LL)\) for any twist~\(\LL\) by
  Proposition~\ref{prop:essential_reduced_coincide}.  This is simple
  by Proposition~\ref{prop:ideals_twisted_groupoids}.  In view of
  Lemma~\ref{lem:ample_change_B}, the type semigroup \(S_\B(\Gr)\)
  remains unchanged if we replace~\(\B\) by another inverse
  semigroup basis consisting of compact open subsets.  By
  Theorem~\ref{the:purely_infinite_semigroup_dichotomy},
  \(\Cst_\red(\Gr,\LL)\) fails to be purely infinite for a
  twist~\(\LL\) if and only if \(S_\B(\Gr)\) admits a nontrivial
  state.  Then \(\Cst_\red(\Gr,\LL)\) admits a lower semicontinuous
  trace by Theorem~\ref{thm:stably_finite_twisted}.  Here the same
  case of the dichotomy occurs for all~\(\LL\) because \(S_\B(\Gr)\)
  does not depend on~\(\LL\).
\end{proof}

Finally, we phrase our dichotomy for simple  groupoid
\(\Cst\)\nb-algebras in terms of Cartan inclusions.  Let
\(A\subseteq B\) be a nondegenerate \emph{regular
  \(\Cst\)\nb-inclusion}, that is, \(A B= B\) and \(B\)  is
the closed linear span of the \emph{normalisers} of~\(A\):
\[
  \NN_A(B)\defeq\setgiven{b\in B}{ b A b^*\subseteq A \text{ and
    }b^* A b\subseteq A}.
\]
Assume also that \(A\cong\Cont_0(X)\) is commutative.
The inclusion
\(A\subseteq B\) is \emph{Cartan} if and only if there is a unique
faithful conditional expectation \(E\colon B\to A\) (see
\cite{Kwasniewski-Meyer:Cartan}*{Theorem~7.2 and Corollary~7.6}).
Then~\(A\) is a maximal abelian subalgebra in~\(B\).  After
tensoring with the standard diagonal \(c_0\) of \(\Comp\), \(A\otimes c_0\) is also a regular maximal abelian
subalgebra in \(B\otimes \Comp\).  Let \(\K(B)\) denote the Pedersen
ideal of~\(B\).  Then \(\K(A)=\Contc(X)\), \(\K(c_0)=c_{00}\) is the
ideal of  sequences with finite support, and \(\K(\Comp)=\Fin\) is the
ideal of finite rank operators.  Thus the inclusion
\(\K(A)\otimes \K(c_0)\subseteq \K(B)\otimes \K(\Comp)\) translates
to
\(\Contc(X\times \N)\cong \Contc(X,c_{00})\subseteq \K(B)\otimes
\Fin\).  The family
\[
  \NN_{\K(A)}^\mathrm{stab}(\K(B))
  \defeq \setgiven{b\otimes k} {b \in \NN_{A}(B),\ k \in \NN_{c_0}(\Comp) \text{ and }bb^*\otimes k^*k\in \K(A)\otimes \K(c_0)}
\]
consists of simple tensors in \(\K(B)\otimes \K(\Comp)\subseteq B\otimes \Comp\) that are normalisers of \(\K(A)\otimes \K(c_0)\cong \Contc(X\times \N)\) (in particular
\(bb^*\in \K(A)\) implies that \(b\in \K(A)B\subseteq \K(B)\)).

For
\(a,b\in (\K(A)\otimes \K(c_0))^+\cong\Contc(X\times \N)^+\), we write \(a\preceq b\) if for
every \(\varepsilon >0\) there are
\(b_1,\dotsc b_n\in \NN_{\K(A)}^\mathrm{stab}(\K(B))\) such that
\begin{equation}\label{eq:Cuntz_relation_Cartan}
  (a-\varepsilon)_+ \le \sum_{i=1}^n b_i^* b_i, \qquad
  \sum_{i=1}^n b_i b_i^* \le b
  \quad\text{ and } \quad
  b_j^*b_i=0\text{ for all }i\neq j.
\end{equation}

We also write \(a\approx b\) if \(a\preceq b\) and \(b\preceq a\).

\begin{proposition}
  \label{prop:Cartan_type_semigroup}
  Let \(A\subseteq B\) be a Cartan inclusion.  Let \((\Gr,\LL)\) be
  the twisted groupoid that models the inclusion \(A\subseteq B\)
  and let~\(\B\) be the inverse semigroup basis consisting of all
  \(\sigma\)\nb-compact and precompact bisections that trivialise
  the bundle~\(\LL\).  The relation~\(\approx\) defined above is an
  equivalence relation on \((\K(A)\otimes \K(c_0))^+\) and the
  corresponding quotient set
  \(W(A,B)=\setgiven{[a]}{a\in(\K(A)\otimes \K(c_0))^+}\) is
  isomorphic to \(S_\B(\Gr)\) as an ordered abelian monoid with the
  operations \([a]+[b]=[a\oplus b]\) and \([a]\preceq [b]\) if and
  only if \(a\preceq b\), for \(a,b\in \K(A)\otimes \K(c_0)\).
\end{proposition}

\begin{proof}
  Replacing \(\NN_{\K(A)}^\mathrm{stab}(\K(B))\) by
  \(\NN\defeq \NN_{\K(A)}^\mathrm{stab}(\K(B))\cup
  \bigl(\K(A)\otimes \K(c_0)\bigr)\) does not change the
  relation~\(\preceq \) on \(\K(A)\otimes \K(c_0))^+\).  Indeed,
  assume that~\eqref{eq:Cuntz_relation_Cartan} holds for
  \(b_1,\dotsc b_n\in \NN\).  Assume that~\(b_i\) for some~\(i\)
  belongs to \(\K(A)\otimes \K(c_0)=\Contc(X\times \N)\).  Then
  \(b_i=\sum_{k\in F}a_k\otimes 1_{k}\) for a finite subset
  \(F\subseteq \N\) and \(a_k\in \Contc(X)\) for \(k\in F\).  Let
  \(b_{i,k}\defeq a_k\otimes 1_{k} \) for \(k\in F\).  Then
  \(b_i^*b_i=b_ib_i^*=\sum_{k\in F} b_{i,k}^*b_{i,k}= \sum_{k\in F}
  b_{i,k} b_{i,k}^*\) and \(b_{i,k}^{(*)}b_{i,l}^{(*)}=0\) for
  \(k\neq l\).  Thus the relations~\eqref{eq:Cuntz_relation_Cartan}
  remain unchanged if we replace~\(b_i\) by the family of
  \(b_{i,k}\in \NN_{\K(A)}^\mathrm{stab}(\K(B))\) for all
  \(k\in F\).  The operations inherited from the \(\Cst\)-algebra
  \(B\otimes \Comp\) make~\(\NN\) a \Star{}semigroup.

  By assumption, \(B=\Cst_\red(\Gr,\LL)\) and~$\Gr$ has unit
  space~\(X\).  By Lemma~\ref{lem:stabilisation_groupoids},
  \(B\otimes \Comp\) is modelled by the stabilised twisted groupoid
  \((\Gr\times \RR,\LL\otimes \C)\) with the unit space
  \(X\times \N\).  Let~\(\B\) be the inverse semigroup basis as in
  the statement.  We consider the inverse semigroup basis
  \(\widetilde{\K}(\B)=\setgiven{U\times V\in\B\times
    \Bis(\RR)}{V\text{ is finite}}\cup \K(\OO)\) for
  \(\Gr \times \RR\) from Remark
  \ref{remark:stabilisation_dynamical2}, and we use the stabilised
  picture
  \(S_{\widetilde{\K}(\B)}(\Gr \times \RR) = \setgiven{[1_U]}{U\in
    \K(\OO)}/{\approx}_{\widetilde{\K}(\B)}\) of \(S_\B(\Gr)\) (see
  Proposition~\ref{prop:stabilisation_vs_type_semigroup}).  Then
  \(\widetilde{\K}(\B)\) consists of precompact,
  \(\sigma\)\nb-compact bisections of \(\Gr\times \RR\) that
  trivialise the bundle~\(\LL\otimes \C\).  There is a natural
  injective bounded linear map
  \(\Cst_\red(\Gr\times \RR,\LL\otimes \C) \subseteq
  \Cont_0(\Gr\times \RR,\LL\otimes \C)\) under which multiplication
  and involution are given by the same formulas as on
  \(\Contc(\Gr\times \RR,\LL\otimes \C)\) (see, for instance,
  \cite{Bardadyn-Kwasniewski-McKee:Banach_algebras_simple_purely_infinite}*{Proposition~3.15}).
  We claim that the map
  \[
    \varphi\colon \NN\to \widetilde{\K}(\B),\qquad
    n\mapsto \supp(n),
  \]
  is well defined and surjective.  Indeed, if \(b\in \NN_{A}(B)\)
  and \(b^*b\in \K(A)=\Contc(X)\), then \(\supp(b)\) is a
  \(\sigma\)-compact open bisection as an open support of a
  normaliser section.  It is also precompact because
  \(\s(\supp(b))=\supp (b^*b)\) is precompact.  Conversely, for any
  \(U\in \B\) there is a section \(b\in \Cont_0(U,\LL)\subseteq B\)
  with \(\supp(b)=U\) so that \(b\in \NN_{A}(B)\) and
  \(b^*b\in \K(A)\).  In particular, there is a surjection
  \[
    \setgiven{k\in \K(\Comp)}{k^*k\in \K(c_0)=\Fin} \onto
    \setgiven{V\in\Bis(\RR)}{V\text{ is finite}},\qquad
    k \mapsto \supp(k).
  \]
  Thus the map
  \(\NN_{\K(A)}^\mathrm{stab}(\K(B))\to \setgiven{U\times
    V\in\B\times \Bis(\RR)}{V\text{ is finite}} \) defined by
  \(b\otimes k \mapsto \supp(b\otimes k)=\supp(b)\times\supp(k)\) is
  surjective.  So is the map \(\K(A)\otimes \K(c_0)\to \K(\OO) \)
  defined by
  \( \sum_{k\in F}a_k\otimes 1_{k} \mapsto \supp\bigl(\sum_{k\in
    F}a_k\otimes 1_{k}\bigr) = \bigcup_{k\in F} \supp(a_k)\times
  \{k\}\).  This proves our claim.

  It is readily seen that the surjection~\(\varphi\) preserves
  multiplication and involution.  This implies that if
  \(a, b\in (\K(A)\otimes \K(c_0))^+=\Contc(X\times \N)^+\), then
  \(a\preceq b\) if and only if \(\supp(a)\precsim \supp(b)\) in the
  sense described in Remark~\ref{remark:stabilisation_dynamical}.
  Hence~\(\preceq\) is a preorder and the surjection
  \(\K(A)\otimes \K(c_0)\to \F(\K(\OO))/{\approx}\) with
  \(a\mapsto [1_{\supp(a)}]\) descends to an order-preserving
  bijection \(W(A,B)\cong S_\B(\Gr)\).  The addition in
  \(S_\B(\Gr)\) transfers to the addition on \(W(A,B)\) as in the
  assertion.
\end{proof}

\begin{corollary}
  \label{cor:Cartan_type_semigroup}
  Let \(A\subseteq B\) be a Cartan inclusion and let \(W(A,B)\) be
  the associated ordered abelian monoid described above.  Then~\(B\)
  is simple if and only if \(W(A,B)\) is simple.  Assume this.
  Then~\(B\) is stably finite if and only if \(W(A,B)\) has a
  nontrivial state.  If \(W(A,B)\) is purely infinite, then so
  is~\(B\), and the converse holds if \(W(A,B)\) has plain paradoxes.
\end{corollary}

\begin{proof}
  By Corollary~\ref{cor:minimal_implies_simple},
  \(W(A,B)\cong S_\B(\Gr)\) is simple if and only if~\(\Gr\) is
  minimal, which is, in turn, equivalent to \(B=\Cst_\red(\Gr,\LL)\)
  being simple because \(\Gr\) is topologically free and Hausdorff.
  The remaining part follows from Theorems
  \ref{the:purely_infinite_semigroup}
  and~\ref{thm:stably_finite_twisted}.
\end{proof}

For a \(\Cst\)\nb-inclusion \(A\subseteq B\), let
\(\mathbb{I}^B(A)\defeq\setgiven{J\cap A}{J \text{ is an ideal in }B}\)
denote the lattice of restricted ideals.  Recall that Kumjian's
\(\Cst\)\nb-diagonals~\cite{Kumjian:Diagonals} are special Cartan
inclusions.

\begin{corollary}
  \label{cor:diagonal_type_semigroup}
  Let \(A\subseteq B\) be a \(\Cst\)\nb-diagonal and let \(W(A,B)\)
  be the associated ordered abelian monoid.  Assume that~\(B\) is
  nuclear and either that \(\mathbb{I}^B(A)\) is finite or that projections
  in~\(A\) separate ideals in~\(\mathbb{I}^B(A)\).  If \(W(A,B)\) is purely
  infinite, then~\(B\) is purely infinite.  If \(W(A,B)\)
  is almost unperforated, then~\(B\) is purely infinite if and only
  if \(W(A,B)\) is purely infinite \textup{(}and otherwise it has a
  nontrivial lower semicontinuous trace\textup{)}.
\end{corollary}

\begin{proof}
  Since \(A=\Cont_0(X)\subseteq B\) is a \(\Cst\)\nb-diagonal, we may
  assume \(B=\Cst_\red(\Gr,\LL)\), where~\(\Gr\) is principal with
  unit space~\(X\).  By
  Theorem~\cite{Kwasniewski-Meyer:Stone_duality}*{Proposition~6.19},
  an ideal in \(\Cont_0(X)\) is restricted if and only if the
  underlying open subset in~\(X\) is \(\Gr\)\nb-invariant, that is,
  \[
    \mathbb{I}^B(A)=\setgiven{\Cont_0(U)}{U\text{ is open
        \(\Gr\)\nb-invariant}}
  \]
  Since~\(B\) is nuclear, \(\Gr\) is amenable.  Hence \((\Gr,\LL)\) is
  exact and residually topologically free.  Thus the assertion follows
  from Theorems \ref{the:purely_infinite_semigroup}
  and~\ref{the:purely_infinite_semigroup_dichotomy}.
\end{proof}

\section{Self-similar group actions on graphs and Exel--Pardo algebras}
\label{sec:self-similar}

We fix a \emph{self-similar action of a group~\(\Gamma\) on a directed
  graph} as defined in~\cite{Exel-Pardo:Self-similar}.  Thus we let
\(E=(E^0,E^1,\rg,\s)\) be a \emph{row-finite directed graph
  without sources}, and we assume that~\(\Gamma\) is a group acting
on~\(E\) by graph automorphisms and equipped with a restriction map
\(\Gamma\times E^1\ni (g, e)\mapsto g|_e \in \Gamma\) satisfying
\[
  (gh)|_e  = g|_{he}  h|_e, \qquad
  g|_e \s(e) = g \s(e) \qquad
  \text{ for all } g\in \Gamma, e\in E^1.
\]
Let~\(E^*\) be the set of finite paths and~\(E^\infty\) the set of
infinite paths in~\(E\).  The \(\Gamma\)\nb-action on~\(E\) extends
uniquely to \(\Gamma\)\nb-actions on \(E^*\) and~\(E^\infty\) such
that \(g (e \mu)=(g e) (g|_e \mu)\) for all \(g\in \Gr\),
\(e\in E^1\) and \(\mu \in E^*\cup E^\infty\) with
\(\s(e)=\rg(\mu)\) (see \cite{Exel-Pardo:Self-similar}*{Propositions
  2.4 and~8.1}).  We define an inverse semigroup
\[
  S_{\Gamma,E} \defeq \setgiven{(\alpha,g,\beta)\in E^*\times
    \Gamma\times E^*}
  {\s(\alpha)=gs(\beta)}\cup \{0\},
\]
with the operations
\begin{align*}
  (\alpha,g,\beta) (\gamma,h,\delta)
  &=
    \begin{cases}
      (\alpha(g\epsilon),g|_\epsilon h,\delta) &\text{if } \gamma=\beta\epsilon \\
      (\alpha,g(h^{-1}|_\epsilon)^{-1},\delta (h^{-1}\epsilon)),
                                               &\text{if } \beta=\gamma\epsilon,\\
      0 &\text{otherwise},
    \end{cases}\\
  (\alpha,g,\beta)^*
  &=(\beta,g^{-1},\alpha).
\end{align*}
In particular,
\((\alpha,g,\beta) (\beta\epsilon,h,\delta)=(\alpha (g\epsilon),
g|_{\epsilon}h, \delta)\).  This inverse semigroup acts naturally
on~\(E^\infty\), where the partial homeomorphism corresponding to an
element \((\alpha,g,\beta)\in S_{\Gamma,E}\) has the cylinder set
\(Z(\beta)\defeq \setgiven{\beta\eta}{\eta\in E^\infty}\) as its
domain and sends \(\beta\eta\) to \(\alpha (g\eta)\).  We let
\[
  \Gr\defeq S_{\Gamma,E}\ltimes  E^\infty
\]
be the associated transformation groupoid.  We denote the
appropriate equivalence class of $(\alpha, g ,\beta;\xi)$ by
$[\alpha,g,\beta; \xi]$.  Then
\[
  \Gr =\setgiven{[\alpha,g,\beta; \xi]}
  {(\alpha,g,\beta) \in S_{\Gamma,E} \text{ and } \xi \in Z(\beta)}
\]
with the multiplication
\[
  [\alpha,g,\beta;  \beta\xi] [\gamma, h,\delta; \delta \eta]
  = [(\alpha,g,\beta)(\gamma,h,\delta); \delta \eta]
\]
whenever $\beta \xi = \gamma(h\eta)$, and the inverse
$[\alpha,g,\beta;\beta\xi]^{-1} = [\beta, g^{-1}, \alpha;
\alpha(g\xi)]$.  A basis for an \'etale topology on~$\Gr$ is
generated by the compact open bisections
\[
  Z(\alpha,g,\beta) \defeq
  \setgiven{ [\alpha,g,\beta;\xi]}{\xi \in Z(\beta)}, \qquad
  (\alpha,g,\beta)\in S_{\Gamma,E}
\]
(see \cite{Exel-Pardo:Self-similar}*{Proposition~9.4}).  In
particular, the source and range of \(Z(\alpha,g,\beta)\) are
\(Z(\beta)\) and \(Z(\alpha)\), respectively.  The unit space
\(X= \setgiven{[\rg(\xi),1, \rg(\xi); \xi]}{ \xi \in E^\infty}\) is
naturally identified with~\(E^\infty\).

Since the range map \(\rg\colon E^\infty\to E^0\) is proper and
surjective, the composition operator
\[
  \tau(f)\defeq f\circ \rg, \qquad f\in \Contc(E^0,\N)
\]
is a well-defined injective monoid homomorphism
\(\tau\colon \Contc(E^0,\N)\to \Contc(E^{\infty},\N)\).  This
homomorphism is determined by the formula \(\tau(1_{\{w\}})=1_{Z(w)}\) for
all \(w\in E^0\).  We want to describe the pull-back of the
dynamical preorder~\(\precsim_\B\) on \(\Contc(E^{\infty},\N)\)
induced by the inverse semigroup~\(\B\) of compact open bisections
of~\(\Gr\), via the map~\(\tau\).  We will also describe the pull-back
of the equivalence relation~\(\sim_{\Gr}\) given
by~\eqref{eq:Rainone_Sims_equiv}, considered
in~\cite{Rainone-Sims:Dichotomy} (and stronger than the
relation~\(\approx_\B\) induced by~\(\precsim_\B\)).  To this end,
for any \(f\colon E^0\to \N\) and \(n>0\) we define
\(\Theta^n(f)\colon E^0\to \N\) by
\[
  \Theta^n(f)(v)\defeq\sum_{\lambda\in E^nv} f(\rg(\lambda)).
\]
This defines an additive map
\(\Theta^n\colon\Contc(E^0,\N)\to \Contc(E^0,\N)\) with
\[
  \Theta^n(1_{\{w\}})=\sum_{v\in E^0} \abs{w E^n v}1_{\{v\}}
  \qquad \text{for }w\in E^0.
\]

\begin{lemma}
  \label{lem:auxillaty_graph_preorder}
  For any \(f\in \Contc(E^0,\N)\),
  \begin{enumerate}
  \item \label{enu:auxillaty_graph_preorder1}%
    if \(n>0\), then \(\tau(f) \sim_{\Gr} \tau(\Theta^n(f))\);
  \item \label{enu:auxillaty_graph_preorder1.5}%
    if \(g\in \Gamma\) and \(g\cdot f(v)\defeq f(g^{-1}v)\), then  \(\tau(g\cdot f)\sim_{\Gr} \tau(f)\);
  \item \label{enu:auxillaty_graph_preorder2}%
    if \(\tau(f)=\sum_{i=1}^m 1_{Z(\beta_i)}\) for some
    \(\beta_i\in E^n\), then
    \(\Theta^n(f)=\sum_{i=1}^m 1_{\{\s(\beta_i)\}}\).
  \end{enumerate}
\end{lemma}

\begin{proof}
  \ref{enu:auxillaty_graph_preorder1}: Since~\(\Theta^n\)
  and~\(\tau\) are additive and~\(\sim_{\Gr}\) respects addition, it
  suffices to consider \(f=1_{\{w\}}\) for a vertex \(w\in E^0\).
  We may enumerate paths in~\(w E^n\) as \(\{\alpha_i\}_{i=1}^m\),
  where \(m\defeq \abs{w E^n}<\infty\).  Then
  \(\Theta^n(1_{\{w\}})=\sum_{i=1}^m 1_{\{\s(\alpha_i)\}}\) and,
  therefore, using the bisections \(Z(\alpha_i,e,\s(\alpha_i))\),
  \(i,\dotsc, m\), we get
  \[
    \tau(\Theta^n(1_{\{w\}}))
    = \sum_{i=1}^m \tau(1_{\{\s(\alpha_i)\}})
    = \sum_{i=1}^m 1_{Z(\s(\alpha_i))}
    \sim_{\Gr} \sum_{i=1}^m 1_{Z(\alpha_i)}
    = 1_{Z(w)}
    = \tau(1_{\{w\}}).
  \]

  \ref{enu:auxillaty_graph_preorder1.5}: As
  in~\ref{enu:auxillaty_graph_preorder1}, we may assume that
  \(f=1_{\{w\}}\).  Then \(g\cdot f=1_{\{gw\}}\).  Hence using the
  bisection \(Z(gw,g,w)\) we get
  \[
    \tau(f)
    = 1_{Z(w)}\sim_{\Gr}1_{Z(gw)}
    = \tau (g\cdot f).
  \]

  \ref{enu:auxillaty_graph_preorder2}: Write an element
  in~\(E^\infty\) as~\(\lambda' \mu\) for some
  \(\lambda'\in E^nr(\mu)\) and \(\mu\in E^{\infty}\).  Then
  \(f(r(\lambda'))=\tau(f)(\lambda'\mu)=
  \abs{\setgiven{i}{\lambda'=\beta_i}}\).  Using this we get
  \begin{align*}
    \tau(\Theta^n(f))(\lambda\mu)
    &  = \Theta^n(f)(\rg(\lambda))
      = \sum_{\lambda'\in E^nr(\lambda)} f(\rg(\lambda'))
      = \sum_{\lambda'\in E^nr(\lambda)} \abs{\setgiven{i}{\lambda'=\beta_i}}
    \\
    &= \abs{\setgiven{i}{\s(\beta_i)
      = \rg(\lambda)}}
      = \sum_i 1_{Z(\s(\beta_i))}(\lambda\mu)
      = \tau\biggl(\sum_{i=1}^m 1_{\{\s(\beta_i)\}}\biggr)(\lambda\mu).
  \end{align*}
  Since~\(\tau\) is injective, we get
  \(\Theta^n(f)=\sum_{i=1}^m 1_{\{\s(\beta_i)\}}\).
\end{proof}

\begin{definition}
  For any \(f,\tilde{f}\in \Contc(E^0,\N)\) we write:
  \begin{enumerate}
  \item \(f\sim_{\Theta} \tilde{f}\) if there are \(p,q\in\N\) such
    that \(\Theta^p(f)=\Theta^q(\tilde{f})\);
  \item \(f\sim_{\Gamma} \tilde{f}\) if there are \(g\in \Gamma\)
    such that \(f(v)= \tilde{f}(gv)\) for all \(v\in E^0\);
  \item \(f\sim \tilde{f}\) if there are
    \(f_1,\dotsc,f_n,\tilde{f}_1\dots,\tilde{f}_n,\tilde{\tilde{f}}_1\dots,\tilde{\tilde{f}}_n\in\Contc(E^0,\N)\)
    such that
    \(f_i\sim_{\Gamma} \tilde{f}_i\sim_{\Theta}
    \tilde{\tilde{f}}_i\) for \(i=1,\dotsc,n\) and
    \(f\sim_{\Theta}\sum_{i=1}^nf_i\) and
    \(\sum_{i=1}^n\tilde{\tilde{f}}_i\sim_{\Theta} \tilde{f}\);
  \item \(f\precsim \tilde{f}\) if there are
    \(f_1,\dotsc,f_n,\tilde{f}_1\dots,\tilde{f}_n,\tilde{\tilde{f}}_1\dots,\tilde{\tilde{f}}_n,\tilde{\tilde{f}}\in\Contc(E^0,\N)\)
    such that
    \(f_i\sim_{\Gamma} \tilde{f}_i\sim_{\Theta}
    \tilde{\tilde{f}}_i\) for \(i=1,\dotsc,n\) and
    \(f\sim_{\Theta}\sum_{i=1}^nf_i\) and
    \(\sum_{i=1}^n\tilde{\tilde{f}}_i\le\tilde{\tilde{f}}\sim_{\Theta}
    \tilde{f}\).
  \item \(f\approx \tilde{f}\) if \(f\precsim \tilde{f}\) and \(\tilde{f}\precsim f\).
  \end{enumerate}
\end{definition}

\begin{proposition}
  \label{prop:pullbacks_of_relations_on_Exel_Pardo}
  Let~\(\B\) be the inverse semigroup of compact open bisections in
  the groupoid \(\Gr=S_{\Gamma,E}\ltimes E^\infty\).  For any
  \(f,\tilde{f}\in \Contc(E^0,\N)\), \(f\precsim \tilde{f}\) if and
  only if \(\tau(f)\precsim_\B\tau(\tilde{f})\).  Similarly,
  \(f\sim \tilde{f}\) if and only if
  \(\tau(f)\sim_\Gr\tau(\tilde{f})\) as
  in~\eqref{eq:Rainone_Sims_equiv}.
\end{proposition}

\begin{proof}
  We only show the first equivalence, as the proof readily
  simplifies to give the second one.  Assume that
  \(f\precsim \tilde{f}\) and let
  \(f_1,\dotsc,f_n,\tilde{f}_1,\dotsc,\tilde{f}_n,\tilde{\tilde{f}}_1\dotsc,
  \tilde{\tilde{f}}_n,\tilde{\tilde{f}}\in\Contc(E^0,\N)\) be
  elements witnessing that.  By \ref{enu:auxillaty_graph_preorder1}
  and~\ref{enu:auxillaty_graph_preorder1.5} of
  Lemma~\ref{lem:auxillaty_graph_preorder} and because~\(\tau\) is
  monotone, we get
  \begin{equation*}
    \tau(f)
    \sim_\Gr\tau \biggl(\sum_{i=1}^n\tilde{f}_i \biggr)
    = \sum_{i=1}^n\tau \bigl(\tilde{f}_i\bigr)
    \sim_\Gr \sum_{i=1}^n\tau\biggl(\tilde{\tilde{f}}_i\biggr)
    = \tau\biggl(\sum_{i=1}^n \tilde{\tilde{f}}_i\biggr)
    \le \tau\biggl(\tilde{\tilde{f}}\biggr)
    \sim_\Gr \tau\bigl(\tilde{f}\bigr).
  \end{equation*}
  Since both \(\sim_\Gr\) and~\(\le\) are stronger
  than~\(\prec_\B\), this implies that
  \(\tau(f) \precsim_\B \tau(\tilde{f})\).

  Conversely, assume that \(\tau(f)\precsim_\B\tau(\tilde{f})\).  By
  Definition~\ref{def:preorder_relation}, this means that there are
  \(b=\sum_i 1_{Z(\alpha_i,g_i,\beta_i)}\in \F(\B)=\Contc(\Gr,\N)\)
  (here \(\alpha_i,\beta_i\in E^*\), \(g_i\in \Gamma\) and
  \(\s(\alpha_i)=g_is(\beta_i)\) for each~\(i\)) such that
  \(\tau(f)\le \sum_i 1_{Z(\beta_i)}\) and
  \(\sum_i 1_{Z(\alpha_i)}\le \tau(\tilde{f})\).  By multiplying
  each \(Z(\alpha_i,g_i,\beta_i)\) with the support of~\(\tau(f)\),
  which is compact open, we arrange that
  \[
    \tau(f)=\sum_{i=1}^n 1_{Z(\beta_i)},
    \qquad \text{ and }\qquad
    \sum_{i=1}^n 1_{Z(\alpha_i)}\le \tau(\tilde{f}).
  \]
  Moreover, we may assume that all~\(\beta_i\) have the same length,
  say \(p\in \N\).  Indeed, putting
  \(p\defeq \max_{i=1,\dotsc,n} \abs{\beta_i}\) and enumerating paths
  in each finite set \(\s(\beta_i)E^{p-\abs{\beta_i}}\) as
  \(\{\gamma_{i,j}\}_j\) we get that \(\{g_i\gamma_{i,j}\}_j\)
  enumerate elements in \(\s(\alpha_i)E^{p-\abs{\beta_i}}\), and
  therefore \(1_{Z(\beta_i)}=\sum_j1_{\beta_i\gamma_{i,j}}\) and
  \(1_{Z(\alpha_i)}=\sum_j1_{\alpha_i g_i\gamma_{i,j}}\).  Thus it
  suffices to replace each~\(\beta_i\) by~\(\beta_i\gamma_{i,j}\),
  each~\(\alpha_i\) by~\(\beta_i g_i\gamma_{i,j}\), and~$g_i$
  by~$g_i|_{\gamma_{i,j}}$.  Then Lemma
  \ref{lem:auxillaty_graph_preorder}.\ref{enu:auxillaty_graph_preorder2}
  gives
  \[
    \Theta^p(f)=\sum_{i=1}^n 1_{\{\s(\beta_i)\}}.
  \]
  As \(g_is(\beta_i)=\s(\alpha_i)\), we get
  \(1_{\{\s(\beta_i)\}}\sim_{\Gamma} 1_{\{\s(\alpha_i)\}}\) for each~\(i\).
  Since \(\sum_{i=1}^n 1_{Z(\alpha_i)}\le \tau(\tilde{f})\), we may
  write
  \(\tau(\tilde{f})=\sum_{i=1}^n 1_{Z(\alpha_i)} +\sum_{i=n+1}^{n'}
  1_{Z(\alpha_i)}\) for some \(\alpha_i\in E^*\), for
  \(i=n+1,\dotsc, n'\).  Putting
  \(q\defeq \max_{i=1,\dotsc, n'} \abs{\alpha_i}\) and enumerating
  paths in each finite set \(\s(\alpha_i)E^{p-\abs{\alpha_i}}\) as
  \(\{\gamma_{i,j}\}_{j=1}^{n_i}\) we get
  \(1_{Z(\alpha_i)}=\sum_j 1_{Z(\alpha_i\gamma_{i,j})}\).  Then
  \(\tau(\tilde{f})=\sum_{i=1}^{n'}\sum_j
  1_{Z(\alpha_i\gamma_{i,j})}\), where each \(\alpha_i\gamma_{i,j}\)
  is in~\(E^q\).  Hence by
  Lemma~\ref{lem:auxillaty_graph_preorder}.\ref{enu:auxillaty_graph_preorder2}
  \[
    \Theta^q(\tilde{f})
    =  \sum_{i=1}^{n'}\sum_{j=1}^{n_i} 1_{\{\s(\gamma_{i,j})\}}
  \]
  For each~\(i\), we have
  \(\tau(1_{\{\s(\alpha_i)\}})= 1_{Z(\s(\alpha_i))}=\sum_{j=1}^{n_i}
  1_{Z(\gamma_{i,j})}\) where each~\(\gamma_{i,j}\) has length
  \(q-\abs{\alpha_i}\).  Hence
  Lemma~\ref{lem:auxillaty_graph_preorder}.\ref{enu:auxillaty_graph_preorder2}
  gives
  \[
    \Theta^{q-\abs{\alpha_i}}(1_{\{\s(\alpha_i)\}})
    =  \sum_{j=1}^{n_i} 1_{\{\s(\gamma_{i,j})\}}
  \]
  for each~\(i\).
  Thus we conclude that
  \(\sum_{i=1}^n\Theta^{q-\abs{\alpha_i}}(1_{\{\s(\alpha_i)\}}) =
  \sum_{i=1}^n\sum_{j=1}^{n_i} 1_{\{\s(\gamma_{i,j})\}}\le
  \Theta^q(\tilde{f})\).  Hence putting
  \(f_i\defeq 1_{\{\s(\beta_i)\}}\),
  \(\tilde{f}_i\defeq 1_{\{\s(\alpha_i)\}}\),
  \(\tilde{\tilde{f}}_i\defeq
  \Theta^{q-\abs{\alpha_i}}(1_{\{\s(\alpha_i)\}})\), and
  \(\tilde{\tilde{f}}\defeq\Theta^q(\tilde{f})\), we get elements
  that witness that \(f\precsim \tilde{f}\).
\end{proof}

\begin{corollary}
  \label{cor:type_semigroups_for_Exel_Pardo}
  The relation~\(\precsim\) is a preorder, while \(\approx\)
  and~\(\sim\) are equivalence relations on \(\Contc(E^0,\N)\).  The
  quotients
  \[
    W(\Gamma,E)\defeq \Contc(E^0,\N)/{\approx},
    \qquad
    V(\Gamma,E)\defeq \Contc(E^0,\N)/{\sim}
  \]
  are preordered monoids with the addition induced from
  \(\Contc(E^0,\N)\), partial order on \(W(\Gamma,E)\) induced
  by~\(\precsim\), and the algebraic preorder on \(V(\Gamma,E)\).  The
  map~\(\tau\) descends to isomorphisms
  \[
    W(\Gamma,E)\cong S_\B(\Gr),\qquad V(\Gamma,E)\cong S(\Gr)
  \]
  of ordered monoids. In particular, the semigroup
  $V(\Gamma,E)\cong S(\Gr)$ is the universal abelian monoid given by
  generators \(\{a_{v}: v\in E^0\}\) subject to relations
  \[
    a_v=\sum_{e\in \rg^{-1}(v)} a_{\s(e)}\quad \text{ and }\quad   a_v = a_{gv}
  \]
  for \(v\in E^0\) and \(g\in \Gamma\).  In addition,
  \(W(\Gamma,E)\cong S_\B(\Gr)\) is the quotient
  $\widetilde{V(\Gamma,E)}$ of $V(\Gamma,E)$ by the congruence
  generated by the algebraic preorder.
\end{corollary}

\begin{proof}
  The map~\(\tau\) descends to surjective maps onto \(S_\B(\Gr)\)
  and~\(S(\Gr)\) because for any \(\alpha\in E^*\), the bisection
  \(Z(\alpha, e, \s(\alpha))\) witnesses the equivalence
  \(1_{Z(\alpha)}\sim_{\Gr} 1_{Z(\s(\alpha))}\), and the latter span
  both variants of the type semigroup.  The injectivity claim follows
  from Proposition~\ref{prop:pullbacks_of_relations_on_Exel_Pardo} and
  the constructions of the ordered monoids \(S_\B(\Gr)\)
  and~\(S(\Gr)\).  The universal description of \(V(\Gamma,E)\)
  translates from the fact that \(\Contc(E^0,\N)\) is the universal
  abelian monoid generated by \(\{1_{\{v\}}: v\in E^0\}\) and
  that~\(\sim\) is, by construction, the smallest congruence on
  \(\Contc(E^0,\N)\) such that $1_{\{v\}}\sim \Theta(1_{\{v\}})$ and
  \(1_{\{v\}}\sim 1_{\{gv\}}\) for \(v\in E^0,\, g\in \Gamma\).  We
  have \(W(\Gamma,E)=\widetilde{V(\Gamma,E)}\) by Proposition
  \ref{prop:type_vs_groupoid_semigroup}.\ref{item:type_vs_groupoid_semigroup1}.
\end{proof}

\begin{remark}
  \label{rem:graph_type_semigroup}
  When the group \(\Gamma=\{e\}\) is trivial, the pair \((\Gamma,E)\)
  reduces to the graph~\(E\).  Then it is reasonable to denote the
  above ordered monoids as \(W(E)\) and~\(V(E)\), respectively.
  Then~\(V(E)\) is exactly the monoid considered
  in~\cite{Rainone-Sims:Dichotomy}*{Definition 8.3}, and according to
  the universal description in the last part of
  Corollary~\ref{cor:type_semigroups_for_Exel_Pardo}, \(V(E)\) is also
  the monoid~\(M_E\) considered
  in~\cite{Ara-Moreno-Pardo:Nonstable_K_for_Graph_Algs}.
  By~\cite{Ara-Moreno-Pardo:Nonstable_K_for_Graph_Algs}, the monoid
  \(V(E)=M_E\) is unperforated and is isomorphic to the monoid given
  by isomorphism classes of finitely generated projective modules over
  the corresponding Leavitt path algebra.  The semigroup~\(W(E)\) is
  often much smaller than \(V(E)\).
\end{remark}

\begin{example}
  \label{ex:type_semigroup_for_Cuntz_algebras}
  Assume \(E^0=\{v\}\) is a singleton.  Then
  \(E^1=\{e_1,\dotsc e_n\}\) is a finite set and self-similar actions
  of a group~\(\Gamma\) on~\(E\) correspond to classical
  self-similar actions of~\(\Gamma\) on the set~\(E^1\).  As noted above,
  the type semigroups do not depend on~\(\Gamma\) in this case.
  Using that \(\Theta(1_{\{v\}})=n 1_{\{v\}}\) one sees that
  \begin{align*}
    W(E)&=\{0, [1_{\{v\}}]_{\approx}\}\cong  \{0,\infty\},\\
    V(E)&=\{0\} \sqcup \setgiven{k [1_{\{v\}}]_{\sim}} {k=1,\dotsc,n-1}
          \cong \{0\}\sqcup \Z_{n-1}
  \end{align*}
  where \(\setgiven{k [1_{\{v\}}]_{\sim}} {k=1,\dotsc,n-1}\) is a cyclic
  group with generator \([1_{\{v\}}]_{\sim}\) and neutral element
  \((n-1)[1_{\{v\}}]_{\sim}\).  Both of these type semigroups are minimal
  and purely infinite.
\end{example}

We now show that the type semigroup of a self-similar action
\((E,\Gamma)\) is isomorphic to the type semigroup of the
graph~\(E_{\Gamma}\) constructed by Larki in
\cite{Larki:dichotomy_for_self-similar_graphs}*{Section~4}.  We let
\(E_{\Gamma}^0\defeq E^0/\Gamma=\setgiven{[v]}{v\in E^0}\) be the
quotient orbit space for the equivalence relation \(w\sim_\Gamma v\)
and we pick a transversal~$\Omega$ (a system of distinct
representatives) for their equivalence classes.  We put
\(E_{\Gamma}^1\defeq \bigcup_{v\in \Omega} \rg^{-1}(v)\subseteq E^1\)
and \(\rg_{\Gamma}(e)\defeq[r(e)]\), \(\s_{\Gamma}(e)\defeq[\s(e)]\)
for \(e\in E_{\Gamma}^1\).  Then
\(E_{\Gamma}\defeq (E_{\Gamma}^0,E_{\Gamma}^1,\rg_{\Gamma},
\s_{\Gamma})\) is a row-finite directed graph that up to an
isomorphism does not depend on the choice of the transversal $\Omega$,
see~\cite{Larki:dichotomy_for_self-similar_graphs}.  The quotient
graph~\(E_\Gamma\) coincides with~\(E\) whenever~\(\Gamma\) does not
move the vertices.

\begin{proposition}
  \label{prop:reducing_to_graphs}
  We have natural isomorphisms of preordered monoids
  \[
    W(\Gamma,E)\cong W(E_\Gamma), \qquad V(\Gamma,E)\cong V(E_\Gamma).
  \]
  In particular, these monoids are unperforated.
\end{proposition}

\begin{proof}
  Let~\(\Phi\) be the map
  \(\Contc(E^0,\N)\ni \sum_{v\in V} 1_{\{v\}}\mapsto \sum_{v\in V}
  1_{[v]}\in \Contc(E^0/\Gamma,\N)\). Clearly, it is additive and
  \(\Phi(1_{\{v\}})=\Phi(1_{\{gv\}})\) for \(g\in \Gamma, v\in E^0\).
  We claim that \(\Phi(\Theta(1_{\{v\}}))=\Theta(\Phi(1_{\{v\}}))\)
  for \(v\in E^0\).  Indeed, since~$g$ acts by automorphisms on~$E$,
  we have that $r^{-1} (g v) = \{ g e: e\in r^{-1}(v)\}$ and thus
  \[
    \Theta(1_{\{v\}})
    = \sum_{e\in \rg^{-1}(v)} 1_{\s(e)}
    = \sum_{f\in r^{-1}(gv)} 1_{s(f)}=\Theta(1_{\{gv\}}).
  \]
  It follows that we may assume that \(v\in \Omega\) is in the
  transversal chosen to construct~\(E_{\Gamma}\).  Then
  \(\Phi(\Theta(1_{\{v\}}))=\sum_{e\in \rg^{-1}(v)}
  1_{[\s(e)]}=\sum_{e\in \rg_{\Gamma}^{-1}(v)}
  1_{\s_{\Gamma}(e)}=\Theta(\Phi(1_{\{v\}}))\) by construction.  This
  shows the claim.  Now it follows that~$\Phi$ descends to a
  homomorphism \(V(\Gamma,E)\to V(E_\Gamma)\).  To see that it is
  injective, consider the well-defined additive map
  \[
    \Psi\colon \Contc(E^0/\Gamma,\N)\to V(\Gamma,E), \qquad
    \sum_{v\in V} 1_{[v]}\mapsto  [\sum_{v\in V} 1_{\{v\}}]
  \]
  The claim above implies that
  \(\Psi(\Theta(1_{[v]}))=\Psi(1_{[v]})\).  Thus~\(\Psi\) descends to
  a homomorphism \(V(E_\Gamma)\to V(\Gamma,E)\) that is inverse to the
  previous one.  Accordingly, \(V(\Gamma,E)\cong V(E_\Gamma)\).  This
  implies that
  \(W(\Gamma,E)=\widetilde{V(\Gamma,E)}\cong
  \widetilde{V(E_\Gamma)}\cong W(E_\Gamma)\).  The monoids are
  unperforated by Remark~\ref{rem:graph_type_semigroup} and
  Proposition
  \ref{prop:type_vs_groupoid_semigroup}.\ref{item:type_vs_groupoid_semigroup1.33}.
\end{proof}

\begin{definition}[see \cite{Exel-Pardo:Self-similar}*{Definitions
    13.4 and~14.1}]
  The \emph{base points} of a path
  \(\alpha=\alpha_1\alpha_2\cdots \in E^* \cup E^{\infty}\) are the
  vertices \(\rg(\alpha_1)\) and \(\s(\alpha_i)\) for all~\(i\).  A
  path has an \emph{entrance} if at least one of its base points
  receives at least two different edges.  A
  \emph{\(\Gamma\)\nb-path} from \(v\in E^0\) to \(w\in E^0\) is a
  pair \((\alpha,g)\in E^*\times \Gamma\) where \(v=\rg(\alpha)\)
  and \(\s(\alpha)=gw\).  If, in addition, \(v=w\) and~\(\alpha\) is
  nontrivial, we call~\(\alpha\) a
  \emph{\(\Gamma\)\nb-cycle}.  We call~\(E\)
  \(\Gamma\)\nb-\emph{cofinal} (or \emph{weakly
    \(\Gamma\)\nb-transitive}) if for every \(\mu\in E^\infty\) and
  every \(v\in E^0\) there is a \(\Gamma\)\nb-path from \(v\) to one of the
  base points of~\(\mu\).
\end{definition}

\begin{proposition}
  \label{prop:unperforation_for_minimal_exel_pardo}
  The following conditions are equivalent:
  \begin{enumerate}
  \item\label{enu:unperforation_for_graphs1}%
    the groupoid~\(\Gr\) associated to \((\Gamma,E)\) is minimal;
  \item\label{enu:unperforation_for_graphs2}%
    the groupoid~\(\Gr_{E_\Gamma}\) associated to~\(E\) is minimal;
  \item\label{enu:unperforation_for_graphs3}%
    the graph~\(E\) is \(\Gamma\)\nb-cofinal;
  \item\label{enu:unperforation_for_graphs4}%
    the graph~\(E_\Gamma\) is cofinal;
  \item\label{enu:unperforation_for_graphs5}%
    the type semigroup
    \(S_\B(\Gr)\cong W(\Gamma,E)\cong W(E_\Gamma)\) is simple;
  \item\label{enu:unperforation_for_graphs6}%
    the type semigroup \(S(\Gr) \cong V(\Gamma,E)\cong V(E_\Gamma)\)
    is simple.
  \end{enumerate}
  If these equivalent conditions hold, then the type semigroups in
  \ref{enu:unperforation_for_graphs5}
  and~\ref{enu:unperforation_for_graphs6}
  are purely infinite if~\(E_\Gamma\) has a cycle
  with an entrance, and they have a finite faithful state when ~\(E_\Gamma\) has no cycles.
\end{proposition}

\begin{proof}
  The
  equivalence \ref{enu:unperforation_for_graphs3}\(\Leftrightarrow\)\ref{enu:unperforation_for_graphs4}
  is  \cite{Larki:dichotomy_for_self-similar_graphs}*{Lemma
    4.7.(2)}. The equivalence
  \ref{enu:unperforation_for_graphs1}\(\Leftrightarrow\)\ref{enu:unperforation_for_graphs3}
  holds by the extension of
  \cite{Exel-Pardo:Self-similar}*{Theorem~13.6} in
  \cite{Exel-Pardo-Starling:Self-similar}*{Theorem~4.3}.  This also
  gives, as a special case, the well-known equivalence
  \ref{enu:unperforation_for_graphs2}\(\Leftrightarrow\)\ref{enu:unperforation_for_graphs4}.
  We have
  \ref{enu:unperforation_for_graphs1}\(\Leftrightarrow\)\ref{enu:unperforation_for_graphs5}
  by Corollary~\ref{cor:minimal_implies_simple} and
  \ref{enu:unperforation_for_graphs5}\(\Leftrightarrow\)\ref{enu:unperforation_for_graphs6}
  by Proposition~\ref{prop:type_vs_groupoid_semigroup}.  This proves
  that
  \ref{enu:unperforation_for_graphs1}--\ref{enu:unperforation_for_graphs6}
  are equivalent.  For the second part of the assertion, it suffices
  to consider~\(W(E)\) for a
  cofinal graph~\(E\).   As then~\(V(E)\) has the same properties by
  Proposition~\ref{prop:type_vs_groupoid_semigroup},
  and we may apply all this to the graph \(E_{\Gamma}\) in
  \ref{enu:unperforation_for_graphs4} to get that the semigroups in
  \ref{enu:unperforation_for_graphs5}
  and~\ref{enu:unperforation_for_graphs6} have the corresponding properties.  Let
  then~\(E\) be a
  cofinal directed graph (row-finite without sources).

  If~\(E\) has no cycles, then~\(\Cst(E)\) is
  a simple AF-algebra (see
  \cite{Kumjian-Pask-Raeburn:Cuntz-Krieger_graphs}*{Theorem~2.4} and
  \cite{Bates-Pask-Raeburn-Szymanski:Row_finite}*{Proposition~5.1}).
  Hence it is stably finite, and~\(W(E)\) has a finite faithful state by
  Theorem~\ref{thm:stably_finite_twisted}.
  Now suppose
  that~\(E\) contains a cycle~\(\alpha\) with an entrance.  If
  \(w=\rg(\alpha)=\s(\alpha)\), then
  \[
    \Theta^{\abs{\alpha}}(1_{\{w\}})
    = \sum_{v\in E^0} \abs{wE^{\abs{\alpha}}v}1_{\{v\}}
    = 1_{\{w\}} +f_0 \quad\text{where } 0\neq f_0\in \Contc(E^0,\N).
  \]
  Hence~\([1_{\{w\}}]\) is infinite in~\(W(E)\).  Since every base point
  of~\(\alpha\) is a range of a cycle with an entrance, the same
  holds for any base point of~\(\alpha\).  Now let \(w\in E\) be
  arbitrary.  Applying that~\(E\) is cofinal to the infinite
  concatenation of the path~\(\alpha\), we see that there is
  a path~\(\beta\) that ends in~\(w\) and starts at some base
  point~\(v\) of~\(\alpha\).  Then
  \(\Theta^{\abs{\beta}}(1_{\{w\}})=1_{\{v\}} +f_0\) for some
  \(f_0\in \Contc(E^0,\N)\).  Hence \([1_{\{w\}}]=[1_{\{v\}}]+[f_0]\) is
  infinite in~\(W(E)\) as it has an infinite summand.  This implies
  that all nonzero elements in~\(W(E)\) are infinite.
  Since~\(W(E)\) is simple, this means that~\(W(E)\) is properly
  infinite (see
  Lemma~\ref{lem:residually_infinite}).
\end{proof}

\begin{definition}
  A \emph{graph trace} (see
  \cite{Hjelmborg:Purely_infinite_graphs}*{Definition~2.7}), is a
  map $T\colon E^0\to[0,\infty]$ such that
  \[
    T(v)=\sum_{\rg(e)=v} T(\s(e))
    \qquad \text{for every }v\in E^0.
  \]
  If, in addition, \(T\) is \(\Gamma\)\nb-invariant in the sense
  that \(T(v)=T(gv)\) for all \(g\in \Gamma\), \(v\in E^0\), we
  call~\(T\) a \emph{graph \(\Gamma\)\nb-trace} (see
  \cite{Larki:dichotomy_for_self-similar_graphs}*{Definition~4.11}).
  We call~\(T\) \emph{nontrivial} if \(0< T(v)<\infty\) for
  some \(v\in E^0\), \emph{faithful} if \(0< T(v)\) for all
  \(v\in E^0\), and \emph{normalised} if \(\sum_{v\in E^0} T(v)=1\).
\end{definition}

\begin{lemma}
  \label{lem:graph_traces_vs_states}
  There is a bijection between states~\(\nu\)
  on~\(W(\Gamma,E)\) and graph \(\Gamma\)\nb-traces~\(T\).  It is
  given by \(\nu([1_{\{v\}}])=T(v)\) for \(v\in E^0\).
\end{lemma}

\begin{proof}
  For any state~\(\nu\) on~\(W(\Gamma,E)\)
  the formula \(T(v)\defeq \nu([1_{\{v\}}])\) for \(v\in E^0\)
  defines a graph \(\Gamma\)-trace because
  \(T(v)
  =\nu([1_{\{v\}}])=\nu([\Theta(1_{\{v\}})])=\nu([\sum_{\rg(e)=v}
  1_{\{\s(e)\}}])=\sum_{\rg(e)=v}\nu([
  1_{\{\s(e)\}}])=\sum_{\rg(e)=v}T(\s(e))\) and \(T(gv)=\nu([1_{\{gv\}}])=\nu([1_{\{v\}}])=T(v)\).
  Conversely, for any graph trace~\(T\), we define an additive map
  \(\nu\colon \Contc(E^0,\N) \to [0,\infty]\) by
  \(\nu(f)\defeq \sum_{v\in E^0} f(v) T(v)\).  Then
  \begin{align*}
    \nu(\Theta(f))
    &=\sum_{v\in E^0} \Theta(f)(v) T(v)
      = \sum_{v\in E^0}\sum_{\s(e)=v} f(\rg(e)) T(v)
      = \sum_{e\in E^1} f(\rg(e)) T(\s(e))
    \\
    &=\sum_{v\in E^0}\sum_{\rg(e)=v} f(v) T(\s(e))
      = \sum_{v\in E^0} f(v) T(v)=\nu(f)
  \end{align*}
  and \( \nu(1_{\{gv\}})=T(gv)=T(v)= \nu(1_{\{v\}})\).
  Hence~\(\nu\) descends to a state on~\(W(E)\).  This shows the
  asserted bijection.
\end{proof}

We now apply these results to Exel--Pardo algebras (see
\cite{Exel-Pardo:Self-similar}*{Definition~3.2},
\cite{Exel-Pardo-Starling:Self-similar}*{Definition~1.2}).

\begin{definition}
  The \emph{Exel--Pardo \(\Cst\)\nb-algebra} \(\OO_{(\Gamma,E)}\) is
  the universal \(\Cst\)\nb-algebra with generators
  \(\setgiven{p_x}{x\in E^0}\cup\setgiven{\s_e}{e\in E^1}\cup
  \setgiven{u_g}{g\in \Gamma}\) where
  \(\setgiven{p_x}{x\in E^0}\cup\setgiven{\s_e}{e\in E^1}\) is a
  Cuntz--Krieger \(E\)\nb-family, \(\setgiven{u_g}{g\in \Gamma}\) is a
  unitary representation of~\(\Gr\) and
  \[
    u_g \s_e = \s_{ge} u_{g|e}, \qquad
    u_gp_xu_g^* = p_{gx}  \qquad
    \text{for all } g\in \Gamma, e\in E^1, x\in E^0.
  \]
\end{definition}

As shown in \cite{Exel-Pardo:Self-similar}*{Proposition~11.1} the
elements \(\setgiven{p_x}{x\in E^0}\cup\setgiven{\s_e}{e\in E^1}\)
generate a copy of the graph \(\Cst\)\nb-algebra \(\Cst(E)\) inside
\(\OO_{(\Gamma,E)}\).  However, usually
\(\setgiven{p_x}{x\in E^{(0)}}\cup \setgiven{u_g}{g\in \Gamma}\)
generates only a homomorphic image of the crossed product
\(\Cont(E^0)\rtimes \Gamma\).  Even more,
\cite{Exel-Pardo:Self-similar}*{Corollary~6.4} implies a natural
isomorphism
\[
  \OO_{(\Gamma,E)}\cong \Cst(\Gr),
  \qquad\text{where }
  \Gr=S_{\Gamma,E}\ltimes  E^\infty.
\]
The canonical diagonal subalgebra
\(\mathcal{D}_E\cong \Cont_0(E^\infty)\) of \(\Cst(E)\) sits
in~\(\OO_{(\Gamma,E)}\) as the algebra of functions on the unit
space of its canonical groupoids model.  In particular, we have a
canonical generalised expectation
\(\E\colon \OO_{(\Gamma,E)}\to \Borel(E^\infty)\).

\begin{proposition}
  \label{prop:tracial_states_on_Exel-Pardo}
  Let~\(\Gamma\) be a group acting self-similarly on a row-finite
  graph \(E=(E^0,E^1,\rg,\s)\).  The equalities
  \(\tau(p_v)=T(v)= [T]([v])=[\tau](p_{[v]})\) for \(v\in E^0\)
  establish bijections between the following objects:
  \begin{enumerate}
  \item \label{it:EP_traces1}%
    lower semicontinuous traces~\(\tau\) on~\(\OO_{(\Gamma,E)}\)
    that factor through~\(\E\) \textup{(}see
    Proposition~\textup{\ref{prop:states_and_traces})};
  \item \label{it:EP_traces2}%
    graph \(\Gamma\)\nb-traces \(T\colon E^0 \to [0,\infty]\);
  \item \label{it:EP_traces3}%
    graph traces \([T]\colon E^0/\Gamma \to [0,\infty]\) for the
    quotient graph~\(E_{\Gamma}\);
  \item \label{it:EP_traces4}%
    lower semicontinuous traces~\([\tau]\) on~\(\Cst(E_\Gamma)\)
    that factor through the canonical expectation from
    \(\Cst(E_\Gamma)\) onto its diagonal subalgebra
    \(\mathcal{D}_{E_\Gamma}\cong \Cont_0(E_\Gamma^\infty)\).
  \end{enumerate}
  If all orbits \(\Gamma v\) for \(v\in E^0\) are finite, then also the
  equalities
  \[
    \tau(p_v)
    = T(v)
    = \abs{\Gamma v}^{-1}[T]([v])
    = \abs{\Gamma v}^{-1}[\tau](p_{[v]})
  \]
  for \(v\in E^0\) establish bijections between the above objects.  In
  addition, these bijections restrict to bijections between tracial
  states and normalised graph traces.
\end{proposition}

\begin{proof}
  Using the isomorphism \(W(\Gamma,E)\cong S_\B(\Gr)\) from
  Corollary~\ref{cor:type_semigroups_for_Exel_Pardo} and applying
  Theorem~\ref{thm:Riesz_for_monoids} to~\(S_\B(\Gr)\) gives a
  bijection between traces~\(\tau\) in~\ref{it:EP_traces1} and
  states~\(\nu\) on~\(W(\Gamma,E)\).  Moreover, \(\tau\) is a state if
  and only if
  \(\sup {}\setgiven{\nu([1_U])}{U\subseteq E^0 \text{ is
      finite}}=1\), which is equivalent to
  \(\sum_{v\in E^0} \nu([1_{\{v\}}])=1\).  Thus combining this with
  the bijection from Lemma~\ref{lem:graph_traces_vs_states}, we get
  that \(\tau(p_v)=T(v)\) establishes a bijection between the objects
  in \ref{it:EP_traces1} and~\ref{it:EP_traces2}, which restricts to a
  bijection between tracial states and normalised graph
  \(\Gamma\)\nb-traces.  This also proves the bijection between the
  objects in \ref{it:EP_traces3} and~\ref{it:EP_traces4}, as it is a
  special case of the previous bijection, where the group is trivial.
  The relation \(T(v)= [T]([v])\) establishes a bijection between the
  objects in \ref{it:EP_traces2} and~\ref{it:EP_traces3}, as they are
  in the corresponding bijections with states on
  \(W(\Gamma,E) \cong W(E_\Gamma)\), see Proposition
  \ref{prop:reducing_to_graphs} and
  Lemma~\ref{lem:graph_traces_vs_states}. When all orbits \(\Gamma v\)
  for \(v\in E^0\) are finite, the relation
  \(T(v)= \abs{\Gamma v}^{-1}[T]([v])\) also establishes a bijection
  between the objects in \ref{it:EP_traces2} and~\ref{it:EP_traces3}
  but this time it restricts to a bijection between the sets of
  normalised \(\Gamma\)\nb-traces on~\(E^0\) and normalised graph
  traces on~\(E^0/\Gamma\).
\end{proof}

\begin{remark}
  In general, there is no bijection between normalised graph
  \(\Gamma\)\nb-traces \(T\colon E^0 \to [0,1]\) and normalised
  graph traces \([T]\colon E^0/\Gamma \to [0,1]\), as any normalised
  \(\Gamma\)\nb-trace has to vanish on all~\(v\) with an infinite
  orbit~\(\Gamma v\).  For instance, take any graph~\(E\) with a
  normalised graph trace~$E$ (like the one in
  Example~\ref{exm:drunken_graph} below) and consider a disjoint
  countable union~\(\Z E\) of its copies.  Let \(\Gamma\defeq \Z\)
  act on~\(\Z E\) by permuting the copies of~\(E\) and let the
  restriction cocycle be trivial.  Then~\(\Z E\) has no
  normalised \(\Gamma\)\nb-traces, as all vertices have infinite
  orbits, but \((\Z E)_{\Gamma} \cong E\) has one.  The explanation is
  that the isomorphism \(W(\Gamma,E)\cong W(E_\Gamma)\) does not
  preserve the ``dimension ranges''.
\end{remark}

\begin{theorem}[Dichotomy for Exel--Pardo algebras]
  \label{thm:Dichotomy_Exel_Pardo}
  Assume that~\(\OO_{(\Gamma,E)}\) is simple.  Then the following
  conditions are equivalent:
  \begin{enumerate}
  \item\label{enu:Dichotomy_Exel_Pardo1}%
    \(\OO_{(\Gamma,E)}\) is not purely infinite;
  \item\label{enu:Dichotomy_Exel_Pardo2}%
    \(\OO_{(\Gamma,E)}\) is stably finite;
  \item\label{enu:Dichotomy_Exel_Pardo3}%
    \(\OO_{(\Gamma,E)}\) has a faithful semifinite lower
    semicontinuous trace;
  \item\label{enu:Dichotomy_Exel_Pardo4}%
    there is a nontrivial graph \(\Gamma\)\nb-trace \textup{(}which
    has to be faithful and finite\textup{)};
  \item\label{enu:Dichotomy_Exel_Pardo5}%
    there is a nontrivial graph trace for~{\(E_\Gamma\)} \textup{(}which has to
    be faithful and finite\textup{)};
  \item\label{enu:Dichotomy_Exel_Pardo6}%
    there are no \(\Gamma\)\nb-cycles in~\(E\);
  \item\label{enu:Dichotomy_Exel_Pardo7}%
    there are no cycles in {\(E_\Gamma\)};
  \item\label{enu:Dichotomy_Exel_Pardo8}%
    the graph \(\Cst\)\nb-algebra {\(\Cst(E_\Gamma)\)} is not purely
    infinite;
  \item\label{enu:Dichotomy_Exel_Pardo9}%
    the graph \(\Cst\)\nb-algebra {\(\Cst(E_\Gamma)\)} is stably
    finite;
  \item\label{enu:Dichotomy_Exel_Pardo10}%
    \(W(\Gamma,E)\cong {W(E_\Gamma)}\not\cong\{0,\infty\}\).
  \end{enumerate}
  If these conditions hold, then
  \(W(\Gamma,E)=V(\Gamma,E)\cong {V(E_\Gamma)=W(E_\Gamma)}\), and
  lower semicontinuous traces on \(\OO_{(\Gamma,E)}\) induced from
  \(\Cont_0(E ^\infty)\) are described in
  Proposition~\textup{\ref{prop:tracial_states_on_Exel-Pardo}}.
\end{theorem}

\begin{proof}
  Since~\(\OO_{(\Gamma,E)}\cong \Cst(\Gr)\) is simple, it must be the essential
  \(\Cst\)\nb-algebra of~\(\Gr\), and~\(\Gr\) must be minimal  and
  effective, see \cite{Clark-Exel-Pardo-Sims-Starling:Simplicity_non-Hausdorff}*{Theorem~4.10}.
  Effectiveness of~\(\Gr\) implies that every
  \(\Gamma\)\nb-cycle  in \(E\) has an entrance (see the proof of \cite{Exel-Pardo:Self-similar}*{Theorem~14.10}),
  which is equivalent to saying that
  every cycle in \(E_\Gamma\) has an entrance (see \cite{Larki:dichotomy_for_self-similar_graphs}*{Lemma
    4.7.(1)}). Further, \(E_\Gamma\) is cofinal by Proposition~\ref{prop:unperforation_for_minimal_exel_pardo}.
  Combining these, we get that~\(\Cst(E_\Gamma)\) is simple
  (see \cite{Bates-Pask-Raeburn-Szymanski:Row_finite}*{Proposition~5.1}) which shows that
  conditions~\ref{enu:Dichotomy_Exel_Pardo7}--\ref{enu:Dichotomy_Exel_Pardo9}
  are equivalent by \cite{Bates-Pask-Raeburn-Szymanski:Row_finite}*{Remark~5.6}.
  Condition~\ref{enu:Dichotomy_Exel_Pardo7} is equivalent to~\ref{enu:Dichotomy_Exel_Pardo6} by
  \cite{Larki:dichotomy_for_self-similar_graphs}*{Lemma~4.5.(3)}, and to~\ref{enu:Dichotomy_Exel_Pardo10} by the last part of
  Proposition~\ref{prop:unperforation_for_minimal_exel_pardo}.
  Conditions~\ref{enu:Dichotomy_Exel_Pardo10} and~\ref{enu:Dichotomy_Exel_Pardo1} are equivalent
  by Theorem~\ref{the:purely_infinite_semigroup_dichotomy} and by minimality of \(\Gr\).	
  Since the monoid \(S_\B(\Gr)\cong W(\Gamma,E)\cong { W(E_\Gamma)}\) is always unperforated
  (see Proposition~\ref{prop:reducing_to_graphs}), apply Corollary~\ref{cor:dichotomy_for_universal} to
  get that \ref{enu:Dichotomy_Exel_Pardo1} and
  \ref{enu:Dichotomy_Exel_Pardo2} are equivalent.
  Conditions~\ref{enu:Dichotomy_Exel_Pardo2}--\ref{enu:Dichotomy_Exel_Pardo4} are equivalent by
  Lemma~\ref{lem:graph_traces_vs_states} and Theorem~\ref{thm:stably_finite_twisted}.
  Finally, Proposition~\ref{prop:tracial_states_on_Exel-Pardo} gives
  us the equivalence between \ref{enu:Dichotomy_Exel_Pardo4}
  and~\ref{enu:Dichotomy_Exel_Pardo5}.
\end{proof}

\begin{remark}
  One may define a reduced Exel--Pardo
  algebra~\(\OO_{(\Gamma,E)}^{\red}\) so that it coincides with the
  reduced groupoid \(\Cst\)-algebra, see
  \cite{Bardadyn-Kwasniewski-McKee:Banach_algebras_simple_purely_infinite}*{Definition~7.37}
  (for \(P=\{2\}\)).  Then assuming that ~\(\OO_{(\Gamma,E)}^{\red}\)
  is simple and every \(\Gamma\)\nb-cycle in~\(E\) has an entrance,
  the proof above shows that~\(\OO_{(\Gamma,E)}^{\red}\) is either
  purely infinite or it has a faithful semifinite lower semicontinuous
  trace.
\end{remark}

The above theorem generalises the result of
Larki~\cite{Larki:dichotomy_for_self-similar_graphs} where it is
assumed that the group in question is amenable and the self-similar
action is pseudo-free (a rather strong assumption that among other
things forces the groupoid~\(\Gr_{\Gamma,E}\) to be Hausdorff).  In
contrast to~\cite{Larki:dichotomy_for_self-similar_graphs}, our
proof exploits results on type semigroups.  In fact, we do not know
how to prove it more directly and we believe that there is a gap in
the proof of \((5)\Rightarrow(4)\) in
\cite{Larki:dichotomy_for_self-similar_graphs}*{Theorem~4.12}, where
it is claimed that traces in finite-dimensional algebras may be
extended (somewhat randomly) to get a trace on the inductive limit.
This is a delicate point, which we illustrate with the following
example.

\begin{example}
  \label{exm:drunken_graph}
  Consider the following infinite row-finite directed graph~\(E\):
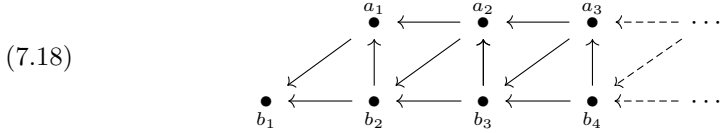
  \begin{equation}
    \label{eq:drunken_diagram}
    \begin{tikzcd}[ampersand replacement=\&]
      \, \& \overset{a_1}{\bullet}  \& \overset{a_2}{\bullet}  \&  \overset{a_3}{\bullet}  \&  \cdots \\
      \underset{b_1}{\bullet}\& \underset{b_2}{\bullet} \& \underset{b_3}{\bullet} \& \underset{b_4}{\bullet}   \& \cdots
      \arrow[from=1-2, to=2-1]
      \arrow[from=2-2, to=2-1]
      \arrow[from=1-3, to=1-2]
      \arrow[from=1-3, to=2-2]
      \arrow[from=2-2, to=1-2]
      \arrow[from=2-3, to=2-2]
      \arrow[from=2-3, to=1-3]
      \arrow[from=1-4, to=1-3]
      \arrow[from=1-4, to=2-3]
      \arrow[from=2-3, to=1-3]
      \arrow[from=2-4, to=2-3]
      \arrow[from=2-4, to=1-4]
      \arrow[from=1-5, to=1-4, dashed]
      \arrow[from=1-5, to=2-4, dashed]
      \arrow[from=2-5, to=2-4, dashed]
    \end{tikzcd}
  \end{equation}
  By Theorem~\ref{thm:Dichotomy_Exel_Pardo}, this graph admits a
  nontrivial graph trace, which has to be finite and faithful.  Thus
  there are nonzero numbers \(a_n, b_n\) attached to vertices as
  in~\eqref{eq:drunken_diagram} that determine a graph trace
  \(T\colon E^0\to [0,\infty)\).  These numbers have to satisfy the
  recurrence formulas \(a_n= a_{n+1} +b_{n+1}\) and
  \(b_n=a_n+b_{n+1}\) for \(n\ge 1\).  Denoting by
  \((F_n)_{n\in \N}\) the Fibonacci sequence \(0,1,1,2,3,5,\dotsc\),
  it follows that
  \[
    a_{n+1}=F_{2n+1}a_1-F_{2n} b_1,\quad
    b_{n+1}=F_{2n-1}b_1-F_{2n} a_1 \qquad \text{for } n\ge 1.
  \]
  Since all these numbers must to be strictly positive, we get the
  constraints
  \(\frac{F_{2n}}{F_{2n-1}}a_1<b_1< \frac{F_{2n+1}}{F_{2n}}a_1\) for
  all \(n\ge 1\).  But \(\frac{F_{n+1}}{F_n}\) converges to the
  golden ratio \(\varphi=\frac{1+\sqrt{5}}{2}\), and it oscillates
  such that
  \(\frac{F_{2n}}{F_{2n-1}}<\varphi< \frac{F_{2n+1}}{F_{2n}}\).
  Another known magic property of the Fibonacci numbers implies
  $\sum_{n=1}^\infty \abs{F_{n+1}-F_n\varphi} =\varphi$.  So
  $\sum_{n=1} a_n +b_n= a_1 \cdot \varphi$.  Hence putting
  $a_1\defeq\varphi^{-1}$, we get the unique graph trace, which
  induces the tracial state on~$\Cst(E)$.
\end{example}

\end{document}